\documentclass{article}
\usepackage[nonatbib,preprint]{neurips_2019}
\usepackage{microtype}
\usepackage{graphicx}
\usepackage{subfigure}
\usepackage{booktabs} 
\usepackage{algorithm}
\usepackage{algorithmic}
\usepackage{hyperref}

\usepackage{fncylab,enumerate,wrapfig,placeins}
\usepackage{graphicx}
\usepackage{amsmath,amsthm,amsfonts,amssymb}  
\usepackage{cuted}
\usepackage{cancel}
\usepackage{xargs}                  
\usepackage[pdftex,dvipsnames]{xcolor}
\usepackage[colorinlistoftodos,prependcaption,textsize=tiny,textwidth=2cm]{todonotes}
\newcommandx{\jianshu}[2][1=]{\todo[linecolor=red,backgroundcolor=red!25,bordercolor=red,#1]{#2}}
\newcommandx{\adithya}[2][1=]{\todo[linecolor=blue,backgroundcolor=blue!25,bordercolor=blue,#1]{#2}}

\usepackage{booktabs}
\usepackage{multirow}
\usepackage{xr}
\usepackage{float}

\newcounter{rmnum}

\def\Expect{{\sf E}}
\def\A{{\sf A}}

\def\state{{\sf X}}

\newcommand{\field}[1]{\mathbb{#1}}

\def\Re{\field{R}}
\def\transpose{{\hbox{\it\tiny T}}}

\def\argmin{\mathop{\rm arg\, min}}

\def\clF{{\cal F}}

\def\state{{\sf S}}

\def\eqdef{\mathbin{:=}}

\def\FRAC#1#2#3{\genfrac{}{}{}{#1}{#2}{#3}}

\def\ddtp{{\mathchoice{\FRAC{1}{d^{\hbox to 2pt{\rm\tiny +\hss}}}{dt}}%
{\FRAC{1}{d^{\hbox to 2pt{\rm\tiny +\hss}}}{dt}}%
{\FRAC{3}{d^{\hbox to 2pt{\rm\tiny +\hss}}}{dt}}%
{\FRAC{3}{d^{\hbox to 2pt{\rm\tiny +\hss}}}{dt}}}}

\newtheorem{theorem}{Theorem}[section]

\newtheorem{lemma}[theorem]{Lemma}
\newtheorem{assumption}[theorem]{Assumption}

\def\Lemma#1{Lemma~\ref{#1}}

\def\Theorem#1{Theorem~\ref{#1}}

\def\Section#1{Section~\ref{#1}}

\def\tiltheta{\widetilde{w}}
\def\tilw{\widetilde{w}}

\graphicspath{{SFQFigures/}}

\def\barf{{\overline {f}}}

\def\barL{{\overline {L}}}
\def\tiltheta{\widetilde \theta}
\def\bartheta{{\overline{\theta}}}

\def\dom{{\bf dom}}
\def\barw{{\overline {w}}}
\def\bartheta{{\overline {\theta}}}

\makeatletter
\def\BState{\State\hskip-\ALG@thistlm}
\makeatother

\usepackage{amsmath} 

\def\nn{\nonumber}
\def\defeq{\triangleq}
\def\prox{\mathrm{prox}}
\def\methodexact{SVRPDA-I}
\def\methodapprox{SVRPDA-II}
\def\LM{\mathrm{LM}}

\makeatletter
\newcommand*{\addFileDependency}[1]{
  \typeout{(#1)}
  \@addtofilelist{#1}
  \IfFileExists{#1}{}{\typeout{No file #1.}}
}
\makeatother

\allowdisplaybreaks

\usepackage{xr}
\allowdisplaybreaks

\title{Stochastic Variance Reduced Primal Dual Algorithms \\ for Empirical Composition Optimization}

\author{Adithya M. Devraj\thanks{Department of Electrical and Computer 
Engineering,
University of Florida, Gainesville, USA. 
Email: {\tt adithyamdevraj@ufl.edu}. The work was done during an internship at Tencent AI Lab, Bellevue, WA.} \hspace{0.1in} and \hspace{0.1in} Jianshu Chen\thanks{
Tencent AI Lab, Bellevue, WA, USA. Email: \texttt{jianshuchen@tencent.com}.}}

\begin{document}

\maketitle

\begin{abstract}

We consider a generic \emph{empirical composition optimization} problem, where there are empirical averages present both outside and inside nonlinear loss functions. Such a problem is of interest in various machine learning applications, and cannot be directly solved by standard methods such as stochastic gradient descent. We take a novel approach to solving this problem by reformulating the original minimization objective into an equivalent min-max objective, which brings out all the empirical averages that are originally inside the nonlinear loss functions. We exploit the rich structures of the reformulated problem and develop a stochastic primal-dual algorithm, \methodexact, to solve the problem efficiently. We carry out extensive theoretical analysis of the proposed algorithm, obtaining the convergence rate, the computation complexity and the storage complexity. In particular, the algorithm is shown to converge at a \emph{linear} rate when the problem is strongly convex. Moreover, we also develop an approximate version of the algorithm, named \methodapprox, which further reduces the memory requirement. Finally, we evaluate our proposed algorithms on several real-world benchmarks, and experimental results show that the proposed algorithms significantly outperform existing techniques.
\end{abstract}

\section{Introduction}
\label{sec:introduction}

In this paper, we consider the following regularized \emph{empirical composition optimization} problem:
    \begin{align}
        \min_{\theta} 
        \frac{1}{n_X} 
        \sum_{i = 0}^{n_X-1} 
        \phi_i \bigg( 
                \frac{1}{n_{Y_i}} \sum_{j = 0}^{n_{Y_i}-1} f_{\theta}(x_i, y_{ij})
        \bigg)
        +
        g(\theta),
        \label{e:general_prob_def}
    \end{align}
where $(x_i,y_{ij}) \in \Re^{m_x} \times \Re^{m_y}$ is the $(i,j)$-th data sample, $f_\theta: \Re^{m_x} \times \Re^{m_y} \to \Re^{\ell}$ is a function parameterized by $\theta \in \Re^d$, $\phi_i: \Re^\ell \to \Re^+$ is a convex \emph{merit function}, which measures a certain loss of the parametric function $f_{\theta}$, and $g(\theta)$ is a $\mu$-strongly convex regularization term.

Problems of the form \eqref{e:general_prob_def} widely appear in many machine learning applications such as reinforcement learning  \cite{ducheli,dai2018sbe,dai2017learning,macua2015distributed}, unsupervised sequence classification \cite{liucheden17uns,yehcheyuyu18uns} and risk-averse learning \cite{russha06opt,wanfanliu17sto, liawanliu2016fin,linfanwanjor18imp,wanliufan16acc} --- see our detailed discussion in Section \ref{sec:motivation}. Note that the cost function \eqref{e:general_prob_def} has an empirical average (over $x_i$) outside the (nonlinear) merit function $\phi_i(\cdot)$ and an empirical average (over $y_{ij}$) inside the merit function, which makes it different from the \emph{empirical risk minimization} problems that are common in machine learning \cite{vapnik1998statistical}. Problem \eqref{e:general_prob_def} can be understood as a generalized version of the one considered in \cite{liawanliu2016fin,linfanwanjor18imp}.\footnote{In addition to the term in \eqref{e:gen_prob_def_wang_et_all}, the cost function in \cite{linfanwanjor18imp} also has another convex regularization term.} In these prior works, $y_{ij}$ and $n_{Y_i}$ are assumed to be independent of $i$ and $f_{\theta}$ is only a function of $y_j$ so that problem \eqref{e:general_prob_def} can be reduced to the following special case:
    \begin{align}
        \min_{\theta} 
        \frac{1}{n_X} 
        \sum_{i = 0}^{n_X-1} 
        \phi_i \bigg( \frac{1}{n_{Y}} \sum_{j = 0}^{n_{Y}-1} f_{\theta}(y_{j})\bigg).
        \label{e:gen_prob_def_wang_et_all}
    \end{align}
Our more general problem formulation \eqref{e:general_prob_def} encompasses wider applications (see Section \ref{sec:motivation}). Furthermore, different from \cite{dai2017learning, wanliufan16acc, wanfanliu17sto}, we focus on the finite sample setting, where we have empirical averages (instead of expectations) in \eqref{e:general_prob_def}. As we shall see below, the finite-sum structures allows us to develop efficient stochastic gradient methods that converges at linear rate.

While problem \eqref{e:general_prob_def} is important in many machine learning applications, there are several key challenges in solving it efficiently. First, the number of samples (i.e., $n_X$ and $n_{Y_i}$) could be extremely large: they could be larger than one million or even one billion. Therefore, it is unrealistic to use batch gradient descent algorithm to solve the problem, which requires going over all the data samples at each gradient update step. Moreover, since there is an empirical average inside the nonlinear merit function $\phi_i(\cdot)$, it is not possible to directly apply the classical stochastic gradient descent (SGD) algorithm. This is because sampling from both empirical averages outside and inside $\phi_i(\cdot)$ simultaneously would make the stochastic gradients intrinsically biased (see Appendix \ref{appendix:SGD_bias} for a discussion). 

To address these challenges, in this paper, we first reformulate the original problem \eqref{e:general_prob_def} into an equivalent saddle point problem (i.e., min-max problem), which brings out all the empirical averages inside $\phi_i(\cdot)$ and exhibits useful dual decomposition and finite-sum structures (Section \ref{sec:algorithm:minmax}). To fully exploit these properties, we develop a stochastic primal-dual algorithm that alternates between a dual step of stochastic variance reduced coordinate ascent and a primal step of stochastic variance reduced gradient descent (Section \ref{sec:algorithm:batch}). In particular, we develop a novel variance reduced stochastic gradient estimator for the primal step, which achieves better variance reduction with low complexity (Section \ref{sec:algorithm:vr}). We derive the convergence rate, the finite-time complexity bound, and the storage complexity of our proposed algorithm (Section \ref{s:theo}). In particular, it is shown that the proposed algorithms converge at a \emph{linear} rate when the problem is strongly convex. Moreover, we also develop an approximate version of the algorithm that further reduces the storage complexity without much performance degradation in experiments. We evaluate the performance of our algorithms on several real-world benchmarks, where the experimental results show that they significantly outperform existing methods (Section \ref{sec:experiments}). Finally, we discuss related works in Section \ref{sec:relatedworks} and conclude our paper in Section \ref{sec:conclusion}.

\section{Motivation and Applications}
\label{sec:motivation}

To motivate our composition optimization problem \eqref{e:general_prob_def}, we discuss several important machine learning applications where cost functions of the form \eqref{e:general_prob_def} arise naturally.

\paragraph{Unsupervised sequence classification:} Developing algorithms that can learn classifiers from unlabeled data could benefit many machine learning systems, which could save a huge amount of human labeling costs. In \cite{liucheden17uns, yehcheyuyu18uns}, the authors proposed such unsupervised learning algorithms by exploiting the sequential output structures. The developed algorithms are applied to optical character recognition (OCR) problems and automatic speech recognition (ASR) problems. In these works, the learning algorithms seek to learn a sequence classifier by optimizing the empirical output distribution match (Empirical-ODM) cost, which is in the following form (written in our notation):
\begin{align}
    \min_\theta 
    \bigg\{\!-\!\sum_{i=0}^{n_X-1} 
    p_{\LM}(x_i) 
    \log \bigg( 
        \frac{1}{n_Y} \sum_{j = 0}^{n_Y - 1} f_\theta(x_i, y_j) 
    \bigg) \bigg\},
    \label{e:unsup_obj}
\end{align}
where $p_{\LM}$ is a known language model (LM) that describes the distribution of output sequence (e.g., $x_i$ represents different $n$-grams), and $f_\theta$ is a functional of the sequence classifier to be learned, with $\theta$ being its model parameter vector. The key idea is to learn the classifier so that its predicted output $n$-gram distribution is close to the prior $n$-gram distribution $p_{\LM}$ (see \cite{liucheden17uns,yehcheyuyu18uns} for more details). The cost function \eqref{e:unsup_obj} can be viewed as a special case of \eqref{e:general_prob_def} by setting $n_{Y_i}=n_Y$, $y_{ij}=y_j$ and $\phi_i(u) = -p_{LM}(x_i) \log(u)$. Note that the formulation \eqref{e:gen_prob_def_wang_et_all} cannot be directly used here, because of the dependency of the function $f_\theta$ on both $x_i$ and $y_j$.

\paragraph{Risk-averse learning:} Another application where \eqref{e:general_prob_def} arises naturally is the risk-averse learning problem, which is common in finance \cite{russha06opt,wanfanliu17sto, liawanliu2016fin,linfanwanjor18imp,wanliufan16acc, xie2018block}. Let $x_i \in \Re^d$ be a vector consisting of the rewards from $d$ assets at the $i$-th instance, where $0 \le i \le n-1$. The objective in risk-averse learning is to find the optimal weights of the $d$ assets so that the average returns are maximized while the risk is minimized. It could be formulated as the following optimization problem:
\begin{equation}
    \min_{\theta}  
    -\frac{1}{n} 
    \sum_{i = 0}^{n-1} 
    \langle x_i, \theta \rangle 
    \!+\! 
    \frac{1}{n} \sum_{i = 0}^{n-1} 
    \bigg(\langle x_i, \theta \rangle 
    \!-\! 
    \frac{1}{n} \sum_{j = 0}^{n-1} \langle x_j, \theta \rangle \bigg)^2,
\label{e:mean_var_to_prob_def}
\end{equation}
where $\theta\in \Re^d$ denotes the weight vector. The objective function in \eqref{e:mean_var_to_prob_def} seeks a tradeoff between the mean (the first term) and the variance (the second term). It can be understood as a special case of \eqref{e:gen_prob_def_wang_et_all} (which is a further special case of \eqref{e:general_prob_def}) by making the following identifications:
    \begin{align}
        &n_X \!=\! n_Y \!=\!n,\; 
        y_i \!\equiv\! x_i,\; 
        f_\theta(y_j) \!=\! [\theta^\transpose, \; - \langle y_j, \; \theta \rangle ]^{\transpose},\;
        \phi_i(u) \!=\! ( \langle x_i, u_{0:d-1} \rangle \!+\! u_d)^2 \!-\! \langle x_i, u_{0:d-1} \rangle,
        \label{e:finance_identification1}
    \end{align}
where $u_{0:d-1}$ denotes the subvector constructed from the first $d$ elements of $u$, and $u_d$ denotes the $d$-th element. An alternative yet simpler way of dealing with \eqref{e:mean_var_to_prob_def} is to treat the second term in \eqref{e:mean_var_to_prob_def} as a special case of \eqref{e:general_prob_def} by setting 
\begin{align}
    &n_X = n_{Y_i} = n, \;
    y_{ij} \equiv x_{j}, \;
    f_\theta(x_i, y_{ij}) 
        = \langle x_i - y_{ij}, \theta \rangle,\;
    \phi_i(u) = u^2, u \in \Re.
    \label{e:finance_identification2}
\end{align}
In addition, we observe that the first term in \eqref{e:mean_var_to_prob_def} is in standard empirical risk minimization form, which can be dealt with in a straightforward manner. This second formulation leads to algorithms with lower complexity due to the lower dimension of the functions: $\ell = 1$ instead of $\ell = d+1$ in the first formulation. Therefore, we will adopt this formulation in our experiment section (Section \ref{sec:experiments}).

\paragraph{Other applications:}
Cost functions of the form \eqref{e:general_prob_def} also appear in reinforcement learning \cite{ducheli,dai2017learning,dai2018sbe} and other applications \cite{wanfanliu17sto}. In Appendix~\ref{appendix:mdp}, we demonstrate its applications in policy evaluation.

\section{Algorithms}
\label{s:algo}

\subsection{Saddle point formulation}
\label{sec:algorithm:minmax}

Recall from \eqref{e:general_prob_def} that there is an empirical average inside each (nonlinear) merit function $\phi_i(\cdot)$, which prevents the direct application of stochastic gradient descent to \eqref{e:general_prob_def} due to the inherent bias (see Appendix \ref{appendix:SGD_bias} for more discussions). Nevertheless, we will show that minimizing the original cost function \eqref{e:general_prob_def} can be transformed into an equivalent saddle point problem, which brings out all the empirical averages inside $\phi_i(\cdot)$. In what follows, we will use the machinery of \emph{convex conjugate functions} \cite{rockafellar2015convex}. For a function $\psi:\Re^\ell \to \Re$, its convex conjugate function $\psi^*: \Re^\ell \to \Re$ is defined as $\psi^*(y) = \sup_{x \, \in \, \Re^\ell} (\langle x, \, y \rangle  - \psi(x))$. Under certain mild conditions on $\psi(x)$ \cite{rockafellar2015convex}, one can also express $\psi(x)$ as a functional of its conjugate function: $\psi(x) = \sup_{y \, \in \, \Re^\ell} (\langle x, \, y \rangle - \psi^*(y))$. Let $\phi_i^*(w_i)$ denote the conjugate function of $\phi_i(u)$. Then, we can express $\phi_i(u)$ as
    \begin{align}
        \phi_i(u)   
                    &=
                            \sup_{w_i \in \Re^{\ell}}(\langle u, w_i \rangle - \phi_i^*(w_i)),
        \label{e:dual_of_dual}
    \end{align}
where $w_i$ is the corresponding dual variable. Substituting \eqref{e:dual_of_dual} into the original minimization problem \eqref{e:general_prob_def}, we obtain its equivalent min-max problem as:
    \begin{align}
        \min_{\theta} \max_{w} 
        \bigg\{
        L(\theta,w) + g(\theta)
        \defeq
        \frac{1}{n_X}
        \sum_{i=0}^{n_X-1} 
        \Big[
            \Big\langle
                \frac{1}{n_{Y_i}}
                \sum_{j=0}^{n_{Y_i}-1}
                f_{\theta}(x_i,y_{ij}), w_i
            \Big\rangle
            - \phi_i^*(w_i)
        \Big]
        +
        g(\theta)
        \bigg\},
        \label{e:gen_prob_min_max}
    \end{align}
where $w \!\defeq\! \{w_0, \ldots, w_{n_X \!-\! 1}\}$, is a collection of all dual variables.
We note that the transformation of the original problem \eqref{e:general_prob_def} into \eqref{e:gen_prob_min_max} brings out all the empirical averages that are present inside  $\phi_i(\cdot)$. This new formulation allows us to develop stochastic variance reduced algorithms below.

\subsection{Stochastic variance reduced primal-dual algorithm}
\label{sec:algorithm:batch}

One common solution for the min-max problem \eqref{e:gen_prob_min_max} is to alternate between the step of minimization (with respect to the \emph{primal variable} $\theta$) and the step of maximization (with respect to the \emph{dual variable} $w$). However, such an approach generally suffers from high computation complexity because each minimization/maximization step requires a summation over many components and requires a full pass over all the data samples. The complexity of such a batch algorithm would be prohibitively high when the number of data samples (i.e., $n_X$ and $n_{Y_i}$) is large (e.g., they could be larger than one million or even one billion in applications like unsupervised speech recognition \cite{yehcheyuyu18uns}). On the other hand, problem \eqref{e:gen_prob_min_max} indeed has rich structures that we can exploit to develop more efficient solutions. 

To this end, we make the following observations. First, expression \eqref{e:gen_prob_min_max} implies that when $\theta$ is fixed, the maximization over the dual variable $w$ can be decoupled into a total of $n_X$ individual maximizations over  different $w_i$'s. Second, the objective function in each individual maximization (with respect to $w_i$) contains a finite-sum structure over $j$. Third, by \eqref{e:gen_prob_min_max}, for a fixed $w$, the minimization with respect to the primal variable $\theta$ is also performed over an objective function with a finite-sum structure. Based on these observations, we will develop an efficient stochastic variance reduced primal-dual algorithm (named \methodexact). It alternates between (i) a dual step of stochastic variance reduced coordinate ascent and (ii) a primal step of stochastic variance reduced gradient descent. The full algorithm is summarized in Algorithm \ref{alg:svrpda1}, with its key ideas explained below.

\paragraph{Dual step: stochastic variance reduced coordinate ascent.}
To exploit the decoupled dual maximization over $w$ in \eqref{e:gen_prob_min_max}, we can randomly sample an index $i$, and update $w_i$ according to:
    \begin{align}
        w_i^{(k)}
                &=
                        \arg\min_{w_i}
                        \Big\{
                            -
                            \Big\langle
                                \frac{1}{n_{Y_i}}
                                \sum_{j=0}^{n_{Y_i}-1}
                                f_{\theta^{(k-1)}}(x_i, y_{ij}), 
                                w_i
                            \Big\rangle
                            +
                            \phi_i^*(w_i)
                            +
                            \frac{1}{2\alpha_w}
                            \|w_i - w_i^{(k-1)}\|^2
                        \Big\},
        \label{e:spda_batch_dual}
    \end{align}
while keeping all other $w_j$'s ($j\neq i$) unchanged, where $\alpha_w$ denotes a step-size. Note that each step of recursion \eqref{e:spda_batch_dual} still requires a summation over $n_{Y_i}$ components. To further reduce the complexity, we approximate the sum over $j$ by a variance reduced stochastic estimator defined in \eqref{e:svrpda1:delta_w_k} (to be discussed in Section \ref{sec:algorithm:vr}). The dual step in our algorithm is summarized in \eqref{e:svrpda1:w_update}, where we assume that the function $\phi_i^*(w_i)$ is in a simple form so that the argmin could be solved in closed-form. Note that we flip the sign of the objective function to change maximization to minimization and apply coordinate descent. We will still refer to the dual step as ``coordinate ascent" (instead of descent).

\paragraph{Primal step: stochastic variance reduced gradient descent}
We now consider the minimization in \eqref{e:gen_prob_min_max} with respect to $\theta$ when $w$ is fixed. The gradient descent step for minimizing $L(\theta,w)$ is given by
    \begin{align}
        \theta^{(k)}
                    &=
                            \arg\min_{\theta}
                            \bigg\{
                                \Big\langle
                                    \sum_{i=0}^{n_X - 1}
                                    \sum_{j=0}^{n_{Y_i} - 1}
                                    \frac{1}{n_X n_{Y_i}}
                                    f_{\theta^{(k-1)}}'(x_i,y_{ij})w_i^{(k)}, 
                                    \theta
                                \Big\rangle
                                +
                                \frac{1}{2\alpha_{\theta}}
                                \|\theta - \theta^{(k-1)}\|^2
                            \bigg\},
        \label{e:spda_batch_primal}
    \end{align}
where $\alpha_{\theta}$ denotes a step-size. It is easy to see that the update equation \eqref{e:spda_batch_primal} has high complexity, it requires evaluating and averaging the gradient $f'_\theta (\cdot, \cdot)$ at every data sample. To reduce the complexity, we use a variance reduced gradient estimator, defined in \eqref{e:svrpda1:delta_theta_k}, to approximate the sums in \eqref{e:spda_batch_primal} (to be discussed in Section \ref{sec:algorithm:vr}). The primal step in our algorithm is summarized in \eqref{e:svrpda1:theta_update} in Algorithm \ref{alg:svrpda1}.

\begin{algorithm*}[t]
\caption{SVRPDA-I}\label{alg:svrpda1}
{\small
\begin{algorithmic}[1]
\STATE {\bf Inputs:} data $\{(x_i,y_{ij}): 0 \!\le\! i \!<\! n_X, 0 \!\le\! j \!<\! n_{Y_i}\}$; step-sizes $\alpha_\theta$ and $\alpha_w$; \# inner iterations $M$.
\STATE {\bf Initialization:} 
$\tilde{\theta}_{0} \in \Re^d$ and $\tilde{w}_{0} \in \Re^{\ell n_X}$.
\FOR{$s=1,2,\ldots$}
\STATE
Set $\tilde{\theta}\!=\!\tilde{\theta}_{s-1}$, $\theta^{(0)}\!=\!\tilde{\theta}$, $\tilde{w}\!=\!\tilde{w}_{s-1}$, $w^{(0)} \!=\! \tilde{w}_{s-1}$, and compute the batch quantities (for each $0 \!\le\! i \!<\! n_X$):
    \begin{align}
        U_0
                    &=
                            \!
                            \sum_{i=0}^{n_X\!-\!1}
                            \!
                            \sum_{j=0}^{n_{Y_i}\!-\!1}
                            \!\!
                            \frac{f_{\tilde{\theta}}'(x_i,y_{ij})w_i^{(0)}}{n_X n_{Y_i}}
                            ,
                            \;\;
        \overline{f}_i(\tilde{\theta})
                    \defeq
                            \!
                            \sum_{j=0}^{n_{Y_i}\!-\!1}
                            \!\!
                            \frac{f_{\tilde{\theta}}(x_i,y_{ij})}{n_{Y_i}},
                            \;\;
        \overline{f}_i'(\tilde{\theta})
                    =
                            \!
                            \sum_{j=0}^{{n_{Y_i}\!-\!1}}
                            \!\!
                            \frac{f_{\tilde{\theta}}'(x_i,y_{ij})}{n_{Y_i}}.
        \label{e:svrpda1:batch_quantities}
    \end{align}
\FOR{$k=1$ {\bfseries to} $M$}
\STATE 
Randomly sample $i_k \in \{0,\ldots, n_X\!-\!1\}$ and then $j_k \in \{0, \ldots, n_{Y_{i_k}}\!-\!1\}$ at uniform.
\STATE 
Compute the stochastic variance reduced gradient for dual update:
    \begin{align}
        \delta_k^{w}
                    &=
                            f_{\theta^{(k-1)}}(x_{i_k}, y_{i_k j_k})
                            -
                            f_{\tilde{\theta}}(x_{i_k}, y_{i_k j_k})
                            +
                            \overline{f}_{i_k}(\tilde{\theta}).
        \label{e:svrpda1:delta_w_k}
    \end{align}
\STATE Update the dual variables:
\begin{align}
    w^{(k)}_{i} 
                    &= 
                            \begin{cases}
                                \displaystyle 
                                \argmin_{w_i} 
                                \Big [ 
                                    - 
                                    \langle  
                                        \delta_k^{w}, 
                                        \, 
                                        w_i
                                    \rangle 
                                    + 
                                    \phi_i^*(w_i)
                                    +  
                                    \frac{1}{2\alpha_w} \|w_i - w^{(k-1)}_{i} \|^2
                                \Big ]
                                & \text{if } i = i_k
                                \\
                                w^{(k-1)}_{i}
                                & \text{if } i \neq i_k
                            \end{cases}.
        \label{e:svrpda1:w_update}
\end{align}
\STATE 
Update $U_k$ (primal batch gradient at $\tilde{\theta}$ and $w^{(k)}$) according to the following recursion:
\begin{align}
    U_k
            = 
                    U_{k-1}
                    + 
                    \frac{1}{n_X}
                    \overline{f}'_{i_k}(\tilde{\theta})
                    \big(
                        w^{(k)}_{i_k} - w^{(k-1)}_{i_k}   
                    \big).
    \label{e:svrpda1:Ukupdate}
\end{align}
\STATE 
Randomly sample $i_k' \in \{0,\ldots,n_X-1\}$ and then $j_k' \in \{0, \ldots, n_{Y_{i_k'}}-1\}$, independent of $i_k$ and $j_k$, and compute the stochastic variance reduced gradient for primal update:
    \begin{align}
        \delta_k^{\theta}
                &=
                        f_{\theta^{(k-1)}}'(x_{i_k'},y_{i_k' j_k'})w_{i_k'}^{(k)}
                        -
                        f_{\tilde{\theta}}'(x_{i_k'},y_{i_k' j_k'}) w_{i_k'}^{(k)}
                        +
                       U_k.
        \label{e:svrpda1:delta_theta_k}
    \end{align}
\STATE 
Update the primal variable:
    \begin{equation}
        \theta^{(k)} 
                = 
                        \displaystyle 
                        \argmin_{\theta} 
                        \Big[ 
                            \langle 
                                \delta_k^{\theta}, 
                                \, 
                                \theta
                            \rangle 
                            + 
                            g(\theta)
                            +
                            \frac{1}{2\alpha_\theta} 
                            \|\theta - \theta^{(k-1)}\|^2 
                        \Big].
        \label{e:svrpda1:theta_update}
    \end{equation}
\ENDFOR
\STATE 
{\bf{Option I:}} Set $\tilde{w}_{s} = w^{(M)}$ and $\tilde{\theta}_s = \theta^{(M)}$.
\STATE 
{\bf{Option II:}} Set $\tilde{w}_s = w^{(M)}$ and $\tilde{\theta}_s = \theta^{(t)}$ for randomly sampled $t \in \{0,\ldots,M\!-\!1\}$.
\ENDFOR
\STATE {\bf Output:} $\tilde{\theta}_s$ at the last outer-loop iteration.
\end{algorithmic}
}
\end{algorithm*}

\subsection{Low-complexity stochastic variance reduced estimators}
\label{sec:algorithm:vr}

We now proceed to explain the design of the variance reduced gradient estimators in both the dual and the primal updates. The main idea is inspired by the stochastic variance reduced gradient (SVRG) algorithm \cite{johzha13acc}. Specifically, for a vector-valued function $h(\theta)=\frac{1}{n}\sum_{i=0}^{n - 1} h_i(\theta)$, we can construct its SVRG estimator $\delta_k$ at each iteration step $k$ by using the following expression:
    \begin{align}
        \delta_k
                    &=
                            h_{i_k}(\theta) - h_{i_k}(\tilde{\theta}) + h(\tilde{\theta}),
        \label{e:algorithm:svrg_general}
    \end{align}
where $i_k$ is a randomly sampled index from $\{0,\ldots,n-1\}$, and $\tilde{\theta}$ is a reference variable that is updated \emph{periodically} (to be explained below). The first term $h_i(\theta)$ in \eqref{e:algorithm:svrg_general} is an unbiased estimator of $h(\theta)$ and is generally known as the \emph{stochastic gradient} when $h(\theta)$ is the gradient of a certain cost function. The last two terms in \eqref{e:algorithm:svrg_general} construct a control variate that has zero mean and is negatively correlated with $h_i(\theta)$, which keeps $\delta_k$ unbiased while significantly reducing its variance. The reference variable $\tilde{\theta}$ is usually set to be a delayed version of $\theta$: for example, after every $M$ updates of $\theta$, it can be reset to the most recent iterate of $\theta$. Note that there is a trade-off in the choice of $M$: a smaller $M$ further reduces the variance of $\delta_k$ since $\tilde{\theta}$ will be closer to $\theta$ and the first two terms in \eqref{e:algorithm:svrg_general} cancel more with each other; on the other hand, it will also require more frequent evaluations of the costly batch term $h(\tilde{\theta})$, which has a complexity of $O(n)$. 

Based on \eqref{e:algorithm:svrg_general}, we develop two stochastic variance reduced estimators, \eqref{e:svrpda1:delta_w_k} and \eqref{e:svrpda1:delta_theta_k}, to approximate the finite-sums in \eqref{e:spda_batch_dual} and \eqref{e:spda_batch_primal}, respectively. The dual gradient estimator $\delta_k^w$ in \eqref{e:svrpda1:delta_w_k} is constructed in a standard manner using \eqref{e:algorithm:svrg_general}, where the reference variable $\tilde{\theta}$ is a delayed version of $\theta^{(k)}$\footnote{As in \cite{johzha13acc}, we also consider Option II wherein $\tilde{\theta}$ is randomly chosen from the previous $M$ $\theta^{(k)}$'s.}.
On the other hand, the primal gradient estimator $\delta_k^{\theta}$ in \eqref{e:svrpda1:delta_theta_k} is constructed by using reference variables $(\tilde{\theta}, w^{(k)})$; that is, we uses the most recent $w^{(k)}$ as the dual reference variable, without any delay. As discussed earlier, such a choice leads to a smaller variance in the stochastic estimator $\delta_\theta^k$ at a potentially higher computation cost (from more frequent evaluation of the batch term). Nevertheless, we are able to show that, with the dual coordinate ascent structure in our algorithm, the batch term $U_k$ in \eqref{e:svrpda1:delta_theta_k}, which is the summation in \eqref{e:spda_batch_primal} evaluated at $(\tilde{\theta}, w^{(k)})$, can be computed efficiently. To see this, note that, after each dual update step in \eqref{e:svrpda1:w_update}, only one term inside this summation in \eqref{e:spda_batch_primal}, has been changed, i.e., the one associated with $i=i_k$. Therefore, we can correct $U_k$ for this term by using recursion \eqref{e:svrpda1:Ukupdate}, which only requires an extra $O(d \ell)$-complexity per step (same complexity as \eqref{e:svrpda1:delta_theta_k}).

Note that \methodexact~(Algorithm \ref{alg:svrpda1}) requires to compute and store all the $\overline{f}_i'(\tilde{\theta})$ in \eqref{e:svrpda1:batch_quantities}, which is $O(n_X d \ell)$-complexity in storage and could be expensive in some applications. To avoid the cost, we develop a variant of Algorithm \ref{alg:svrpda1}, named as \methodapprox~(see Algorithm \ref{alg:svrpda2} in the supplementary material), by approximating $\overline{f}_{i_k}(\tilde{\theta})$ in \eqref{e:svrpda1:Ukupdate} with $f_{\tilde{\theta}}'(x_{i_k},y_{i_kj_k''})$, where $j_k''$ is another randomly sampled index from $\{0,\ldots,n_{Y_i}-1\}$, independent of all other indexes. By doing this, we can significantly reduce the memory requirement from $O(n_X d \ell)$ in \methodexact~to $O(d + n_X \ell)$ in \methodapprox~(see Section \ref{s:theo:storage}). In addition, experimental results in Section \ref{sec:experiments} will show that such an approximation only cause slight performance loss compared to that of \methodexact~algorithm.

\section{Theoretical Analysis}
\label{s:theo}

\begin{table}[t]
\caption{The total complexities of different stochastic composition optimization algorithms. For C-SAGA, $\alpha=2/3$ in the minibatch setting and $\alpha=1$ when batch-size=1. In the bound for ASCVRG, the dependency on $\kappa$ has been dropped since it was not reported in \cite{linfanwanjor18imp}.}
\label{tab:comparison_complexity}
\begin{center}
\tiny
\setlength{\tabcolsep}{1pt}
\renewcommand{\arraystretch}{1.5}
\begin{tabular}{c|c|c|c|c|c}
\toprule
Methods & SVRPDA-I (Ours) & Comp-SVRG \cite{liawanliu2016fin} & C-SAGA \cite{zhaxia19com} & MSPBE-SVRG/SAGA \cite{ducheli} & ASCVRG \cite{linfanwanjor18imp}
\\
\midrule
General: problem \eqref{e:general_prob_def}  & $(n_Xn_Y \!+\! n_X \kappa)\!\ln\!\frac{1}{\epsilon}$  & \cancel{\phantom{Nothing}} & \cancel{\phantom{Nothing}} & \cancel{\phantom{Nothing}} & \cancel{\phantom{Nothing}}
\\
Special: problem \eqref{e:gen_prob_def_wang_et_all}  & $(n_X \!\!+\!\! n_Y \!\!+\!\! n_X \kappa)\! \ln\!\frac{1}{\epsilon}$  & $(n_X\!\!+\!\!n_Y\!\!+\!\!\kappa^3)\!\ln\!\frac{1}{\epsilon}$ & $(n_X\!\!+\!\!n_Y\!\!+\!\!(n_X\!\!+\!\!n_Y)^{\alpha}\kappa)\!\ln\!\frac{1}{\epsilon}$ & \cancel{\phantom{Nothing}} & $(n_X\!\!+\!\!n_Y)\!\ln\!\frac{1}{\epsilon} \!\!+\!\! \frac{1}{\epsilon^3}$
\\
Special: \eqref{e:gen_prob_def_wang_et_all} \& $n_X\!=\!1$ & $(n_Y \!+\! \kappa)\ln\frac{1}{\epsilon}$ & $(n_Y\!+\!\kappa^3)\ln\frac{1}{\epsilon}$ & $(n_Y\!+\!n_Y^{\alpha}\kappa)\ln\frac{1}{\epsilon}$ & $(n_Y\!\!+\!\!\kappa^2)\ln\frac{1}{\epsilon}$ & $n_Y \ln\frac{1}{\epsilon} \!+\! \frac{1}{\epsilon^3}$ \\
\bottomrule
\end{tabular}
\end{center}
\label{default}
\end{table}%

\subsection{Computation complexity}
\label{s:theo:convergence}
We now perform convergence analysis for the SVRPDA-I algorithm and also derive their complexities in computation and storage. To begin with, we first introduce the following assumptions.
\begin{assumption}
\label{a:main:g_phi:sc_sm}
The function $g(\theta)$ is $\mu$-strongly convex in $\theta$, and each $\phi_i$ is $1/\gamma$-smooth.
\end{assumption}
\begin{assumption}
\label{a:main:phi:lipschitz}
The merit functions $\phi_i(u)$ are Lipschitz with a uniform constant $B_w$:
        \begin{align}
            | \phi_i(u) - \phi_i(u') |
                    &\le
                            B_w \| u - u' \|, \quad \forall u, u'; \; \forall i=0,\ldots, n_X-1.
                            \nn
        \end{align}
\end{assumption}
\begin{assumption}
\label{a:main:f:smooth_bg}
$f_{\theta}(x_i, y_{ij})$ is $B_\theta$-smooth in $\theta$, and has bounded gradients with constant $B_f$:
    \begin{align}
        \| f'_{\theta_1} (x_i, y_{ij}) - f'_{\theta_2}(x_i, y_{ij}) \| \leq B_{\theta} \|\theta_1  - \theta_2\|,
        \quad
        \| f'_\theta(x_i, y_{ij}) \| \leq B_f,
        \quad
        \forall \theta, \theta_1, \theta_2,
        \; \forall i,j.
        \nn
    \end{align}
\end{assumption}
\begin{assumption}
\label{a:main:L:convex}
For each given $w$ in its domain, the function $L(\theta, w)$ defined in \eqref{e:gen_prob_min_max} is convex in $\theta$:
\begin{align}
    L(\theta_1, w) - L(\theta_2, w) \geq \langle L'_\theta(\theta_2, w), \,\, \theta_1 - \theta_2 \rangle, \quad \forall \theta_1, \theta_2.
    \nn
\end{align}
\end{assumption}
The above assumptions are commonly used in existing compositional optimization works \cite{liawanliu2016fin,linfanwanjor18imp,wanfanliu17sto,wanliufan16acc,zhaxia19com}.
Based on these assumptions, we establish the non-asymptotic error bounds for \methodexact~(using either Option I or Option II in Algorithm \ref{alg:svrpda1}). The main results are summarized in the following theorems, and their proofs can be found in Appendix~\ref{appendix:proof_svrpda1}.
    \begin{theorem}
    \label{t:final_bound_SVRPDA_I_opt1}
    Suppose Assumptions \ref{a:main:g_phi:sc_sm}--\ref{a:main:L:convex} hold. If in Algorithm \ref{alg:svrpda1} (with Option I) we choose
        \begin{align}
            \alpha_{\theta}
                &=
                        \frac{
                            1
                        }
                        {
                            n_X \mu
                            (64 \kappa + 1)
                        }, \quad
        \alpha_w
                =
                        \frac{n_X \mu}{\gamma}
                        \alpha_\theta,
                        \quad
        M
                = 
                        \big\lceil 
                            78.8n_X 
                            \kappa 
                            \!+\! 
                            1.3 n_X \!+\! 1.3 
                        \big\rceil
                        \nn
        \end{align}
    where $\lceil x \rceil$ denotes the roundup operation and $\kappa = B_f^2 / \gamma \mu + B_w^2B_\theta^2 / \mu^2$,
    then the Lyapunov function $P_s \eqdef \Expect\|\tilde{\theta}_s - \theta^*\|^2 
    + \frac{\gamma}{\mu} \cdot \frac{64\kappa + 3}{64n_X \kappa + n_X + 1} \Expect \|\tilde{w}_s - w^*\|^2$ satisfies $P_s \le (3/4)^s P_0$. Furthermore, the overall computational cost (in number of oracle calls\footnote{One \emph{oracle call} is defined as querying $f_\theta$, $f'_\theta$, or $\phi_i(u)$ for any $0 \leq i < n$ and $u \in \Re^\ell$.}) for reaching $P_s \le \epsilon$ is upper bounded by
        \begin{align}
            O\big(
                ( n_X n_Y + n_X \kappa + n_X ) \ln (1/\epsilon)
            \big).
            \label{e:thm:svrpda1_opt1:total_complexity}
        \end{align}
    where, with a slight abuse of notation, $n_Y$ is defined as $n_Y = (n_{Y_0} + \cdots + n_{Y_{n_X-1}})/n_X$.
    \end{theorem}
    \begin{theorem}
    \label{t:final_bound_SVRPDA_I_opt2}
    Suppose Assumptions \ref{a:main:g_phi:sc_sm}--\ref{a:main:L:convex} hold. If in Algorithm \ref{alg:svrpda1} (with Option II) we choose
        \begin{equation}
        \alpha_\theta 
                = 
                        \Big( \frac{25 B_f^2}{\gamma} \! + \! {10 B_\theta B_w}\! + \! \frac{80 B_w^2 B_\theta^2}{\mu} \Big)^{-1},
        \quad \alpha_w 
                = 
                        \frac{\mu}{40 B_f^2},
        \quad M 
                = 
                        \max \bigg( \frac{10}{\alpha_\theta \mu} \,, \frac{2 {n_X}}{\alpha_w \gamma} \,, 4 n_X \bigg),
                        \nn
        \end{equation}
    then $P_s \eqdef \Expect\|\tilde{\theta}_s \!-\! \theta^*\|^2 
    \!+\!\frac{\gamma}{n_X \mu} \Expect \|\tilde{w}_s \!-\! w^*\|^2 \le (5/8)^s P_0$. Furthermore, let $\kappa = \frac{B_f^2}{\gamma \mu} \!+\! \frac{B_w^2B_\theta^2}{\mu^2}$. Then, the overall computational cost (in number of oracle calls) for reaching $P_s \le \epsilon$ is upper bounded by
        \begin{align}
            O\big(
                ( n_X n_Y + n_X \kappa + n_X ) \ln (1/\epsilon)
            \big).
            \label{e:thm:svrpda1_opt2:total_complexity}
        \end{align}
    \end{theorem}
The above theorems show that the Lyapunov function $P_s$ for \methodexact~converges to zero at a linear rate when either Option I or II is used. Since $\Expect\|\tilde{\theta}_s-\theta^*\|^2 \le P_s$, they imply that the computational cost (in number of oracle calls) for reaching $\Expect\|\tilde{\theta}_s-\theta^*\|^2 \le \epsilon$ is also upper bounded by \eqref{e:thm:svrpda1_opt1:total_complexity} and \eqref{e:thm:svrpda1_opt2:total_complexity}.

\paragraph{Comparison with existing composition optimization algorithms}
Table \ref{tab:comparison_complexity} summarizes the complexity bounds for our \methodexact~algorithm and compares them with existing stochastic composition optimization algorithms. First, to our best knowledge, none of the existing methods consider the general objective function \eqref{e:general_prob_def} as we did. Instead, they consider its special case \eqref{e:gen_prob_def_wang_et_all}, and even in this special case, our algorithm still has better (or comparable) complexity bound than other methods. For example, our bound is better than that of \cite{liawanliu2016fin} since $\kappa^2 > n_X$ generally holds, and it is better than that of ASCVRG, which does not achieve linear convergence rate (as no strong convexity is assumed). 
In addition, our method has better complexity than C-SAGA algorithm when $n_X=1$ (regardless of mini-batch size in C-SAGA), and it is better than C-SAGA for \eqref{e:gen_prob_def_wang_et_all} when the mini-batch size is $1$.\footnote{In Appendix \ref{appendix:mdp}, we also show that our algorithms outperform C-SAGA in experiments.} However, since we have not derived our bound for mini-batch setting, it is unclear which one is better in this case, and is an interesting topic for future work.
One notable fact from Table \ref{tab:comparison_complexity} is that in this special case \eqref{e:gen_prob_def_wang_et_all}, the complexity of \methodexact~is reduced from $O((n_Xn_Y \!+\!n_X\kappa)\ln\frac{1}{\epsilon})$ to $O((n_X\!+\!n_Y\!+\!n_X\kappa)\ln\frac{1}{\epsilon})$. This is because the complexity for evaluating the batch quantities in \eqref{e:svrpda1:batch_quantities} (Algorithm~\ref{alg:svrpda1}) can be reduced from $O(n_X n_Y)$ in the general case \eqref{e:general_prob_def} to $O(n_X+n_Y)$ in the special case \eqref{e:gen_prob_def_wang_et_all}. To see this, note that $f_{\theta}$ and $n_{Y_i}=n_Y$ become independent of $i$ in \eqref{e:gen_prob_def_wang_et_all} and \eqref{e:svrpda1:batch_quantities}, meaning that we can factor $U_0$ in \eqref{e:svrpda1:batch_quantities} as $U_0 = \frac{1}{n_X n_Y} \sum_{j=0}^{n_Y-1}f_{\tilde{\theta}}'(y_j)\sum_{i=0}^{n_X}w_i^{(0)}$,
where the two sums can be evaluated independently with complexity $O(n_Y)$ and $O(n_X)$, respectively. The other two quantities in \eqref{e:svrpda1:batch_quantities} need only $O(n_Y)$ due to their independence of $i$. Second, we consider the further special case of \eqref{e:gen_prob_def_wang_et_all} with $n_X=1$, which simplifies the objective function \eqref{e:general_prob_def} so that there is no empirical average outside $\phi_i(\cdot)$. This takes the form of the unsupervised learning objective function that appears in \cite{liucheden17uns}. Note that our results $O((n_Y\!+\!\kappa)\log\frac{1}{\epsilon})$ enjoys a linear convergence rate (i.e., log-dependency on $\epsilon$) due to the variance reduction technique. In contrast, stochastic primal-dual gradient (SPDG) method in \cite{liucheden17uns}, which does not use variance reduction, can only have sublinear convergence rate (i.e., $O(\frac{1}{\epsilon})$).

\paragraph{Relation to SPDC \cite{zhaxia17sto}}
Lastly, we consider the case where $n_{Y_i}=1$ for all $1 \leq i \leq n_X$ and $f_{\theta}$ is a linear function in $\theta$. This simplifies \eqref{e:general_prob_def} to the problem considered in \cite{zhaxia17sto}, known as the \emph{regularized empirical risk minimization of linear predictors}. It has applications in support vector machines, regularized logistic regression, and more, depending on how the merit function $\phi_i$ is defined. In this special case, the overall complexity for \methodexact~becomes (see Appendix~\ref{appendix:proof_svrpda1_splcase}): 
\begin{align}
O\big(
( {n_X + \kappa} ) \ln ({1}/{\epsilon})
\big) \,,
\label{e:thm:svrpda1_spl_case:total_complexity_main}
\end{align}
where the condition number $\kappa = B_f^2 / \mu \gamma$. In comparison, the authors in \cite{zhaxia17sto} propose a \emph{stochastic primal dual coordinate} (SPDC) algorithm for this special case and prove an overall complexity of 
$O\big(
\big( n_X + {\sqrt{n_X \kappa}} \big) \ln \big(\frac{1}{\epsilon}\big)
\big)$ to achieve an $\epsilon$-error solution. It is interesting to note that the complexity result in \eqref{e:thm:svrpda1_spl_case:total_complexity_main} and the complexity result in \cite{zhaxia17sto} only differ in their dependency on $\kappa$. This difference is most likely due to the acceleration technique that is employed in the primal update of the SPDC algorithm. We conjecture that the dependency on the condition number of SVRPDA-I can be further improved using a similar acceleration technique.

\begin{table}[t]
    \centering
    \caption{The storage complexity of \methodexact~and \methodapprox.}
    \label{tab:storage_complexity}
    {\small
    \setlength\tabcolsep{5pt}
    \begin{tabular}{c|cccccccc|c}
        \toprule
         Methods & $U_0$ & $\{\overline{f}_i\}$ & $\{\overline{f}_i'\}$ & $\theta^{(k)}$ & $\tilde{\theta}$ & $\{w_i^{(k)}\}$ & $\delta_k^\theta$ & $\delta_k^w$ & Total\\
         \midrule
         \methodexact & $O(d)$ & $O(n_X \ell)$ & $O(n_X d \ell)$ & $O(d)$ & $O(d)$ & $O(n_X \ell)$ & $O(d)$ & $O(\ell)$ & $O(n_X d \ell)$ \\
         \methodapprox & $O(d)$ & $O(n_X \ell)$ & \cancel{\phantom{None}} & $O(d)$ & $O(d)$ & $O(n_X \ell)$ & $O(d)$ & $O(\ell)$ & $O(d \!+\! n_X \ell)$ \\
         \bottomrule
    \end{tabular}
    }
\end{table}

\subsection{Storage complexity}
\label{s:theo:storage}
We now briefly discuss and compare the storage complexities of both \methodexact~and \methodapprox. In Table \ref{tab:storage_complexity}, we report the itemized and total storage complexities for both algorithms, which shows that \methodapprox~significantly reduces the memory footprint. We also observe that the batch quantities in \eqref{e:svrpda1:batch_quantities}, especially $\overline{f}_i'(\tilde{\theta})$, dominates the storage complexity in \methodexact.
On the other hand, the memory usage in \methodapprox~is more uniformly distributed over different quantities. Furthermore, although the total complexity of \methodapprox, $O(d+n_X\ell)$, grows with the number of samples $n_X$, the $n_X \ell $ term is relatively small because the dimension $\ell$ is small in many practical problems (e.g., $\ell=1$ in \eqref{e:unsup_obj} and \eqref{e:mean_var_to_prob_def}). This is similar to the storage requirement in SPDC \cite{zhaxia17sto} and SAGA \cite{defazio2014saga}.

\section{Experiments}
\label{sec:experiments}


\begin{figure*}[t]
\centering
\hbox{\hspace{-2em}\includegraphics[trim={0.18cm 3.4cm 0 3cm}, width=1.12\textwidth]{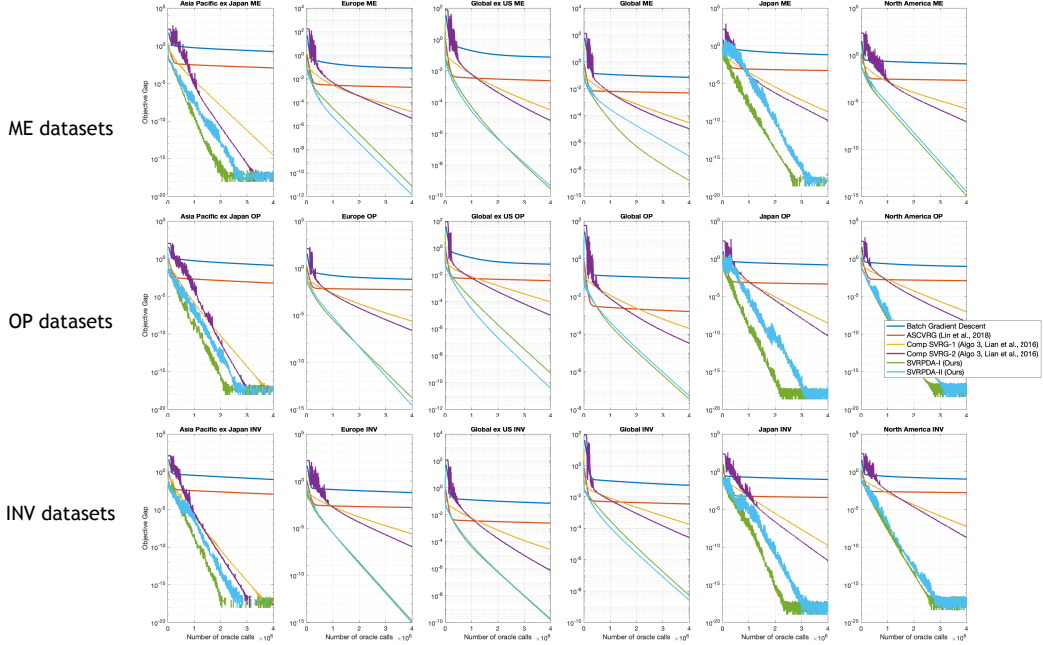}
\vspace{1.5em}
}
\caption{Performance of different algorithms on the risk-averse learning for portfolio management optimization problem. The performance is measured in terms of the number of oracle calls required to achieve a certain objective gap.}
\label{fig:finance:oraclecalls}
\end{figure*}

In this section we consider the problem of risk-averse learning for portfolio management optimization \cite{liawanliu2016fin, linfanwanjor18imp}, introduced in \Section{sec:motivation}.\footnote{Additional experiments on the application to policy evaluation in MDPs can be found in Appendix~\ref{appendix:mdp}.} Specifically, we want to solve the optimization problem \eqref{e:mean_var_to_prob_def} for a given set of reward vectors $\{x_i \in \Re^d: 0 \leq i \leq n-1\}$. As we discussed in Section \ref{sec:motivation}, we adopt the alternative formulation \eqref{e:finance_identification2} for the second term so that it becomes a special case of our general problem \eqref{e:general_prob_def}. Then, we rewrite the cost function into a min-max problem by following the argument in Section \ref{sec:algorithm:minmax} and apply our \methodexact~and \methodapprox~algorithms (see Appendix~\ref{appendix:experiment_details:finance} for the details).

We evaluate our algorithms on 18 real-world US Research Returns datasets obtained from the Center for Research in Security Prices (CRSP) website\footnote{The processed data in the form of .mat file was obtained from \url{https://github.com/tyDLin/SCVRG}}, with the same setup as in \cite{linfanwanjor18imp}. In each of these datasets, we have $d = 25$ and $n = 7240$. We compare the performance of our proposed \methodexact~and \methodapprox~algorithms\footnote{Code: \url{https://github.com/adidevraj/SVRPDA}. The hyper-parameters can be found in Appendix \ref{appendix:experiment_details:hyperparameters}.} with the following state-of-the art algorithms designed to solve composition optimization problems: (i) Compositional-SVRG-1 (Algorithm~$2$ of \cite{liawanliu2016fin}), (ii) Compositional-SVRG-2 (Algorithm~$3$ of \cite{liawanliu2016fin}), (iii) Full batch gradient descent, and (iv) ASCVRG algorithm \cite{linfanwanjor18imp}. For the compositional-SVRG algorithms, we follow \cite{liawanliu2016fin} to formulate it as a special case of the form \eqref{e:gen_prob_def_wang_et_all} by using the identification \eqref{e:finance_identification1}. Note that we cannot use the identification \eqref{e:finance_identification2} for the compositional SVRG algorithms because it will lead to the more general formulation \eqref{e:general_prob_def} with $f_{\theta}$ depending on both $x_i$ and $y_{ij}\equiv x_j$. For further details, the reader is referred to \cite{liawanliu2016fin}.

As in previous works, we compare different algorithms based on the number of \emph{oracle calls} required to achieve a certain objective gap (the difference between the objective function evaluated at the current iterate and at the optimal parameters). One \emph{oracle call} is defined as accessing the function $f_\theta$, its derivative $f'_\theta$, or $\phi_i(u)$ for any $0 \leq i < n$ and $u \in \Re^\ell$. The results are shown in Figure \ref{fig:finance:oraclecalls}, which shows that our proposed algorithms significantly outperform the baseline methods on all datasets. In addition, we also observe that \methodapprox~also converges at a linear rate, and the performance loss caused by the approximation is relatively small compared to \methodexact.




\section{Related Works}
\label{sec:relatedworks}

Composition optimization have attracted significant attention in optimization literature. The stochastic version of the problem \eqref{e:gen_prob_def_wang_et_all}, where the empirical averages are replaced by expectations, is studied in \cite{wanfanliu17sto}. The authors propose a two-timescale stochastic approximation algorithm known as SCGD, and establish \emph{sublinear} convergence rates. In \cite{wanliufan16acc}, the authors propose the ASC-PG algorithm by using a proximal gradient method to deal with nonsmooth regularizations. The works that are more closely related to our setting are \cite{liawanliu2016fin} and \cite{linfanwanjor18imp}, which consider a finite-sum minimization problem \eqref{e:gen_prob_def_wang_et_all} (a special case of our general formulation \eqref{e:general_prob_def}). In \cite{liawanliu2016fin}, the authors propose the compositional-SVRG methods, which combine SCGD with the SVRG technique from \cite{johzha13acc} and obtain \emph{linear} convergence rates. In \cite{linfanwanjor18imp}, the authors propose the ASCVRG algorithms that extends to convex but non-smooth objectives. Recently, the authors in \cite{zhaxia19com} propose a C-SAGA algorithm to solve the special case of \eqref{e:gen_prob_def_wang_et_all} with $n_X=1$, and extend to general $n_X$. Different from these works, we take an efficient primal-dual approach that fully exploits the dual decomposition and the finite-sum structures.


On the other hand, problems similar to \eqref{e:general_prob_def} (and its stochastic versions) are also examined in different specific machine learning problems. \cite{sutton2009fast} considers the minimization of the mean square projected Bellman error (MSPBE) for policy evaluation, which has an expectation inside a \emph{quadratic loss}. The authors propose a two-timescale stochastic approximation algorithm, GTD2, and establish its asymptotic convergence. \cite{liu2015finite} and \cite{macua2015distributed} independently showed that the GTD2 is a stochastic gradient method for solving an equivalent saddle-point problem. In \cite{dai2017learning} and \cite{dai2018sbe}, the authors derived saddle-point formulations for two other variants of costs (MSBE and MSCBE) in the policy evaluation and the control settings, and develop their stochastic primal-dual algorithms. 
All these works consider the stochastic version of the composition optimization and the proposed algorithms have sublinear convergence rates. In \cite{ducheli}, different variance reduction methods are developed to solve the finite-sum version of MSPBE and achieve linear rate even without strongly convex regularization.  Then the authors in \cite{duhu19lin} extends this linear convergence results to the general convex-concave problem with linear coupling and without strong convexity. Besides, problem of the form \eqref{e:general_prob_def} was also studied in the context of unsupervised learning \cite{liucheden17uns,yehcheyuyu18uns} in the stochastic setting (with expectations in \eqref{e:general_prob_def}).

Finally, our work is inspired by the stochastic variance reduction techniques in optimization \cite{le2012stochastic,johzha13acc,defazio2014saga,balamurugan2016stochastic,zhaxia17sto}, which considers the minimization of a cost that is a finite-sum of many component functions. Different versions of variance reduced stochastic gradients are constructed in these works to achieve linear convergence rate. In particular, our variance reduced stochastic estimators are constructed based on the idea of SVRG \cite{johzha13acc} with a novel design of the control variates. Our work is also related to the SPDC algorithm \cite{zhaxia17sto}, which also integrates dual coordinate ascent with variance reduced primal gradient. However, our work is different from SPDC in the following aspects. First, we consider a more general composition optimization problem \eqref{e:general_prob_def} while SPDC focuses on regularized empirical risk minimization with linear predictors, i.e., $n_{Y_i} \equiv 1$ and $f_{\theta}$ is linear in $\theta$. Second, because of the composition structures in the problem, our algorithms also needs SVRG in the dual coordinate ascent update, while SPDC does not. Third, the primal update in SPDC is specifically designed for linear predictors. In contrast, our work is not restricted to that by using a novel variance reduced gradient.

\section{Conclusions and Future Work}
\label{sec:conclusion}
We developed a stochastic primal-dual algorithms, \methodexact~to \emph{efficiently} solve the empirical composition optimization problem. This is achieved by fully exploiting the rich structures inherent in the reformulated min-max problem, including the dual decomposition and the finite-sum structures. 
It alternates between (i) a dual step of stochastic variance reduced coordinate ascent and (ii) a primal step of stochastic variance reduced gradient descent. 
In particular, we proposed a novel variance reduced gradient for the primal update, which achieves better variance reduction with low complexity. We derive a \emph{non-asymptotic} bound for the error sequence and show that it converges at a \emph{linear} rate when the problem is strongly convex. Moreover, we also developed an approximate version of the algorithm named \methodapprox, which further reduces the storage complexity. Experimental results on several real-world benchmarks showed that both \methodexact~and \methodapprox~significantly outperform existing techniques on all these tasks, and the approximation in \methodapprox~only caused a slight performance loss. Future extensions of our work include the theoretical analysis of \methodapprox, the generalization of our algorithms to Bregman divergences, and applying it to large-scale machine learning problems with non-convex cost functions (e.g., unsupervised sequence classifications).

\bibliographystyle{abbrv}
\bibliography{spdgm}






\appendix
\onecolumn
\begin{center}
{\bf \huge Appendix}
\end{center}

\section{Solving \eqref{e:general_prob_def} in the main paper directly by SGD is biased}
\label{appendix:SGD_bias}

Applying the standard chain rule, we obtain the gradient of the cost function in \eqref{e:general_prob_def} as
\begin{equation}
\frac{1}{n_X} \sum_{i = 0}^{n_X-1} \phi'_i \big( \overline{f}_{i}(\theta) \big) \overline{f}_{i}'(\theta)
\label{e:batch_grad_obj}
\end{equation}
where:
\begin{equation}
\begin{aligned}
\overline{f}_{i}(\theta) & \defeq \frac{1}{n_{Y_i}} \sum_{j = 0}^{n_{Y_i}-1}  f_\theta( x_i, y_{ij})
\\
\overline{f}_{i}'(\theta) & \defeq \frac{1}{n_{Y_i}} \sum_{j = 0}^{n_{Y_i}-1}  f'_\theta( x_i, y_{ij})
\label{e:barf_barf_dash}
\end{aligned}
\end{equation}
and $f'_{\theta} \big(x, y \big)$ denotes the $d \times \ell$ matrix, with its $(i,j)^{\text{th}}$ element defined to be:
\begin{equation}
\big[f'_{\theta} \big(x, y \big) \big]_{i,j} = \frac{\partial}{\partial \theta_i} \big[(f_{\theta} \big(x, y \big) \big]_j
\label{e:f_dash_def}
\end{equation}
Note from \eqref{e:batch_grad_obj} that there are empirical averages inside and outside $\phi'(\cdot)$. Therefore, if we sample these empirical averages simultaneously, the stochastic gradient estimator would be biased. In other words, a direct application of stochastic gradient descent to \eqref{e:general_prob_def} would be intrinsically biased.


\begin{algorithm*}[t]
\caption{SVRPDA-II}\label{alg:svrpda2}
\begin{algorithmic}[1]
\STATE {\bf Inputs:} data $\{(x_i,y_{ij}): 0 \!\le\! i \!<\! n_X, 0 \!\le\! j \!<\! n_{Y_i}\}$; step-sizes $\alpha_\theta$ and $\alpha_w$; \# inner iterations $M$.
\STATE {\bf Initialization:} 
$\tilde{\theta}_{0} \in \Re^d$ and $\tilde{w}_{0} \in \Re^{\ell n_X}$.
\FOR{$s=1,2,\ldots$}
\STATE
Set $\tilde{\theta}=\tilde{\theta}_{s-1}$, $\theta^{(0)}=\tilde{\theta}$, $w^{(0)} = \tilde{w}_{s-1}$, and compute the batch quantities (for each $0 \!\le\! i \!<\! n_X$):
    \begin{align}
        U_0
                    &=
                            \!
                            \sum_{i=0}^{n_X\!-\!1}
                            \!
                            \sum_{j=0}^{n_{Y_i}\!-\!1}
                            \!\!
                            \frac{f_{\tilde{\theta}}'(x_i,y_{ij})w_i^{(0)}}{n_X n_{Y_i}}
                            ,
                            \;\;
        \overline{f}_i(\tilde{\theta})
                    \defeq
                            \!
                            \sum_{j=0}^{n_{Y_i}\!-\!1}
                            \!\!
                            \frac{f_{\tilde{\theta}}(x_i,y_{ij})}{n_{Y_i}}.
        \label{e:svrpda2:batch_quantities}
    \end{align}
\FOR{$k=1$ {\bfseries to} $M$}
\STATE 
Randomly sample $i_k \in \{0,\ldots, n_X\!-\!1\}$ and then $j_k \in \{0, \ldots, n_{Y_{i_k}}\!-\!1\}$ at uniform.
\STATE 
Compute the stochastic variance reduced gradient for dual update:
    \begin{align}
        \delta_k^{w}
                    &=
                            f_{\theta^{(k-1)}}(x_{i_k}, y_{i_k j_k})
                            -
                            f_{\tilde{\theta}}(x_{i_k}, y_{i_k j_k})
                            +
                            \overline{f}_{i_k}(\tilde{\theta}).
        \label{e:svrpda2:delta_w_k}
    \end{align}
\STATE Update the dual variables:
\begin{align}
    w^{(k)}_{i} 
                    &= 
                            \begin{cases}
                                \displaystyle 
                                \argmin_{w_i} 
                                \Big [ 
                                    - 
                                    \langle  
                                        \delta_k^{w}, 
                                        \, 
                                        w_i
                                    \rangle 
                                    + 
                                    \phi_i^*(w_i)
                                    +  
                                    \frac{1}{2\alpha_w} \|w_i - w^{(k-1)}_{i} \|^2
                                \Big ]
                                & \text{if } i = i_k
                                \\
                                w^{(k-1)}_{i}
                                & \text{if } i \neq i_k
                            \end{cases}.
        \label{e:svrpda2:w_update}
\end{align}
\STATE 
Update $U_k$ according to the following recursion:
\begin{align}
    U_k
            = 
                    U_{k-1}
                    + 
                    \frac{1}{n_X}
                    f'_{\tilde{\theta}}(x_{i_k}, y_{i_kj_k''})
                    \big(
                        w^{(k)}_{i_k} - w^{(k-1)}_{i_k}   
                    \big).
    \label{e:svrpda2:Ukupdate}
\end{align}
\STATE 
Randomly sample $i_k' \in \{0,\ldots,n_X-1\}$ and then $j_k' \in \{0, \ldots, n_{Y_{i_k'}}-1\}$, independent of $i_k$ and $j_k$, and compute the stochastic variance reduced gradient for primal update:
    \begin{align}
        \delta_k^{\theta}
                &=
                        f_{\theta^{(k-1)}}'(x_{i_k'},y_{i_k' j_k'})w_{i_k'}^{(k)}
                        -
                        f_{\tilde{\theta}}'(x_{i_k'},y_{i_k' j_k'}) w_{i_k'}^{(k)}
                        +
                       U_k.
        \label{e:svrpda2:delta_theta_k}
    \end{align}
\STATE 
Update the primal variable:
    \begin{equation}
        \theta^{(k)} 
                = 
                        \displaystyle 
                        \argmin_{\theta} 
                        \Big[ 
                            \langle 
                                \delta_k^{\theta}, 
                                \, 
                                \theta
                            \rangle 
                            + 
                            g(\theta)
                            +
                            \frac{1}{2\alpha_\theta} 
                            \|\theta - \theta^{(k-1)}\|^2 
                        \Big].
        \label{e:svrpda2:theta_update}
    \end{equation}
\ENDFOR
\STATE 
{\bf{Option I:}} Set $\tilde{w}_{s} = w^{(k)}$ and $\tilde{\theta}_s = \theta^{(k)}$.
\STATE 
{\bf{Option II:}} Set $\tilde{w}_s = w^{(k)}$ and $\tilde{\theta}_s = \theta^{(t)}$ for randomly sampled $t \in \{0,\ldots,M\!-\!1\}$.
\ENDFOR
\STATE {\bf Output:} $\tilde{\theta}_s$ at the last outer-loop iteration.
\end{algorithmic}
\end{algorithm*}

\section{\methodapprox~algorithm}
\label{appendix:svrpda2}

Algorithm \ref{alg:svrpda2} in this supplementary material summarizes the full details of the \methodapprox~algorithm, which was developed in Section \ref{sec:algorithm:vr} of the main paper. Note that it no longer requires the computation or the storage of $\overline{f}_i(\tilde{\theta})$ in \eqref{e:svrpda2:batch_quantities}. Also note that the $\overline{f}_{i_k}(\tilde{\theta})$ in \eqref{e:svrpda2:Ukupdate} is replaced with $f_{\tilde{\theta}}'(x_{i_k}, y_{i_kj_k''})$ now.

\section{Experiment details}
\label{appendix:experiment_details}

\subsection{Implementation details in risk-averse learning}
\label{appendix:experiment_details:finance}

As we discussed in Section \ref{sec:motivation}, we adopt the alternative formulation \eqref{e:finance_identification2} for the second term so that it becomes a special case of our general problem \eqref{e:general_prob_def}. Then, using the argument in Section \ref{sec:algorithm:minmax}, the second term in \eqref{e:mean_var_to_prob_def} can be rewritten into the objective in \eqref{e:gen_prob_min_max}. Combining it with the first term in \eqref{e:mean_var_to_prob_def}, the original problem \eqref{e:mean_var_to_prob_def} can be reformulated into the following equivalent min-max form:
    \begin{align}
        \min_{\theta\in \Re^d} \! \max_{w} \!
        \frac{1}{n} \! \sum_{i = 0}^{n-1} \!\!
        \bigg(\!
        \Big\langle \!
            \frac{1}{n} \! \sum_{j = 0}^{n-1} \langle x_i \!-\! x_j, \theta \rangle , w_i 
            \!
        \Big\rangle 
        \!-\! \phi^*(w_i)
        \!-\! \langle x_i, \theta \rangle 
        \!\bigg)
        \label{e:exper:mean_var_to_prob_def_pd}
    \end{align}
where $w_i \in \Re$, $\phi^*(w_i) = w_i^2 / 4$ and $w = \{w_0,\ldots,w_{n-1}\}$. Note that the above min-max problem has an extra $\langle x_i, \theta\rangle$ term within the sum. Since it is in a standard empirical average form, we can deal with it in a straightforward manner. Notice that \eqref{e:exper:mean_var_to_prob_def_pd} is exactly of the form \eqref{e:general_prob_def} in the main paper except the last term $\langle x_i, \theta \rangle$ within the summation, which as we will show next, can be dealt with in a straightforward manner.

Taking out the $\langle x_i, \theta \rangle$ term in \eqref{e:exper:mean_var_to_prob_def_pd}, based on the discussion in \Section{s:algo} of the main paper, the batch gradients used in the algorithm are as follows. Batch gradient of \eqref{e:exper:mean_var_to_prob_def_pd} with respect to $w_i$, for each $0 \leq i \leq n-1$ can be written as: 
\begin{equation}
\begin{aligned}
\hspace{-0.05in} 
\overline{f}_{i}(\theta) & = \frac{1}{n} \sum_{j = 0}^{n-1}  f_\theta( x_i, x_{j}) = \frac{1}{n} \sum_{j = 0}^{n-1} \langle x_i - x_{j}, \theta \rangle
\end{aligned}
\label{e:barf_pfm}
\end{equation}
Batch gradient of \eqref{e:exper:mean_var_to_prob_def_pd} with respect to $\theta$ (without the $\langle x_i, \theta \rangle$ term) is given by: 
\begin{equation}
\begin{aligned}
L'_\theta(\theta, w)
& = \frac{1}{n^2} \sum_{i = 0}^{n-1} \sum_{j = 0}^{n-1}  f'_{\theta} \big(x_i, x_{j} \big) w_i 
\\
& = \frac{1}{n^2} \sum_{i = 0}^{n-1} \sum_{j = 0}^{n-1}   \big(x_i - x_{j} \big) w_i
\label{e:bg_g_theta_pfm}
\end{aligned}
\end{equation}
For each $0 \leq i \leq n-1$, gradient of $\overline{f}_{\theta,i}(x_i)$ is given by: 
\begin{equation}
\begin{aligned}
\hspace{-0.05in} 
\overline{f}_{i}'(\theta) & \defeq \frac{1}{n} \sum_{j = 0}^{n-1}  f'_\theta( x_i, x_{j}) = x_i - \frac{1}{n} \sum_{j = 0}^{n-1} x_{j}
\label{e:barf_dash_pfm}
\end{aligned}
\end{equation}
Based on the above derivation and the expression \eqref{e:algorithm:svrg_general} in the main paper, the stochastic variance reduced gradient for the dual update in both \methodexact~and \methodapprox~is given by
\begin{equation}
\begin{aligned}
\delta_k^{w}  = \langle x_i - x_{j}, \theta \rangle + \overline{f}_{\tilde{\theta}} (x_i) 
\label{e:vr_dual_pfm}
\end{aligned}
\end{equation}
and the stochastic variance reduced gradient for the primal update is given by
\begin{equation}
\begin{aligned}
\delta_k^{\theta}  & \defeq \big (x_{i} - x_{j} \big) w_{i}
 - (x_{i} - x_{j} \big) w_{i} + \barL'_\theta(\tilde{\theta}, w)
= \barL'_\theta(\tilde{\theta}, w)
\label{e:vr_primal_algo_2_pfm}
\end{aligned}
\end{equation}
Note that, since the function $f_\theta$ is linear in $\theta$, the variance reduced gradient for the primal variable is in-fact the full batch gradient. 

Next, due to the additional $\langle x_i, \theta \rangle$ term in \eqref{e:exper:mean_var_to_prob_def_pd} (which was ommitted in the above definitions), there is an additional term that needs to be added to the variance reduced gradient in \eqref{e:vr_primal_algo_2_pfm}. Denoting $g_\theta(x_i) = \langle x_i \,, \theta \rangle$ and $g'_\theta(x_i) = x_i$ the correction batch term is given by:
\begin{equation}
\begin{aligned}
\overline{g}_\theta' & = \frac{1}{n} \sum_{i=0}^{n-1} g'_\theta(x_i)
= \frac{1}{n} \sum_{i=0}^{n-1} x_i
\end{aligned}
\label{e:batch_corr}
\end{equation}
which is independent of $\theta$. In summary, the final variance reduced stochastic gradient for the primal update in both \methodexact~and \methodapprox~is given by:
\begin{equation}
\begin{aligned}
\delta_k^{\theta}  & = \barL'_\theta(\tilde{\theta}, w) - \frac{1}{n} \sum_{i=0}^{n-1} x_i
\label{e:vr_primal_pfm_corrected}
\end{aligned}
\end{equation}

\subsection{Hyper-parameter choices for algorithms}
\label{appendix:experiment_details:hyperparameters}

In this subsection, we provide the hyper-parameters that are used in our experiments on risk-averse learning (Section \ref{sec:experiments}). We first list the hyper-parameters of our methods below:
\begin{itemize}
    \item 
    SVRPDA-I: $M = n$, $\alpha_\theta=0.0003$, $\alpha_w = 100$.
    \item
    SVRPDA-II: $M = n$, $\alpha_\theta=0.0003$, $\alpha_w = 100$.
\end{itemize}
Then, we provide the hyper-parameters used in the baseline methods:
\begin{itemize}
    \item
    Compositional-SVRG-1 (Algorithm~2 of \cite{liawanliu2016fin}): $K=n$, $A=6$, $\gamma = 0.0003$;
    \item
    Compositional-SVRG-2 (Algorithm~3 of \cite{liawanliu2016fin}): $K=n$, $A=3$, $B = 3$, $\gamma = 0.0004$;
    \item
    ASCVRG: The results are obtained by using their publicly released code on github: \url{https://github.com/tyDLin/SCVRG} with the same setting and choice of hyper-parameters.
    \item
    Batch gradient descent: step size $\alpha = 0.01$;
\end{itemize}
Note that, for the Compotional-SVRGs and batch gradient algorithms, the above choice of the hyper-parameters are obtained by sweeping through a set of hyper-parameters and choosing the ones with the best performance. For ASCVRG, we use the the publicly released code by the authors. 

\section{Additional experiments on MDP policy evaluation}
\label{appendix:mdp}

Consider a Markov decision process (MDP) problem with state space $\state$ and action space $\A$. We assume that both $\state$ and $\A$ are finite, and define $\state = \{ 1 , \ldots, S \}$. For any $1 \leq i, j \leq S$, we denote by $r_{i, j}$ the reward associated with transition from state $i$ to state $j$. Given a policy $\pi: \state \to \mathcal{P}(\A)$, where $\mathcal{P}(\A)$ denote the probability space over $\A$, we let $P^\pi \in \Re^{S \times S}$ denote the associated state transition probability matrix. The goal in policy evaluation is to estimate the value function $V^\pi: \state \to \Re$ associated with the policy $\pi$, which is a fixed-point solution to the following Bellman equation:
\[
V^{\pi}(i) = \sum_{j = 1}^{S} P_{i,j}^\pi \big (r_{i, j} + \gamma V^{\pi} (j) \big )  \,, \qquad 1 \leq i \leq S,
\]
where $0 < \gamma < 1 $ denotes the discount factor. We consider a linear function approximation to the value function: $V^{\pi}(i) \approx \langle \Psi_i \,, \theta \rangle$, where $\{\Psi_i \in \Re^d: 1 \leq i \leq S\}$ denotes the feature vectors, and $\theta \in \Re^d$ denotes the weight vector to be learned. The problem of finding the optimal weight vector $\theta^*$ that best approximates $V^{\pi}$ can be formulated as the following optimization problem \cite{wanliufan16acc,zhaxia19com}:
\begin{equation}
\begin{aligned}
\theta^* & = \argmin_\theta \bigg\{
F(\theta)  \eqdef \frac{1}{S} \sum_{i = 1}^{S} \Big( \langle \Psi_i \,, \theta \rangle  -  \sum_{j = 1}^{S} P_{i,j}^\pi \big (r_{i, j} + \gamma \langle \Psi_j \,, \theta \rangle  \big )  \Big)^2
\bigg\}.
\label{e:appendix:goal_mdp}
\end{aligned}
\end{equation}

Note that the above problem can be expressed as a special case of \eqref{e:general_prob_def} by using the following identifications: For each $1 \leq i \leq S$, $
n_X = n_{Y_i} = S $, $\phi_i (u) = u^2$, and
\begin{equation}
f_\theta(x_i, y_{ij}) = \langle \Psi_i \,, \theta \rangle  -  S \cdot P_{i,j}^\pi \big (r_{i, j} + \gamma \langle \Psi_j \,, \theta \rangle  \big ) \,, \qquad 1 \leq i , j \leq S.
\label{e:appendix:ftheta_svrpda}
\end{equation}
And it is also possible, although less intuitive, to rewrite \eqref{e:appendix:goal_mdp} as a special case of \eqref{e:gen_prob_def_wang_et_all}. In existing composition optimization literature such as \cite{zhaxia19com,wanliufan16acc}, this is achieved via higher dimensional transformation, with $f_\theta: \Re^d \to \Re^{2 S}$, $\phi_i: \Re^{2 S} \to \Re$, $n_Y = S$, and $n_X = S$. Denote by $Q_\theta^\pi$ the Q-function:
\[
Q_\theta^\pi(i) \eqdef \sum_{j = 1}^{S} P_{i,j}^\pi \big (r_{i, j} + \gamma \langle \Psi_j \,, \theta \rangle  \big ) \,, \qquad 1 \leq i \leq S.
\]
Then, by defining the function $f_\theta$ and $\phi_i$ such that
\begin{equation}
\frac{1}{S} \sum_{j = 1}^{S} f_\theta(y_j) = \Big[ \langle \Psi_1, \theta \rangle \,, Q_\theta^\pi(1)\,, \ldots \,, \langle \Psi_S, \theta \rangle \,, Q_\theta^\pi(S)  \Big]
\label{e:appendix:ftheta_csaga}
\end{equation}
\[
\frac{1}{S} \sum_{i = 1}^{S} \phi_i \Bigg( \Big[ \langle \Psi_1, \theta \rangle \,, Q_\theta^\pi(1)\,, \ldots \,, \langle \Psi_S, \theta \rangle \,, Q_\theta^\pi(S)  \Big]  \Bigg) = \frac{1}{S} \sum_{i = 1}^{S} \Big(  \langle \Psi_i, \theta \rangle - Q_\theta^\pi(i)  \Big)^2,
\]
the problem \eqref{e:appendix:goal_mdp} can be reformulated as \eqref{e:gen_prob_def_wang_et_all}. The reader is referred to \cite{zhaxia19com,wanliufan16acc} for more details.

We evaluate our algorithms on two experimental settings, one with $S = 10$, $d = 5$, and another with $S = 10^4$ and $d = 10$. In each of these two cases, $\gamma$ was set to be $0.9$, and both the transition probability matrix $P^\pi \in \Re^{S \times S}$ and the feature vectors $\{\Psi_i \in \Re^d: 1 \leq i \leq S \}$ were randomly generated. We compare the performance our algorithms \methodexact~and \methodapprox~with the C-SAGA algorithm of \cite{zhaxia19com}, as it is the most recent composition optimization that we are aware of, and is shown to be superior to all existing algorithms on this MDP policy evaluation task \cite{zhaxia19com}. The hyper-parameters for the C-SAGA algorithm were chosen as follows. For $S = 10$ case, we choose all hyper-parameters as in \cite{zhaxia19com}: Mini-batch size $s = 1$ and step-size $\eta = 0.1$. For the $S=10^4$ case, we choose mini-batch size $s = 100$, and step-size $\eta =  0.0005$ (see \cite{zhaxia19com} for details on what these hyper-parameters mean). The hyper-parameters for the \methodexact~and \methodapprox~are chosen as follows. For the MDP with $S = 10$, we choose $M = 150$, $\alpha_\theta = 0.1$, and $\alpha_w = 	0.25$ for \methodexact, and choose $M = 15$, $\alpha_\theta = 0.5$, and $\alpha_w = 	1.25$ for \methodapprox. For the MDP with $S = 10^4$, we choose $M = 13500$, $\alpha_\theta = 0.01$ and $\alpha_w = 16 \times 10^4$ for \methodexact. For \methodapprox, $M = 1350$, $\alpha_\theta = 0.01$ and $\alpha_w = 16 \times 10^4$.

The performance criteria was chosen to be the number of oracle calls required to achieve a certain ``objective gap", defined as $F(\theta) - F(\theta^*)$. Notice that ``one function call", when the function is of the form \eqref{e:appendix:ftheta_svrpda} \emph{is not comparable} to one function call, when the function is defined according to \eqref{e:appendix:ftheta_csaga}. Due to the fact that $f_\theta$ for the C-SAGA algorithm is of dimension $2 S$ (as opposed to $1$ in our formulation), we count $2 S$ oracle calls whenever a function of this form is called for the purpose of fair comparison. The results for MDP with $S=10$ and $S = 10^4$ are reported in Figures~\ref{fig:mdp10:oraclecalls} and~\ref{fig:mdp1e4:oraclecalls}, respectively. We observe that despite having much smaller memory requirement, \methodapprox~has a comparable/better performance than C-SAGA, while \methodexact~is clearly better than C-SAGA.

\begin{figure}[H]
\centering
\hbox{\hspace{5mm}\includegraphics[width=0.8\textwidth]{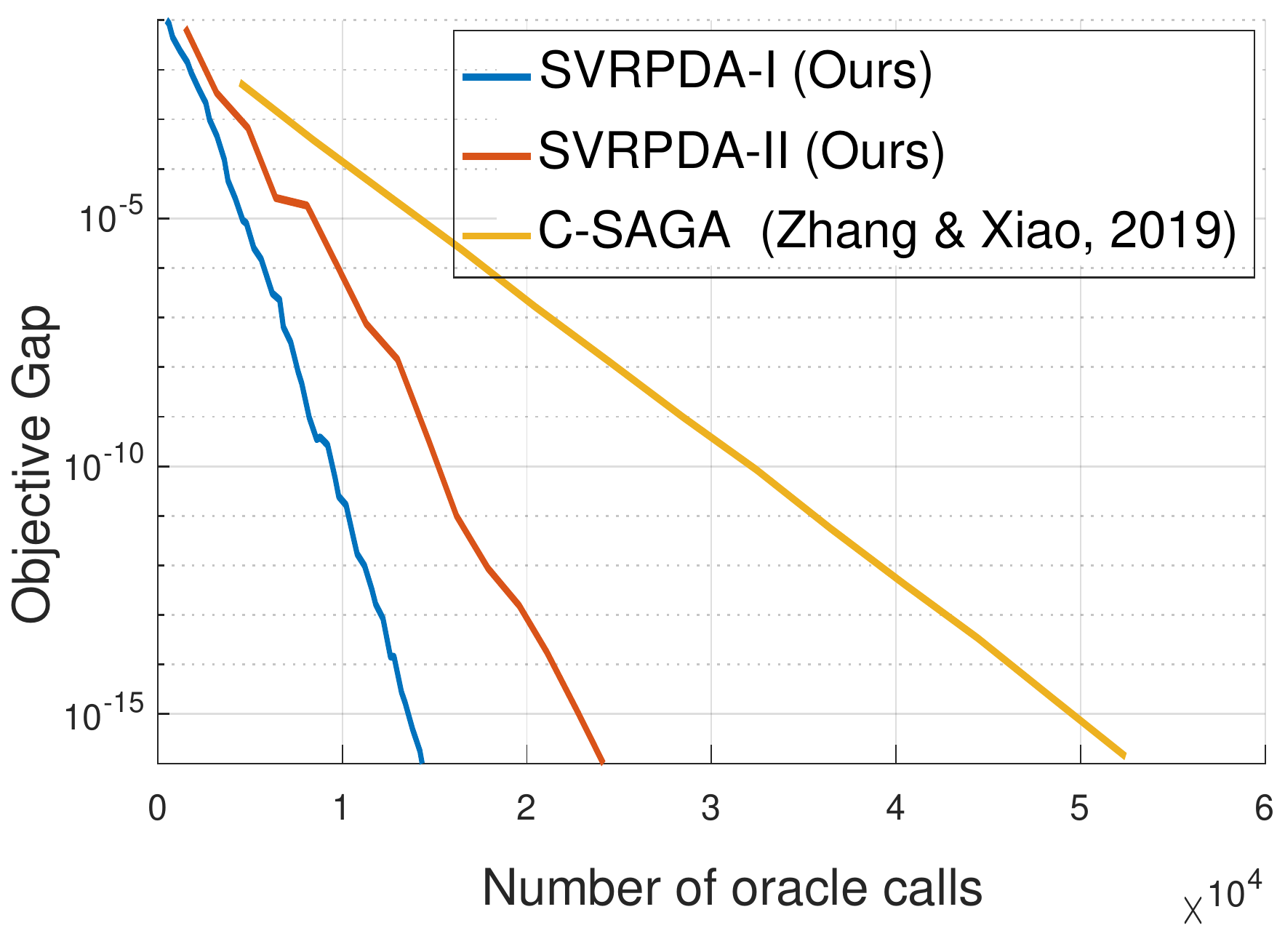}
}
\caption{Performance comparison of our algorithms with the C-SAGA on the MDP with $S = 10$. The performance is measured in terms of the number of oracle calls to achieve a certain objective gap.}
\label{fig:mdp10:oraclecalls}
\end{figure}
\begin{figure}[H]
\centering
\hbox{\hspace{10mm}\includegraphics[width=0.8\textwidth]{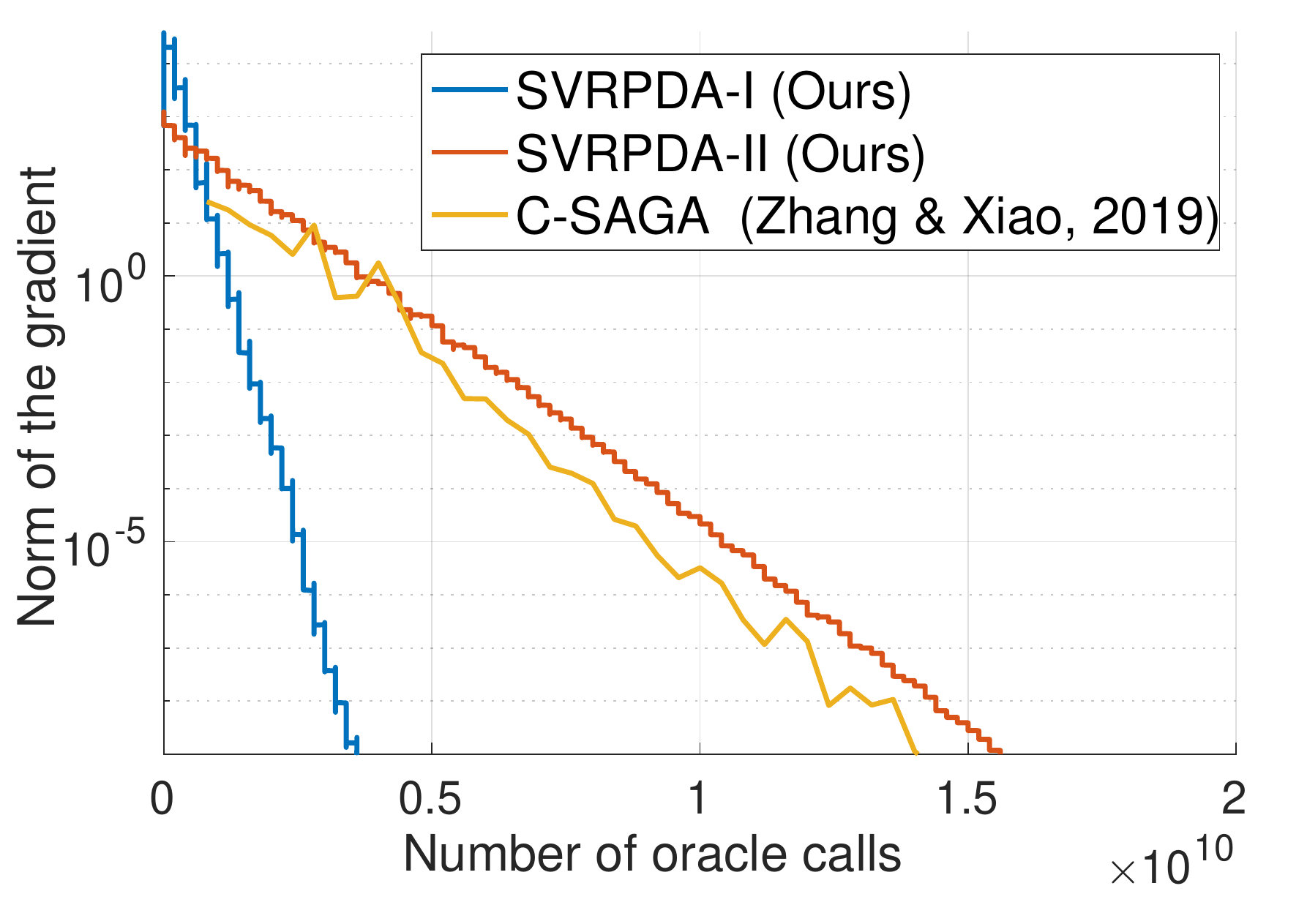}
}
\vspace{1.5em}
\caption{Performance comparison of our algorithms with the C-SAGA on the MDP with $S = 10^4$. The performance is measured in terms of the number of oracle calls to achieve a certain objective gap.}
\label{fig:mdp1e4:oraclecalls}
\end{figure}

\section{Convergence and complexity of SVRPDA-I: Proof}
\label{appendix:proof_svrpda1}

In this section, we derive the (non-asymptotic) bound of the \methodexact~algorithm and its total computation complexity. For convenience, we first repeat the saddle point formulation and the definition of several quantities below:
\begin{align}
\min_{\theta} \max_{w} \frac{1}{n_X} \sum_{i = 0}^{n_X - 1}  \frac{1}{n_{Y_i}} \sum_{j = 0}^{n_{Y_i} - 1}  \Big (  \big \langle  f_\theta( x_i, y_{ij}), \, w_i \big \rangle - \phi_i^* \big(w_i \big) \Big ) + g(\theta)
\label{e:SVRPDA_I:gen_prob_min_max_f_g}
\end{align}
Also, recall the definitions of $\barf_i(\theta)$ and $\barf_i'(\theta)$:
\begin{align}
\overline{f}_{i}(\theta) 
=
\frac{1}{n_{Y_i}} \sum_{j=0}^{n_{Y_i}-1} f_{\theta}(x_i,y_{ij}),
\quad
\overline{f}_{i}'(\theta) 
=        
\frac{1}{n_{Y_i}} \sum_{j=0}^{n_{Y_i}-1} f_{\theta}'(x_i,y_{ij})
\end{align}
Furthermore, we defined $L(\theta,w)$ and its gradient as
\begin{align}
L(\theta, w) 
\eqdef&
\frac{1}{n_X} \sum_{i = 0}^{n_X - 1}  \frac{1}{n_{Y_i}} \sum_{j = 0}^{n_{Y_i} - 1}  \Big (  \big \langle  f_\theta( x_i, y_{ij}), \, w_i \big \rangle - \phi_i^* \big(w_i \big) \Big )
\label{e:SVRPDA_I:b_g}
\\
=&
\frac{1}{n_X} \sum_{i = 0}^{n_X - 1} \Big ( \big \langle  \barf_i(\theta), \, w_i \big \rangle - \phi_i^* \big(w_i \big) \Big )
\label{e:SVRPDA_I:b_g_2}
\\
L_\theta'(\theta, w) 
\eqdef&
\frac{1}{n_X} \sum_{i = 0}^{n_X - 1}   \barf'_i(\theta) w_i 
\label{e:SVRPDA_I:b_g_3}
\end{align}
Using the above notations, the saddle point problem \eqref{e:SVRPDA_I:gen_prob_min_max_f_g} can be rewritten as
\begin{align}
\min_{\theta} \min_{w}
\frac{1}{n_X}
\sum_{i=0}^{n_X-1}
\Big(
\big\langle 
\barf_i(\theta), w_i
\big\rangle
-
\phi_i^{*}(w_i)
\Big)
+ g(\theta)
\label{e:SVRPDA_I:gen_prob_min_max_f_g_interm1}
\end{align}

\subsection{Compact Notation}

Throughout this section, we introduce the following compact notation, to ease exposition of the proof: For any $\theta \in \Re^d$, and $0 \leq i \leq n_X - 1$ and $0 \leq j \leq n_{Y_i} - 1$, we denote:
\begin{equation}
\label{com_notation:eq:fij}
f_{i  j} (\theta) \equiv f_\theta (x_i , y_{ij})  
\end{equation}

Therefore, the stochastic variance reduced gradient for dual update defined in \eqref{e:svrpda1:delta_w_k} of the main paper is rewritten as:
\begin{equation}
\begin{aligned}
\delta_k^w 
&  = f_{i_k j_k} (\theta^{(k-1)} )  -   f_{i_k j_k} (\tiltheta)   +   \barf_{i_k}(\tiltheta)
\end{aligned}
\label{e:SVRPDA_I:delta_def}
\end{equation}
Similarly, the stochastic variance reduced gradient for primal update defined in \eqref{e:svrpda1:delta_theta_k} of the main paper is:
\begin{equation}
\begin{aligned}
\delta_k^\theta
&  =  f'_{i_k' j_k'} \big(\theta^{(k - 1)}\big) w^{(k)}_{i_k'} - f'_{i_k'j_k'} \big( \tiltheta \big) w^{(k)}_{i_k'} + L'_\theta(\tiltheta, w^{(k)})
\end{aligned}
\label{e:SVRPDA_I:delta_theta_def}
\end{equation}
where we used the fact that $U_k \equiv L'_\theta(\tiltheta, w^{(k)})$. We now proceed to recall the Algorithm~\ref{alg:svrpda1} rewritten in a simplified form, using the compact notation.

\subsection{Algorithm}

Before we proceed to prove the convergence of the algorithm, we first recall the update equations of the algorithm. The following updates are at stage $s$ of the outerloop; To simplify exposition, we suppress dependency on $s$, and let $\tiltheta \equiv \tiltheta_s$ throughout.

For the dual update, at each iteration $k$, we first randomly pick an index $0 \leq i_k \leq n_{X} - 1$ at uniform, and then pick another index $0 \leq j_k \leq n_{Y_{i_k}} - 1$ at uniform. For the chosen $(i_k, j_k)$, we first compute the variance reduced stochastic gradient $\delta_k^w$ of $\barf_i(\theta)$ using \eqref{e:SVRPDA_I:delta_def}:
\[
\delta_k^w  =  f_{i_k j_k} (\theta^{(k-1)} )  -   f_{i_k j_k} (\tiltheta)   +   \barf_{i_k}(\tiltheta)
\]
Then, we update the dual variables according to the recursion \eqref{e:svrpda1:w_update}:
\begin{equation}
\begin{aligned}
\hspace{-0.1in}w^{(k)}_{i} &=
\begin{cases}
\displaystyle 
\argmin_{w_i} 
\Big [ 
- \Big \langle \delta_k^w, \, w_i - w_i^{(k - 1)} \Big \rangle 
+ 
\phi_i^*(w_i) 
+ 
\frac{1}{2 \alpha_w} \|w_i - w_i^{(k - 1)} \|^2 
\Big ] 
& \text{if } i = i_k 
\vspace{0.1in}
\\
w^{(k - 1)}_i & \text{if } i \neq i_k
\end{cases}
\label{e:SVRPDA_I:dual_update}
\end{aligned}
\end{equation}
For the primal update, at iteration $k$, we randomly pick another independent set of indices $(i_k',j_k')$ with $0 \leq i_k' \leq n_X - 1$ and $0 \le j_k' \le n_{Y_{i_k}} - 1$, and compute the variance reduced stochastic gradient $\delta_k^\theta$ of $L(\theta,w)$ with respect to $\theta$ using \eqref{e:SVRPDA_I:delta_theta_def}:
\[
\delta_k^\theta= f'_{i_k' j_k'} \big(\theta^{(k - 1)}\big) w^{(k)}_{i_k'} - f'_{i_k'j_k'} \big( \tiltheta \big) w^{(k)}_{i_k'} + L'_\theta(\tiltheta, w^{(k)})
\]
Then, we update the primal variable $\theta$ according to the recursion \eqref{e:svrpda1:theta_update}:
\begin{equation}
\begin{aligned}
\theta^{(k)} 
&= 
\arg\min_{\theta} 
\Big\{ 
\langle \delta_k^\theta, \theta \rangle
+ 
g(\theta)
+ 
\frac{1}{2\alpha_\theta} \| \theta - \theta^{(k-1)} \|^2 
\Big\}
\label{e:SVRPDA_I:primal_update}
\end{aligned}
\end{equation}

\subsection{Assumptions}

We restate the Assumptions in Section~\ref{s:theo} here using the notation in \eqref{com_notation:eq:fij} to make the reading easier:

\begin{assumption}
	\label{a:g_phi:sc_sm}
	The function $g(\theta)$ is $\mu$-strongly convex in $\theta$, and each $\phi_i$ is $1/\gamma$-smooth.
\end{assumption}
\begin{assumption}
	\label{a:phi:lipschitz}
	The merit functions $\phi_i(u)$ are Lipschitz with a uniform constant $B_w$:
	\begin{align}
	| \phi_i(u) - \phi_i(u') |
	&\le
	B_w \| u - u' \|, \quad \forall u, u' \in \Re^\ell, 0 \leq i \leq n_X-1.
	\nn
	\end{align}
\end{assumption}
\begin{assumption}
	\label{a:f:smooth_bg}
	$f_{ij}(\theta)$ is $B_\theta$-smooth in $\theta$, and has bounded gradients with constant $B_f$: For each $0 \leq i \leq n_X-1$ and $0 \leq j \leq n_{Y_i} - 1$,
	\begin{align}
	\| f'_{ij} (\theta_1) - f'_{ij}(\theta_2) \| \leq B_{\theta} \|\theta_1  - \theta_2\|,
	\quad
	\| f'_{ij}(\theta) \| \leq B_f,
	\quad
	\forall \theta, \theta_1, \theta_2 \in \Re^d
	\nn
	\end{align}
\end{assumption}
\begin{assumption}
	\label{a:L:convex}
	For each given $w$ in its domain, the function $L(\theta, w)$ defined in \eqref{e:gen_prob_min_max} is convex in $\theta$:
	\begin{align}
	L(\theta_1, w) - L(\theta_2, w) \geq \langle L'_\theta(\theta_2, w), \,\, \theta_1 - \theta_2 \rangle.
	\nn
	\end{align}
\end{assumption}


\subsection{Preliminary results}

In this subsection, we introduce lemmas which lay the foundation for the proof of the main convergence result that follows. First, our proof relies on the following important lemma, which is a slightly adjusted version of Lemma 3 in \cite{xiao2014proximal} for our problem setting.
\begin{lemma}
	\label{t:lemma_3} Consider any function of the form $P(x) = f(x) + g(x)$, with $x \in \Re^d$.
	Suppose $f(x)$ is {linear in $x$}, and $g(x)$ is $\mu_g$-strongly convex. Then, for {$ \alpha > 0$}, the following holds for any vector $v \in \Re^d$ and $y \in \Re^d$:
	\[
	P(y) \geq P(x^{(+)}) + \frac{1}{\alpha} \langle (x - x^{(+)}), (y - x) \rangle + \langle ( v - f' (x) ),  (x^{(+)} - y) \rangle + { \frac{1}{\alpha} } \| x - x^{(+)} \|^2 + \frac{\mu_g}{2} \| y - x^{(+)} \|^2
	\]
	where:
	\[
	\begin{aligned}
	x^{(+)}
	&= \prox_{\alpha g} \big\{ x - \alpha v \big\}
	\\
	& = \arg\min_{w} \Big\{ g(w) + \frac{1}{2\alpha} \| w - x + \alpha v\|^2 \Big\}        \label{e:lemma:proxsvrg_prelim}
	\end{aligned}
	\]
\end{lemma}
\begin{proof}
	Based on the definition of $x^{(+)}$, the optimality condition associated with the proximal operator states that there exists a sub-gradient $\xi \in \partial g(x^{(+)})$ such that
	\begin{equation}
	\begin{aligned}
	\frac{x^{(+)} - x  +  \alpha v}{\alpha} + \xi & = 0
	\end{aligned}
	\label{e:lemma3:subgrad}
	\end{equation}
	where $\partial g(x^{(+)})$ denotes the sub-differential of $g$ at $x^{(+)}$. Next, by the linearity of $f$ and the strong convexity of $g$, we have, for any $x, y \in \Re^d$,
	\begin{equation}
	\begin{aligned}
	\hspace{-0.32in}P(y) 
	& = f(y) + g(y)
	\\
	& \overset{(a)}{\geq} f(x^{(+)}) + \langle f'(x)\,, (y - x^{(+)}) \rangle + g (x^{(+)}) + \langle \xi\,, (y - x^{(+)}) \rangle + \frac{\mu_g}{2} \| y - x^{(+)} \|^2
	\\
	& \overset{(b)}{=} P(x^{(+)}) + \langle f'(x)\,, (y - x^{(+)}) \rangle + \langle \xi\,, (y - x^{(+)}) \rangle + \frac{\mu_g}{2} \| y - x^{(+)} \|^2
	\\
	& \overset{(c)}{=} P(x^{(+)}) + \langle f'(x)\,, (y - x^{(+)}) \rangle - \frac{1}{\alpha} \langle ( x^{(+)} - x + \alpha v ) \,, (y - x^{(+)}) \rangle + \frac{\mu_g}{2} \| y - x^{(+)} \|^2
	\\
	& \overset{(d)}{=} P(x^{(+)}) + \langle f'(x) - v , (y - x^{(+)}) \rangle - \frac{1}{\alpha} \langle ( x^{(+)} - x) \,, (y - x^{(+)}) \rangle + \frac{\mu_g}{2} \| y - x^{(+)} \|^2
	\\
	& \overset{(e)}{=} P(x^{(+)}) + \langle f'(x) - v , (y - x^{(+)}) \rangle + \frac{1}{\alpha} \| x^{(+)} - x \|^2 
	\\
	& \hspace{0.25in}- \frac{1}{\alpha} \langle ( x^{(+)} - x) \,, (y - x) \rangle + \frac{\mu_g}{2} \| y - x^{(+)} \|^2
	\end{aligned}
	\end{equation}
	where step (a) follows from the linearity of $f$ and the strong convexity of the function $g$, step $(b)$ uses the definition $P(x^{(+)}) = f(x^{(+)}) + g(x^{(+)})$, step $(c)$ substitutes the expression of $\xi$ from \eqref{e:lemma3:subgrad}, step $(d)$ rearrange the second and the third terms, and step $(e)$ completes the proof by adding and subtracting $x$ in the second inner product.
\end{proof}
The difference between our Lemma \ref{t:lemma_3} and Lemma 3 in \cite{xiao2014proximal} is that our function $f(x)$ is linear (instead of being strongly convex) in $x$, which is the setting that we are mainly interested in (i.e., linear dependency on the dual variables). As a result, our $\alpha$ can be any positive number (it is constrained to be smaller than a certain positive number in Lemma 3 of \cite{xiao2014proximal}). This lemma is useful for deriving a bound when the update recursions are defined by a proximal mapping with an arbitrary update vector $v$. This is particularly helpful for our case as both our primal and dual updates are in proximal mapping form with the update vector $v$ being variance reduced stochastic gradient. 

Next, we quote the Lemma 2 of \cite{xiao2014proximal} below:
\begin{lemma}
	Let $R$ be a closed convex function on $\Re^d$ and let $x, y \in \dom(R)$. Then:
	\begin{equation}
	\|\prox_R(x) - \prox_R(y) \| \leq \|x - y\|
	\end{equation}
	\label{t:SVRPDA_I:lemma2}
\end{lemma}

We next introduce a useful property of the conjugate function:
\begin{lemma}
	\label{lemma:conjugate_Lipschitz}
	When Assumption~\ref{a:phi:lipschitz} holds, the domain of $\phi_i^*(w_i)$, denoted as $\mathrm{domain}(\phi_i)$, satisfies
	\begin{align}
	\mathrm{domain}(\phi_i) \subseteq \{w_i: \|w_i\| \le B_w\}
	\end{align}
	That is, for any $w_i$ that satisfies $\|w_i\| > B_w$, we will have $\phi_i^*(w_i) = +\infty$. In consequence, the dual variables $w_i^{(k)}$ obtained from the dual update \eqref{e:SVRPDA_I:dual_update} will always be bounded by $B_w$ throughout the iterations.
\end{lemma}
\begin{proof}
	For any given $w_i$ that satisfies $\|w_i\| > B_w$, define $u_i^t = u_i+\frac{w_i}{\|w_i\|} t$, where $t$ is an arbitrary real scalar. Then, by the definition of conjugate function, we have
	\begin{align}
	\phi_i^*(w_i)
	&=
	\sup_{u_i} \Big[
	\langle w_i, u_i \rangle
	-
	\phi_i(u_i)
	\Big]
	\nn\\
	&\overset{(a)}{\ge}
	\sup_{t}
	\Big[
	\Big\langle w_i, u_i + \frac{w_i}{\|w_i\|} t \Big\rangle
	-
	\phi_i\Big(u_i + \frac{w_i}{\|w_i\|} t\Big)
	\Big]
	\nn\\
	&=
	\sup_{t}
	\Big[
	\|w_i\| t
	-
	\phi_i\Big(u_i + \frac{w_i}{\|w_i\|} t\Big)
	+
	\phi_i(u_i)
	\Big]
	+
	\langle w_i, u_i \rangle
	-
	\phi_i(u_i)
	\nn\\
	&\overset{(b)}{\ge}
	\sup_{t}
	\Big[
	\|w_i\| t
	-
	B_w
	\Big\|
	\frac{w_i}{\|w_i\|} t
	\Big\|
	\Big]
	+
	\langle w_i, u_i \rangle
	-
	\phi_i(u_i)
	\nn\\
	&=
	\sup_{t}
	\Big[
	(\|w_i\| - B_w) t
	\Big]
	+
	\langle w_i, u_i \rangle
	-
	\phi_i(u_i)
	\nn\\
	&=
	+\infty
	\end{align}
	where step (a) uses the fact that the supremum over a subset (line) is smaller, and step (b) uses the following inequality obtained from the Lipschitz property of $\phi_i(u_i)$: 
	\begin{align}
	| \phi_i(u) - \phi_i(u') |
	&\le
	B_w \| u - u' \|, \quad \forall u, u'
	\quad \Rightarrow \quad
	-\phi_i(u) + \phi_i(u')
	\ge
	- B_w \| u - u' \|
	\end{align}
\end{proof}

\subsection{Dual Bound}

In order to derive the bound for the dual update, we first introduce an auxiliary dummy variable $w_{ij}'$:
\begin{align}
w'_{ij} 
&=
\displaystyle 
\argmin_{w_i} 
\Big [ 
- \Big \langle \delta_{i j}^w, \, w_i - w_i^{(k - 1)} \Big \rangle + \phi_i^*(w_i) + \frac{1}{2 \alpha_w} \|w_i - w_i^{(k - 1)} \|^2 
\Big ]
\nn\\
&=
\prox_{\alpha_w \phi_i^*}
\Big[
w_i^{(k-1)} - \alpha_w \delta_{ij}^w
\Big]
\label{e:SVRPDA_I:dual_dummy_variable}
\end{align}
where,
\begin{equation}
\delta_{i j}^w \eqdef f_{i j} (\theta^{(k-1)} )  -   f_{i j} (\tiltheta)   +   \barf_{i}(\tiltheta)
\label{e:SVRPDA_I:delta_ij_w}
\end{equation}
The variable $w_{ij}'$ can be understood as the updated value of the dual variable if $i$ and $j$ is selected. 

Our analysis in this section focuses on deriving bounds for the $\|w^{(k)}-w^*\|^2$. We will first examine $\|w_{ij}' - w_i^*\|^2 $ and then relate it to $\|w^{(k)}-w^*\|^2$. To begin with, for each $i$ and $j$, we have
\begin{equation}
\begin{aligned}
\|w_{ij}' - w_i^*\|^2 
& = \|w_{ij}' - w_i^{(k - 1)} + w_i^{(k - 1)} - w_i^*\|^2
\\
& = \|w_{ij}' - w_i^{(k - 1)} \|^2 + \| w_i^{(k - 1)} - w_i^* \|^2 + 2 \langle (w_{ij}' - w_{i}^{(k - 1)}), (w_{i}^{(k - 1)} - w_{i}^{*}) \rangle
\end{aligned}
\label{e:SVRPDA_I:dummy_initial_eq}
\end{equation}
Now, we upper bound the first and the third terms in \eqref{e:SVRPDA_I:dummy_initial_eq} together. For a given $\theta^{(k)}$ and $i$, define
\begin{equation}
P_{w_i}(x) \eqdef - \big  \langle \barf_i(\theta^{(k - 1)})  , x \big \rangle + \phi_i^*(x)
\label{e:SVRPDA_I:Pwi_no_cor_dual}
\end{equation}
Note that the first part of the function is linear in $x$ and the second part of the function is $\gamma$-strongly convex (since $\phi_i$ is $1/\gamma$-smooth by Assumption~\ref{a:g_phi:sc_sm}). 
Furthermore, by \eqref{e:SVRPDA_I:dual_dummy_variable}, the update rule for the dummy variables $w_{ij}'$ is defined by a proximal operator. Therefore, we can apply \Lemma{t:lemma_3} with $P(x) \equiv P_{w_i}(x)$ and the following identifications:
\[
f(x) = - \big  \langle \barf_i(\theta^{(k - 1)})  , x \big \rangle
\quad g(x) = \phi_i^*(x)
\quad v = - \delta_{i j}^w
\quad x = w_{i}^{(k - 1)} 
\quad x^{(+)} = w_{ij}' 
\quad y = w_i^*
\quad \alpha = \alpha_w
\]
which leads to
\begin{equation}
\begin{aligned}
- \big  \langle \barf_i(\theta^{(k - 1)})  , w_i^* \big \rangle + \phi_i^*(w_i^*) 
\geq 
& - \big  \langle \barf_i(\theta^{(k - 1)})  , w_{ij}' \big \rangle + \phi_i^*(w_{ij}')  
\\
& + \frac{1}{\alpha_w} \langle (w_{i}^{(k - 1)} - w_{ij}'), (w_{i}^{*} - w_{i}^{(k - 1)}) \rangle 
\\
& - \big \langle \delta_{i j}^w - \barf_i(\theta^{(k - 1)}) \,, w'_{ij} - w^*_i \big \rangle
\\
& + \frac{1}{\alpha_w} \| w_{i}^{(k - 1)} - w_{i j}' \|^2 
+ \frac{\gamma}{2} \| w_{i}^{*} - w_{i j}' \|^2
\end{aligned}
\label{e:SVRPDA_I:lemma3_dual}
\end{equation}
Furthermore, by definition, since $w_i^*$ is the optimal solution to the following optimization problem,
\[
w_i^* = \argmin_{w_i} \,\, \Big \{  \phi_i^*(w_i) - \big \langle \barf_{i}(\theta^*), w_i \big \rangle \Big \}
\]
and by the fact that the cost function inside the above $\argmin$ is $\gamma$-strongly convex due to $\phi_i^*(\cdot)$, we have
\begin{equation}
\begin{aligned}
- \Big \langle \barf_i(\theta^*), \,  w_{ij}'  \Big \rangle + \phi_i^*(w_{ij}') \geq
- \Big \langle \barf_i(\theta^*), \, w_i^* \Big \rangle + \phi_i^*(w_i^*) +  \frac{\gamma}{2} {\|w_{ij}' - w_i^* \|^2}
\label{e:SVRPDA_I:w_i_star_is_optimal}
\end{aligned}
\end{equation}
Adding \eqref{e:SVRPDA_I:lemma3_dual} and \eqref{e:SVRPDA_I:w_i_star_is_optimal} cancels the $\phi_i^*$ terms and leads to
\begin{equation}
\begin{aligned}
\big  \langle \barf_i(\theta^{(k - 1)}) - \barf_i(\theta^*) , w_{ij}' - w_i^* \big \rangle
\geq 
&\frac{1}{\alpha_w} \langle (w_{i}^{(k - 1)} - w_{ij}'), (w_{i}^{*} - w_{i}^{(k - 1)}) \rangle 
\\
& - \big \langle \delta_{i j}^w - \barf_i(\theta^{(k - 1)}) \,, w'_{ij} - w^*_i \big \rangle
\\
& + \frac{1}{\alpha_w} \| w_{i}^{(k - 1)} - w_{ij}' \|^2 
+ {\gamma} \| w_{i}^{*} - w_{ij}' \|^2
\end{aligned}
\label{e:SVRPDA_I:initial_ineq}
\end{equation}
Multiplying both sides by $2\alpha_w$ and rearranging the terms, we obtain
\begin{equation}
\begin{aligned}
& \hspace{-0.32in} 2 \alpha_w \big  \langle \barf_i(\theta^{(k - 1)}) - \barf_i(\theta^*) , w_{ij}' - w_i^* \big \rangle 
\!+\! 
2 \alpha_w \big \langle \delta_{i j}^w - \barf_i(\theta^{(k - 1)}) \,, w'_{ij} - w^*_i \big \rangle
\! - \! 
2 \alpha_w \gamma \|w_{ij}' - w_i^*\|^2
\\
& \geq 
2 \langle (w_{i}^{(k - 1)} - w_{ij}'), (w_{i}^{*} - w_{i}^{(k - 1)}) \rangle 
\rangle 
+ 2 \| w_{i}^{(k - 1)} - w_{ij}' \|^2
\end{aligned}
\label{e:SVRPDA_I:bd_inner_and_norm_sq}
\end{equation}
Now, observe that inequality \eqref{e:SVRPDA_I:bd_inner_and_norm_sq} could be used as an uppper bound for the first and third terms on the right hand side of \eqref{e:SVRPDA_I:dummy_initial_eq}. Using this, \eqref{e:SVRPDA_I:dummy_initial_eq} becomes:
\begin{equation}
\begin{aligned}
\|&w_{ij}' - w_i^*\|^2 \\
& = \| w_i^{(k - 1)} - w_i^* \|^2 + \|w_{ij}' - w_i^{(k - 1)} \|^2 + 2 \langle (w_{ij}' - w_{i}^{(k - 1)}), (w_{i}^{(k - 1)} - w_{i}^{*}) \rangle
\\
& \overset{(a)}{=}
\| w_i^{(k - 1)} - w_i^* \|^2 
-
\|w_{ij}' - w_i^{(k - 1)} \|^2
+ 
2\|w_{ij}' - w_i^{(k - 1)} \|^2 + 2 \langle (w_{ij}' - w_{i}^{(k - 1)}), (w_{i}^{(k - 1)} - w_{i}^{*}) \rangle
\\
&\le
\| w_i^{(k - 1)} - w_i^* \|^2 
- \|w_{ij}' - w_i^{(k - 1)} \|^2 
+ 2 \alpha_w \big  \langle \barf_i(\theta^{(k - 1)}) - \barf_i(\theta^*) , w_{ij}' - w_i^* \big \rangle 
\\
& + {2 \alpha_w \big \langle \delta_{i j}^w - \barf_i(\theta^{(k - 1)}) \,, w'_{ij} - w^*_i \big \rangle}
- 2 \alpha_w \gamma \|w_{ij}' - w_i^*\|^2
\end{aligned}
\label{e:SVRPDA_I:dummy_initial_eq_bdd}
\end{equation}
where step (a) added and subtracted a $\| w_{i}^{(k - 1)} - w_{ij}' \|^2$ in order to apply \eqref{e:SVRPDA_I:bd_inner_and_norm_sq} in the following inequality. Dividing both sides by $2 \alpha_w$ and combining common terms, we get the following bound:
\begin{equation}
\begin{aligned}
\hspace{-0.32in} \Big( \frac{1}{2 \alpha_w} + \gamma \Big)\|w_{ij}' - w_i^*\|^2
& \leq \frac{1}{2 \alpha_w} \| w_i^{(k - 1)} - w_i^* \|^2 
- \frac{1}{2 \alpha_w} \|w_{ij}' - w_i^{(k - 1)} \|^2
\\
& + \big  \langle \barf_i(\theta^{(k - 1)}) - \barf_i(\theta^*) , w_{ij}' - w_i^* \big \rangle 
+ \big \langle \delta_{i j}^w - \barf_i(\theta^{(k - 1)}) \,, w'_{ij} - w^*_i \big \rangle
\end{aligned}
\label{e:SVRPDA_I:dummy_bd}
\end{equation}

Next, we will bound the last term in \eqref{e:SVRPDA_I:dummy_bd}.
Consider the full batch dual ascent algorithm. In this case, for each $0 \leq i \leq n_X - 1$, the update rule is given by:
\begin{equation}
\begin{aligned}
\barw^{(k)}_{i} 
&=
\displaystyle 
\argmin_{w_i} 
\Big [ 
- 
\Big \langle \barf_i(\theta^{(k - 1)}), \, w_i \Big \rangle + \phi_i^*(w_i) + 
\frac{1}{2 \alpha_w} \|w_i - w_i^{(k - 1)} \|^2 
\Big ]
\\
&=
\prox_{\alpha_w \phi_i^*}
\Big[
w_i^{(k-1)} - \alpha_w \barf_i(\theta^{(k-1)})
\Big]
\label{e:SVRPDA_I:PDA_SCO_batch_dual_prox_svrg_2}
\end{aligned}
\end{equation}
The above update rule will only be used for analysis.
Considering the last term in the right hand side of \eqref{e:SVRPDA_I:dummy_bd}; we have:
\begin{equation}
\begin{aligned}
\hspace{-0.3in} \big \langle \delta_{i j}^w - \barf_i(\theta^{(k - 1)}) \,, w'_{ij} - w^*_i \big \rangle
\! = \! 
& 
\big \langle 
\delta_{i j}^w - \barf_i(\theta^{(k - 1)}) \,, w'_{ij} - \barw^{(k)}_i 
\big \rangle
\! + \! 
\big \langle 
\delta_{i j}^w - \barf_i(\theta^{(k - 1)}) \,, \barw^{(k)}_i - w_i^* 
\big \rangle
\\
\overset{(a)}{\leq}
& 
\| \delta_{i j}^w - \barf_i(\theta^{(k - 1)}) \| 
\cdot 
\| w_{ij}' - \barw^{(k)}_i \|
+ 
\big \langle 
\delta_{i j}^w - \barf_i(\theta^{(k - 1)}) \,, \barw^{(k)}_i - w_i^* 
\big\rangle
\\
\overset{(b)}{\leq}
& 
\alpha_w \| \delta_{i j}^w - \barf_i(\theta^{(k - 1)}) \|^2
+ 
\big \langle 
\delta_{i j}^w - \barf_i(\theta^{(k - 1)}) \,, \barw^{(k)}_i - w_i^* 
\big \rangle
\end{aligned}
\label{e:SVRPDA_I:delta_ij_term}
\end{equation}
where $(a)$ uses Cauchy-Schwarz inequality, and step $(b)$ substitutes the proximal expressions of $w_{ij}'$ in \eqref{e:SVRPDA_I:dual_dummy_variable} and $\barw_i^{(k)}$ in \eqref{e:SVRPDA_I:delta_ij_term} followed by \Lemma{t:SVRPDA_I:lemma2}.  
Averaging both sides of the inequality over all $0 \leq j \leq n_{Y_i} - 1$ and using the fact that the average of the second term on the right hand side of \eqref{e:SVRPDA_I:delta_ij_term} is zero, we get:
\begin{align}
&\frac{1}{n_{Y_i}} \sum_{j = 0}^{n_{Y_i} - 1} 
\big[ 
\big \langle 
\delta_{i j}^w - \barf_i(\theta^{(k - 1)}) \,, w_{ij}' - w^*_i 
\big \rangle 
\big ] 
\nn\\
\leq& 
\frac{\alpha_w}{n_{Y_i}} \sum_{j = 0}^{n_{Y_i} - 1} 
\big[ 
\| \delta_{i j}^w - \barf_i(\theta^{(k - 1)}) \|^2 
\big ]
\nn\\
=&
\frac{\alpha_w}{n_{Y_i}} \sum_{j = 0}^{n_{Y_i} - 1} 
\big[ 
\| 
f_{ij}(\theta^{(k-1)}) - f_{ij}(\tiltheta) + \barf_{i}(\tiltheta)
- \barf_i(\theta^{(k-1)})
\|^2 
\big ]
\nn\\
\overset{(a)}{=}&
\frac{\alpha_w}{n_{Y_i}} \sum_{j = 0}^{n_{Y_i} - 1} 
\Big\| 
\Big(1-\frac{1}{n_{Y_i}}\Big)
f_{ij}(\theta^{(k-1)}) 
\nn\\
& \hspace{0.34in} - 
\Big(1 \! - \!\frac{1}{n_{Y_i}}\Big)
f_{ij}(\tiltheta) 
\! - \! 
\Big(
\barf_i(\theta^{(k-1)})
\! - \!
\barf_{i}(\tiltheta)
\! - \!
\frac{1}{n_{Y_i}} f_{ij}(\theta^{(k-1)})
\! + \!
\frac{1}{n_{Y_i}} f_{ij}(\tiltheta)
\Big)
\Big\|^2
\nn\\
\overset{(b)}{\le}&
\frac{2\alpha_w}{n_{Y_i}} \sum_{j = 0}^{n_{Y_i} - 1} 
\Bigg[
\Big\| 
\Big(1-\frac{1}{n_{Y_i}}\Big)
f_{ij}(\theta^{(k-1)}) 
- 
\Big(1-\frac{1}{n_{Y_i}}\Big)
f_{ij}(\tiltheta) 
\Big\|^2
\nn\\
& \hspace{1.44in} +
\Big\|
\barf_i(\theta^{(k-1)})
-
\barf_{i}(\tiltheta)
-
\frac{1}{n_{Y_i}} f_{ij}(\theta^{(k-1)})
+
\frac{1}{n_{Y_i}} f_{ij}(\tiltheta)
\Big\|^2
\Bigg]
\nn\\
=&
\frac{2\alpha_w}{n_{Y_i}} \sum_{j = 0}^{n_{Y_i} - 1} 
\Bigg[
\Big\| 
\Big(1-\frac{1}{n_{Y_i}}\Big)
f_{ij}(\theta^{(k-1)}) 
- 
\Big(1-\frac{1}{n_{Y_i}}\Big)
f_{ij}(\tiltheta) 
\Big\|^2
\nn\\
& \hspace{1.44in} +
\Big\|
\frac{n_{Y_i}-1}{n_{Y_i}}
\frac{1}{n_{Y_i}-1}
\sum_{\ell \neq j}
\big[
f_{i\ell}(\theta^{k-1})
-
f_{i\ell}(\tiltheta)
\big]
\Big\|^2
\Bigg]
\nn\\
\overset{(c)}{\le}&
\frac{2\alpha_w}{n_{Y_i}} \sum_{j = 0}^{n_{Y_i} - 1} 
\Bigg[
\Big\| 
\Big(1-\frac{1}{n_{Y_i}}\Big)
f_{ij}(\theta^{(k-1)}) 
- 
\Big(1-\frac{1}{n_{Y_i}}\Big)
f_{ij}(\tiltheta) 
\Big\|^2
\nn\\
& \hspace{1.44in} +
\Big(\frac{n_{Y_i}-1}{n_{Y_i}}\Big)^2
\frac{1}{n_{Y_i}-1}
\sum_{\ell \neq j}
\Big\|
\big[
f_{i\ell}(\theta^{k-1})
-
f_{i\ell}(\tiltheta)
\big]
\Big\|^2
\Bigg]
\nn\\
\overset{(d)}{\le}&
2\alpha_w 
\bigg[
\Big( 1 - \frac{1}{n_{Y_i}} \Big)^2
B_f^2 \| \theta^{(k-1)} - \tiltheta \|^2
+
\frac{n_{Y_i}-1}{n_{Y_i}^2}
B_f^2 \| \theta^{(k-1)} - \tiltheta \|^2
\bigg]
\nn\\
=&
2\alpha_w B_f^2 \Big( 1 - \frac{1}{n_{Y_i}}\Big) 
\| \theta^{(k-1)} - \tiltheta \|^2
\label{e:SVRPDA_I:expect_delta_ij_term}
\end{align}
where step (a) follows by adding and subtracting $\frac{1}{n_{Y_i}} \big( f_{ij}(\theta^{(k-1)}) \!-\! f_{ij}(\tiltheta) \big)$, step (b) uses $\|a+b\|^2 \le 2\|a\|^2+2\|b\|^2$, step (c) applies Jensen's inequality to the second term, step (d) applies Assumption~\ref{a:f:smooth_bg} (uniformly bounded gradients implies the uniform Liphschitz continuity of the functions $f_{ij}$) to both terms. Averaging both sides of \eqref{e:SVRPDA_I:dummy_bd} over all $0 \leq j \leq n_{Y_i} - 1$, and using \eqref{e:SVRPDA_I:expect_delta_ij_term}, we obtain:
\begin{equation}
\begin{aligned}
& \hspace{-0.45in} \Bigg( \frac{1}{2 \alpha_w} + \gamma \Bigg) \frac{1}{n_{Y_i}} \sum_{j = 0}^{n_{Y_i} - 1} \|w_{ij}' - w_i^*\|^2
\! \leq \! \frac{1}{2 \alpha_w} \| w_i^{(k - 1)} - w_i^* \|^2 
- \frac{1}{2 \alpha_w} \frac{1}{n_{Y_i}} \sum_{j = 0}^{n_{Y_i} - 1} \|w_{ij}' - w_i^{(k - 1)} \|^2
\\
& \hspace{0.3in} + \frac{1}{n_{Y_i}} \sum_{j = 0}^{n_{Y_i} - 1} \big  \langle \barf_i(\theta^{(k - 1)}) - \barf_i(\theta^*) , w_{ij}' - w_i^* \big \rangle 
+ 2\alpha_w B_f^2 \Big( 1 - \frac{1}{n_{Y_i}} \Big) \|\theta^{(k-1)} - \tiltheta  \|^2
\end{aligned}
\label{e:SVRPDA_I:dummy_bd_sum_j}
\end{equation}

Finally, we relate the bound for $w_{ij}'$ back to the bound for the dual variable $w_i^{(k)}$. Recall that, for each $w_i$, there is a probability $1/n_X$ that it will be selected and updated, and a probability of $(n_X-1)/n_X$ that it will be kept the same as $w_i^{(k-1)}$. Furthermore, conditioned on the fact that $w_i$ is selected, it will be updated to $w_{ij}'$ with probability $1/n_{Y_i}$. Therefore, for each $w_i$, there is a probability $1/n_X n_{Y_i}$ that it will be updated to $w_{ij}'$ for $j=0,\ldots,n_{Y_i}-1$, and a probability of $(n_X-1)/n_X$ that it remains the same. Therefore, letting $\clF_k$ denote the filtration of all events upto the beginning of iteration $k$ (before the dual update step), we have:
\begin{equation}
\begin{aligned}
\Expect \{ w_{i}^{(k)} \mid \clF_k \} &= \frac{1}{n_X} \frac{1}{n_{Y_i}} \sum_{j=0}^{n_{Y_i}-1} w'_{ij} + \frac{n_X-1}{n_X} w_{i}^{(k - 1)} 
\\
\Expect \{ w_{i}^{(k)} - w_i^*\mid \clF_k \} &= \frac{1}{n_X} \frac{1}{n_{Y_i}} \sum_{j=0}^{n_{Y_i}-1} (w'_{ij} - w_i^*) + \frac{n_X-1}{n_X} (w_{i}^{(k - 1)} - w_i^*)
\\
\Expect \{ \| w_{i}^{(k)} - w_i^*\|^2 \mid \clF_k \} &= \frac{1}{n_X} \frac{1}{n_{Y_i}} \sum_{j=0}^{n_{Y_i}-1} \| w'_{ij} - w_i^*\|^2   + \frac{n_X-1}{n_X} \| w_{i}^{(k - 1)} - w_i^*\|^2
\\
\Expect \{ \| w_{i}^{(k)} - w_i^{(k - 1)}\|^2 \mid \clF_k \} &= \frac{1}{n_X} \frac{1}{n_{Y_i}} \sum_{j=0}^{n_{Y_i}-1} \| w'_{ij} - w_i^{(k - 1)}\|^2
\end{aligned}
\label{e:SVRPDA_I:expec_w_i}
\end{equation}
Using \eqref{e:SVRPDA_I:expec_w_i} in \eqref{e:SVRPDA_I:dummy_bd_sum_j}, we obtain:
\begin{equation}
\begin{aligned}
&\hspace{-1.2in} \Bigg( \frac{1}{2 \alpha_w} + \gamma \Bigg) \Bigg( n_X \Expect \{ \|w_{i}^{(k)} - w_i^*\|^2 \mid \clF_k \} - (n_X - 1) \| w_i^{(k - 1)} - w_i^* \|^2  \Bigg)
\\
& \leq \frac{1}{2 \alpha_w} \| w_i^{(k - 1)} - w_i^* \|^2 
- \frac{n_X}{2 \alpha_w} \Expect \{ \|w_{i}^{(k)} - w_i^{(k - 1)} \|^2 \mid \clF_k \}
\\
& + {2\alpha_w B_f^2 \Big( 1 - \frac{1}{n_{Y_i}} \Big) \|\theta^{(k-1)} - \tiltheta  \|^2}
\\
& + n_X \Expect \Big \{ \big  \langle \barf_i(\theta^{(k - 1)}) - \barf_i(\theta^*) , w_{i}^{(k)} - w_i^* \big \rangle \mid \clF_k \Big \}
\\
& - \big (n_X - 1 \big ) \Big \{ \big \langle \barf_i(\theta^{(k - 1)}) - \barf_i(\theta^*) , w_{i}^{(k - 1)} - w_i^* \big \rangle \Big \}
\end{aligned}
\label{e:SVRPDA_I:w_i_k_bd}
\end{equation}
Summing both sides of \eqref{e:SVRPDA_I:w_i_k_bd} over $0 \leq i \leq n_X-1$, using the fact that $\|w^{(k)} - w^*\|^2 = \sum_{i=0}^{n_X-1} \| w_i^{(k)} - w_i^* \|^2$, and then dividing by $n_X$, we get
\begin{equation}
\begin{aligned}
& \Bigg( \frac{1}{2 \alpha_w} + \gamma \Bigg) \Bigg( \Expect \{ \|w^{(k)} - w^*\|^2 \mid \clF_k \} - \frac{n_X - 1}{n_X} \| w^{(k - 1)} - w^* \|^2  \Bigg)
\\
& \leq \frac{1}{2 \alpha_w n _X} \| w^{(k - 1)} - w^* \|^2 
- \frac{1}{2 \alpha_w} \Expect \{ \|w^{(k)} - w^{(k - 1)} \|^2 \mid \clF_k \}
\\
& + {2\alpha_w B_f^2 \Big( 1 - \overline{1/n_{Y}} \Big) \|\theta^{(k-1)} - \tiltheta  \|^2}
\\
& + n_X \Expect \Big \{ L(\theta^{(k - 1)}, w^{(k)} - w^*) - L(\theta^*, w^{(k)} - w^*) \mid \clF_k \Big \}
\\
& - \big(n_X - 1) \Big \{ L(\theta^{(k - 1)}, w^{(k - 1)} - w^*) - L(\theta^*, w^{(k - 1)} - w^*) \Big \}
\end{aligned}
\label{e:SVRPDA_I:dual_bd_1}
\end{equation}
where, we have used the notation:
\begin{equation}
\overline{1/n_{Y}} \eqdef \frac{1}{n_X} \sum_{i = 0}^{n_X - 1} 1/n_{Y_i}
\end{equation}




Rearranging and combining the common terms, we obtain the final dual bound:
\begin{equation}
\begin{aligned}
& \Bigg( \frac{1}{2 \alpha_w} + \gamma \Bigg) \Expect \{ \|w^{(k)} - w^*\|^2 \mid \clF_k \} + \frac{1}{2 \alpha_w} \Expect \{ \|w^{(k)} - w^{(k - 1)} \|^2 \mid \clF_k \}
\\
& \leq \Bigg( \frac{1}{2 \alpha_w} + \frac{\gamma \big(n_X - 1 \big)}{n_X} \Bigg) \| w^{(k - 1)} - w^* \|^2 
+ {2\alpha_w B_f^2 \Big( 1 - \overline{1/n_{Y}} \Big) \|\theta^{(k-1)} - \tiltheta  \|^2}
\\
& + \Expect \Big \{ L(\theta^{(k - 1)}, w^{(k)} - w^*) - L(\theta^*, w^{(k)} - w^*) \mid \clF_k \Big \}
\\
& + \big(n_X - 1) \Expect \Big \{ L(\theta^{(k - 1)}, w^{(k)} - w^{(k - 1)}) - L(\theta^*, w^{(k)} - w^{(k - 1)}) \mid \clF_k \Big \}
\end{aligned}
\label{e:SVRPDA_I:dual_bd_2}
\end{equation}
Note that the above bound still have terms related to $L(\cdot,\cdot)$. We will combine these terms together with the $L$ terms in the primal bound, and then bound them all together thereafter.

\subsection{Primal Bound}
Now we proceed to derive the bound for the primal variable $\theta$. Specifically, we will focus on examining $\|\theta^{(k)} - \theta^*\|^2$, which can be written as
\begin{equation}
\begin{aligned}
\|\theta^{(k)} - \theta^*\|^2 
& = \|\theta^{(k)} - \theta^{(k - 1)} + \theta^{(k - 1)} - \theta^*\|^2
\\
& = \|\theta^{(k)} - \theta^{(k - 1)} \|^2 + \| 
\theta^{(k - 1)} - \theta^* \|^2 + 2 \langle (\theta^{(k)} - \theta^{(k - 1)}), (\theta^{(k - 1)} - \theta^{*}) \rangle
\end{aligned}
\label{e:SVRPDA_I:batch_dual_prox_svrg:initial_eq_primal}
\end{equation}
Similar to the dual bound, we now bound the first term and the third term together. Introduce the following function of $x$ (for fixed $\theta^{(k-1)}$ and $w^{(k)}$):
\begin{equation}
P_{\theta}(x) \eqdef \big \langle L'_\theta(\theta^{(k - 1)}, w^{(k)}) ,\, x  \big \rangle + {g(x)}
\label{e:SVRPDA_I:batch_dual_prox_svrg:Pwtheta}
\end{equation}
The first part of the function is linear in $x$ (and hence convex), and the second part of the function is $\mu$-strongly convex (Assumption~\ref{a:g_phi:sc_sm}).
Recall the primal update rule in \eqref{e:SVRPDA_I:primal_update}, which can be written in the following proximal mapping form:
\begin{align}
\theta^{(k)} 
&= 
\arg\min_{\theta} 
\Big\{ 
\langle \delta_k^\theta, \theta \rangle
+ g(\theta)
+ \frac{1}{2\alpha_\theta} \| \theta - \theta^{(k-1)} \|^2 
\Big\}
\nn\\
&=
\prox_{\alpha_\theta g}
\Big\{
\theta^{(k-1)} - \alpha_{\theta} \delta_k^{\theta}
\Big\}
\label{e:SVRPDA_I:SPDA_primal_interm1}
\end{align}
We now apply \Lemma{t:lemma_3} with $P(x) \equiv P_{\theta}(x)$ and the following identifications:
\[
f(x) = \big \langle L'_\theta(\theta^{(k - 1)}, w^{(k)}) ,\, x  \big \rangle
\quad g(x) = g(x)
\quad v = \delta_k^\theta
\quad x = \theta^{(k - 1)} 
\quad x^{(+)} = \theta^{(k)} 
\quad y = \theta^*
\quad \alpha = \alpha_\theta
\]
which leads to
\begin{equation}
\begin{aligned}
\Big \langle L'_\theta(\theta^{(k - 1)}, w^{(k)}) ,\, \theta^{*}  \Big \rangle + g(\theta^{*})
\geq 
& \Big \langle L'_\theta(\theta^{(k - 1)}, w^{(k)}) ,\, \theta^{(k)}  \Big \rangle + g(\theta^{(k)})  
\\
& + \frac{1}{\alpha_\theta} \langle (\theta^{(k - 1)} - \theta^{(k)}), (\theta^{*} - \theta^{(k - 1)}) \rangle 
\\
& - {\langle ( \delta_k^\theta- L'_\theta(\theta^{(k - 1)}, w^{(k)}) ),  (\theta^{(k)} - \theta^*) \rangle }
\\
& + \frac{1}{\alpha_\theta} \| \theta^{(k - 1)} - \theta^{(k)} \|^2 
+ \frac{\mu}{2} \| \theta^{*} - \theta^{(k)} \|^2
\end{aligned}
\end{equation}
Rearranging the terms in the above inequality, we obtain
\begin{align}
&\| \theta^{(k-1)} - \theta^{(k)} \|^2
+
\langle
\theta^{(k)} - \theta^{(k-1)},
\theta^{(k-1)} - \theta^*
\rangle
\nn\\
\le&
\alpha_\theta\langle
L'_{\theta}(\theta^{(k-1)}, w^{(k)}), \theta^*
\rangle
+
\alpha_\theta g(\theta^*)
- \alpha_\theta g(\theta^{(k)})
-
\alpha_\theta \langle L'_\theta(\theta^{(k-1)}, w^{(k)}), \theta^{(k)}\rangle 
\nn\\
&
+
\alpha_\theta
\langle
\delta_k^\theta- L'(\theta^{(k-1)}, w^{(k)}), \theta^{(k)}-\theta^*\rangle
-
\frac{\alpha_\theta \mu}{2}
\| \theta^* - \theta^{(k)} \|^2
\label{e:SVRPDA_I:batch_dual_prox_svrg:lemma3_primal}
\end{align}
Using \eqref{e:SVRPDA_I:batch_dual_prox_svrg:lemma3_primal} to bound the first and the third term in \eqref{e:SVRPDA_I:batch_dual_prox_svrg:initial_eq_primal}, we obtain:
\begin{equation}
\begin{aligned}
\hspace{-0.2in} \|\theta^{(k)} - \theta^*\|^2 
& = \|\theta^{(k)} - \theta^{(k - 1)} \|^2 + \| 
\theta^{(k - 1)} - \theta^* \|^2 + 2 \langle (\theta^{(k)} - \theta^{(k - 1)}), (\theta^{(k - 1)} - \theta^{*}) \rangle
\\
& \overset{(a)}{=} \|\theta^{(k)} - \theta^{(k - 1)} \|^2 - \| 
\theta^{(k - 1)} - \theta^* \|^2 + 2 \| 
\theta^{(k - 1)} - \theta^* \|^2 
\\
& + 2 \langle (\theta^{(k)} - \theta^{(k - 1)}), (\theta^{(k - 1)} - \theta^{*}) \rangle
\\
& \leq \| \theta^{(k - 1)} - \theta^* \|^2 
- \|\theta^{(k)} - \theta^{(k - 1)} \|^2 
- \alpha_\theta \mu \|\theta^{(k)} - \theta^*\|^2 
\\
& + 2 \alpha_\theta \Big[ \Big \langle L'_\theta(\theta^{(k - 1)}, w^{(k)}) ,\, \theta^{*}  \Big \rangle + g(\theta^{*}) - \Big \langle L'_\theta(\theta^{(k - 1)}, w^{(k)}) ,\, \theta^{(k)}  \Big \rangle - g(\theta^{(k)}) \Big]
\\
& + {2 \alpha_\theta \langle ( \delta_k^\theta- L'_\theta(\theta^{(k - 1)}, w^{(k)}) ),  (\theta^{(k)} - \theta^*) \rangle }
\end{aligned}
\label{e:SVRPDA_I:batch_dual_prox_svrg:initial_bd_primal}
\end{equation}
where step (a) subtracts and adds the second term.
Furthermore, note that $\theta^*$ is the optimal solution to the following optimization problem:
\begin{equation*}
\theta^* = \argmin_\theta \Big[ L(\theta, w^*) + g(\theta) \Big]
\end{equation*}
which implies (from the fact that $g$ is $\mu$-strongly convex):
\begin{equation}
L(\theta^{(k)} , w^*) + g(\theta^{(k)} ) \geq L(\theta^*, w^*) + g(\theta^*) + \frac{\mu}{2} \|\theta^{(k)} - \theta^* \|^2
\label{e:SVRPDA_I:batch_dual_prox_svrg:primal_inequality_2}
\end{equation}
Multiplying both sides of the above inequality by $2\alpha_\theta$ and then adding it to \eqref{e:SVRPDA_I:batch_dual_prox_svrg:initial_bd_primal}, we obtain:
\begin{equation}
\begin{aligned}
(1 + 2 \alpha_\theta \mu)\|\theta^{(k)} - \theta^*\|^2
& \leq \| \theta^{(k - 1)} - \theta^* \|^2 
- \|\theta^{(k)} - \theta^{(k - 1)} \|^2 
\\
& + 2 \alpha_\theta \Big[ L(\theta^{(k)} , w^*) - L(\theta^{*} , w^*) \Big ]
\\
& + 2 \alpha_\theta \Big[ \Big \langle L'_\theta(\theta^{(k - 1)}, w^{(k)}) ,\, \theta^{*}  \Big \rangle - \Big \langle L'_\theta(\theta^{(k - 1)}, w^{(k)}) ,\, \theta^{(k)} \Big \rangle \Big]
\\
& + {2 \alpha_\theta \langle ( \delta_k^\theta- L'_\theta(\theta^{(k - 1)}, w^{(k)}) ),  (\theta^{(k)} - \theta^*) \rangle }
\end{aligned}
\label{e:SVRPDA_I:final_bd_primal}
\end{equation}

Next, we bound the last term in \eqref{e:SVRPDA_I:final_bd_primal}. To this end, we first introduce the following auxiliary variable $\bartheta^{(k)}$, which is the updated primal variable if \emph{full batch gradient} were used: 
\begin{align}
\bartheta^{(k)} 
&= 
\argmin_\theta \Big [ \Big \langle L'_\theta(\theta^{(k - 1)}, w^{(k)}) ,\, \theta  \Big \rangle + {g(\theta)} + \frac{1}{2 \alpha_\theta} \|\theta - \theta^{(k - 1)} \|^2 \Big ]
\nn\\
&=
\prox_{\alpha_\theta g}
\Big\{
\theta^{(k-1)} - \alpha_\theta L'_\theta(\theta^{(k-1)}, w^{(k)})
\Big\}
\label{e:SVRPDA_I:PDA_primal}
\end{align}
Note that both \eqref{e:SVRPDA_I:SPDA_primal_interm1} and \eqref{e:SVRPDA_I:PDA_primal} are written in proximal mapping form. We now bound the last term  \eqref{e:SVRPDA_I:final_bd_primal}:
\begin{align}
&\langle ( \delta_k^\theta- L'_\theta(\theta^{(k - 1)}, w^{(k)}) ),  (\theta^{(k)} - \theta^*) \rangle 
\nn\\
\overset{(a)}{=} & \langle ( \delta_k^\theta- L'_\theta(\theta^{(k - 1)}, w^{(k)}) ),  (\theta^{(k)} - \bartheta^{(k)}) \rangle
+ \langle ( \delta_k^\theta- L'_\theta(\theta^{(k - 1)}, w^{(k)}) ),  (\bartheta^{(k)} - \theta^*) \rangle
\nn\\
\overset{(b)}{\leq} & \| \delta_k^\theta- L'_\theta(\theta^{(k - 1)}, w^{(k)}) \| \cdot \| \theta^{(k)} - \bartheta^{(k)} \|
+ \langle ( \delta_k^\theta- L'_\theta(\theta^{(k - 1)}, w^{(k)}) ),  (\bartheta^{(k)} - \theta^*) \rangle
\nn\\
\overset{(c)}{\leq} & \alpha_\theta \| \delta_k^\theta- L'_\theta(\theta^{(k - 1)}, w^{(k)}) \|^2
+ \langle ( \delta_k^\theta- L'_\theta(\theta^{(k - 1)}, w^{(k)}) ),  (\bartheta^{(k)} - \theta^*) \rangle
\label{e:SVRPDA_I:gradient_error_term}
\end{align}
where step (a) adds and subtracts $\bartheta^{(k)}$, step (b) uses Cauchy-Schwartz inequality, and step (c) substitutes \eqref{e:SVRPDA_I:SPDA_primal_interm1} and \eqref{e:SVRPDA_I:PDA_primal} and then applies \Lemma{t:SVRPDA_I:lemma2}. Let $\clF_k^{(+)}$ denote the filtration of all events up to and including the dual update in the $k$-th iteration. 
Applying expectation to both sides of \eqref{e:SVRPDA_I:gradient_error_term} conditioned on $\clF_k^{(+)}$, we have:
\begin{align}
&\Expect \Big \{ \langle ( \delta_k^\theta- L'_\theta(\theta^{(k - 1)}, w^{(k)}) ),  (\theta^{(k)} - \theta^*) \rangle \big | \clF_k^{(+)} \Big \}
\nn\\
\leq&
\alpha_\theta \Expect \Big \{ \| \delta_k^\theta- L'_\theta(\theta^{(k - 1)}, w^{(k)}) \|^2 \big | \clF_k^{(+)} \Big \}
+
\Expect \Big\{
\langle
\delta_k^\theta
-
L'_\theta(\theta^{(k-1)}, w^{(k)}), \bartheta^{(k)} - \theta^*
\rangle
\Big|
\clF_k^{(+)}
\Big\}
\nn\\
\overset{(a)}{=}&
\alpha_\theta \Expect \Big \{ \| \delta_k^\theta- L'_\theta(\theta^{(k - 1)}, w^{(k)}) \|^2 \big | \clF_k^{(+)} \Big \}
+
\Big \langle
\Expect \big[ \delta_k^\theta
\big|
\clF_k^{(+)}
\big] 
-
L'_\theta(\theta^{(k-1)}, w^{(k)}), \bartheta^{(k)} - \theta^*
\Big \rangle
\nn\\
\overset{(b)}{=}&
\alpha_\theta \Expect \Big \{ \| \delta_k^\theta- L'_\theta(\theta^{(k - 1)}, w^{(k)}) \|^2 \big | \clF_k^{(+)} \Big \}
\label{e:SVRPDA_I:gradient_error_term_final_bd}
\end{align}
where step $(a)$ uses the fact that $\bartheta^{(k)}$, $\theta^{(k-1)}$ and $w^{(k)}$ are deterministic conditioned on $\clF_k^{(+)}$, and step $(b)$ uses the fact that the conditional expectation of $\delta_k^\theta$ is the batch gradient. 

We will now upper bound the right hand side of \eqref{e:SVRPDA_I:gradient_error_term_final_bd}. To this end, we have:
\begin{equation}
\begin{aligned}
& \Expect \Big \{ \| \delta_k^\theta - L'_\theta(\theta^{(k - 1)}, w^{(k)}) \|^2 \big | \clF_k^{(+)} \Big \} 
\\
\overset{(a)}{=} & \Expect \Big \{ \| f'_{i_kj_k} \big(\theta^{(k - 1)}\big) w^{(k)}_{i_k} - f'_{i_kj_k} \big( \tiltheta \big) w^{(k)}_{i_k} + L'_\theta(\tiltheta, w^{(k)}) - L'_\theta(\theta^{(k - 1)}, w^{(k)}) \|^2 \big | \clF_k^{(+)} \Big \} 
\\
\overset{(b)}{\leq} & \Expect \Big \{ \| f'_{i_kj_k} \big(\theta^{(k - 1)}\big) w^{(k)}_{i_k} - f'_{i_kj_k} \big( \tiltheta \big) w^{(k)}_{i_k} \|^2 \big | \clF_k^{(+)} \Big \} 
\\
\overset{(c)}{\leq} & B_w^2 \Expect \Big \{ \| f'_{i_kj_k} \big(\theta^{(k - 1)}\big) - f'_{i_kj_k} \big( \tiltheta \big) \|^2 \big | \clF_k^{(+)} \Big \}
\\
\overset{(d)}{\leq} & B_w^2 B_\theta^2 \| \theta^{(k - 1)} -  \tiltheta \|^2  
\label{e:SVRPDA_I:variance_of_gradient}
\end{aligned}
\end{equation}
where step $(a)$ uses the definition of  $\delta_k^\theta$, step $(b)$ uses $\Expect [X - \Expect[X] ]^2 \leq \Expect[X^2]$, step $(c)$ uses Lemma \ref{lemma:conjugate_Lipschitz}, and step $(d)$ uses the Lipschitz continuity of the gradients (Assumption~\ref{a:f:smooth_bg}). Substituting \eqref{e:SVRPDA_I:variance_of_gradient} into \eqref{e:SVRPDA_I:gradient_error_term_final_bd}, we get:
\begin{align}
&\Expect\Big\{
\langle ( \delta_{i_k' j_k'}^\theta - L'_\theta(\theta^{(k - 1)}, w^{(k)}) ),  (\theta^{(k)} - \theta^*) \rangle \Big | \clF_k^{(+)} \Big\}
\le
\alpha_\theta B_w^2 B_\theta^2 \| \theta^{(k - 1)} -  \tiltheta \|^2 
\label{e:SVRPDA_I:variance_of_gradient_final}
\end{align}

Finally, substituting \eqref{e:SVRPDA_I:variance_of_gradient_final} into \eqref{e:SVRPDA_I:final_bd_primal} and then further applying expectation conditioned on $\clF_k$, we obtain:
\begin{equation}
\begin{aligned}
(1 + 2 \alpha_\theta \mu) \Expect \Big \{ \|\theta^{(k)} - \theta^*\|^2 \big | \clF_k \Big \}
& \leq \| \theta^{(k - 1)} - \theta^* \|^2 
- \Expect \Big \{ \|\theta^{(k)} - \theta^{(k - 1)} \|^2 \big | \clF_k \Big \}
\\
& + 2 \alpha_\theta \Expect \Big \{ \Big[ L(\theta^{(k)} , w^*) - L(\theta^{*} , w^*) \Big ] \big | \clF_k \Big \}
\\
& + 2 \alpha_\theta \Expect \Big \{ \Big[ \Big \langle L'_\theta(\theta^{(k - 1)}, w^{(k)}) ,\, \theta^{*}  \Big \rangle - \Big \langle L'_\theta(\theta^{(k - 1)}, w^{(k)}) ,\, \theta^{(k)} \Big \rangle \Big] \big | \clF_k \Big \}
\\
& + {2 \alpha_\theta^2 B_w^2 B_\theta^2 \| \theta^{(k - 1)} -  \tiltheta \|^2}
\end{aligned}
\label{e:SVRPDA_I:final_bd_primal_final_0}
\end{equation}
Dividing both sides by $2 \alpha_\theta$ and combining common terms, we obtain the final bound for the primal variable:
\begin{equation}
\begin{aligned}
& \bigg (\frac{1}{2 \alpha_\theta} + \mu \bigg ) \Expect \Big \{ \|\theta^{(k)} - \theta^*\|^2 \big | \clF_k \Big \}
+ \frac{1}{2 \alpha_\theta} \Expect \Big \{ \|\theta^{(k)} - \theta^{(k - 1)} \|^2 \big | \clF_k \Big \}
\\
\leq& \frac{1}{2 \alpha_\theta} \| \theta^{(k - 1)} - \theta^* \|^2 
{+ \alpha_\theta B_w^2 B_\theta^2 \| \theta^{(k - 1)} -  \tiltheta \|^2}
+ \Expect \Big \{ \Big[ L(\theta^{(k)} , w^*) - L(\theta^{*} , w^*) \Big ] \big | \clF_k \Big \}
\\
& 
+  \Expect \Big \{ \Big[ \Big \langle L'_\theta(\theta^{(k - 1)}, w^{(k)}) ,\, \theta^{*}  \Big \rangle - \Big \langle L'_\theta(\theta^{(k - 1)}, w^{(k)}) ,\, \theta^{(k)} \Big \rangle \Big] \big | \clF_k \Big \}
\end{aligned}
\label{e:SVRPDA_I:final_bd_primal_final}
\end{equation}
\subsection{Convergence for Option I}

Based on the derived primal and dual bounds above, we now proceed to prove the convergence of \methodexact~with Option I: updating $\tilde{\theta}$ using the most recent $\theta^{(k)}$ (see Algorithm \ref{alg:svrpda1}).

Adding \eqref{e:SVRPDA_I:dual_bd_2} and \eqref{e:SVRPDA_I:final_bd_primal_final} we obtain the total bound for the primal and dual variable updates:
\begin{equation}
\begin{aligned}
& \Bigg( \frac{1}{2 \alpha_w} + \gamma \Bigg) \Expect \{ \|w^{(k)} - w^*\|^2 \mid \clF_k \} + \frac{1}{2 \alpha_w} \Expect \{ \|w^{(k)} - w^{(k - 1)} \|^2 \mid \clF_k \}
\\
& \Bigg (\frac{1}{2 \alpha_\theta} + \mu \Bigg ) \Expect \Big \{ \|\theta^{(k)} - \theta^*\|^2 \big | \clF_k \Big \}
+ \frac{1}{2 \alpha_\theta} \Expect \Big \{ \|\theta^{(k)} - \theta^{(k - 1)} \|^2 \big | \clF_k \Big \}
\\
& \leq \Bigg( \frac{1}{2 \alpha_w} + \frac{\gamma \big(n_X - 1 \big)}{n_X} \Bigg) \| w^{(k - 1)} - w^* \|^2 
+ \frac{1}{2 \alpha_\theta} \| \theta^{(k - 1)} - \theta^* \|^2 \\
& + \Big( 2\alpha_w B_f^2 \Big( 1 - \overline{1/n_{Y}} \Big) + \alpha_\theta B_w^2 B_\theta^2 \Big) \| \theta^{(k - 1)} -  \tiltheta \|^2
\\
& + \Expect \Big \{ L(\theta^{(k - 1)}, w^{(k)} - w^*) - L(\theta^*, w^{(k)} - w^*) \mid \clF_k \Big \}
\\
& + \Expect \Big \{ \Big[ L(\theta^{(k)} , w^*) - L(\theta^{*} , w^*) \Big ] \big | \clF_k \Big \}
+  \Expect \Big \{ \Big[ \Big \langle L'_\theta(\theta^{(k - 1)}, w^{(k)}) ,\, \theta^{*}  \Big \rangle - \Big \langle L'_\theta(\theta^{(k - 1)}, w^{(k)}) ,\, \theta^{(k)} \Big \rangle \Big] \big | \clF_k \Big \}
\\
& + \big(n_X - 1) \Expect \Big \{ L(\theta^{(k - 1)}, w^{(k)} - w^{(k - 1)}) - L(\theta^*, w^{(k)} - w^{(k - 1)}) \mid \clF_k \Big \}
\end{aligned}
\label{e:SVRPDA_I:total_bound_1}
\end{equation}

Next, we need to upper bound the $L$ terms on the right-hand side of the above inequality. To this end, we first show the following inequality:
    \begin{align}
    \label{e:svrpda1:option1:Lterms1}
        &-L(\theta^*, w^{(k)}-w^*) + L(\theta^{(k)},w^*) - L(\theta^*,w^*)
        +
        \big\langle L_\theta'(\theta^{(k-1)}, w^{(k)}), \theta^* \big\rangle
        \nn\\
        &
        -
        \big\langle L_\theta'(\theta^{(k-1)}, w^{(k)}), \theta^{(k)} \big\rangle
        +
        L(\theta^{(k)}, w^{(k)} - w^*)
        \nn\\
                &=
                        -L(\theta^*, w^{(k)}) + L(\theta^{(k)}, w^{(k)})
                        +
                        \big\langle L_\theta'(\theta^{(k-1)}, w^{(k)}), \theta^* - \theta^{(k)}
                        \big\rangle
                        \big\rangle
                        \nn\\
                &\overset{(a)}{\le}
                        -
                        \big\langle
                            L_{\theta}'(\theta^{(k)}, w^{(k)}), \theta^* - \theta^{(k)}
                        \big\rangle
                        +
                        \big\langle L_\theta'(\theta^{(k-1)}, w^{(k)}), \theta^* - \theta^{(k)}
                        \big\rangle
                        \nn\\
                &=
                        \big\langle
                            L_{\theta}'(\theta^{(k)}, w^{(k)})
                            -
                            L_\theta'(\theta^{(k-1)}, w^{(k)})
                            , 
                            \theta^{(k)} - \theta^* 
                        \big\rangle
                        \nn\\
                &\overset{(b)}{=}
                        \frac{1}{n_X} \sum_{i=0}^{n_X-1}
                        \Big\langle
                            \Big(
                                \barf_i'(\theta^{(k)})
                                -
                                \barf_i'(\theta^{(k-1)}
                            \Big) w_i^{(k)}
                            ,
                            \theta^{(k)} - \theta^*
                        \Big\rangle
                        \nn\\
                &\le
                        \bigg|
                            \frac{1}{n_X} \sum_{i=0}^{n_X-1}
                            \Big\langle
                                \Big(
                                    \barf_i'(\theta^{(k)})
                                    -
                                    \barf_i'(\theta^{(k-1)}
                                \Big) w_i^{(k)}
                                ,
                                \theta^{(k)} - \theta^*
                            \Big\rangle
                        \bigg|
                        \nn\\
                &\overset{(c)}{\le}
                        \frac{1}{n_X} \sum_{i=0}^{n_X-1}
                        \bigg|
                            \Big\langle
                                \Big(
                                    \barf_i'(\theta^{(k)})
                                    -
                                    \barf_i'(\theta^{(k-1)}
                                \Big) w_i^{(k)}
                                ,
                                \theta^{(k)} - \theta^*
                            \Big\rangle
                        \bigg|
                        \nn\\
                &\overset{(d)}{\le}
                        \frac{1}{n_X} \sum_{i=0}^{n_X-1}
                        \Big\|
                            \Big(
                                \barf_i'(\theta^{(k)})
                                -
                                \barf_i'(\theta^{(k-1)}
                            \Big) w_i^{(k)}
                        \Big\|
                        \cdot
                        \|
                            \theta^{(k)} - \theta^*
                        \|
                        \nn\\
                &\le
                        \frac{1}{n_X} \sum_{i=0}^{n_X-1}
                        \big\| 
                            \barf_i'(\theta^{(k)})
                            -
                            \barf_i'(\theta^{(k-1)}
                        \big\|
                        \cdot
                        \|
                            w_i^{(k)}
                        \|
                        \cdot
                        \|
                            \theta^{(k)} - \theta^*
                        \|
                        \nn\\
                &\overset{(e)}{=}
                        \frac{1}{n_X} \sum_{i=0}^{n_X-1}
                        \bigg\| 
                            \frac{1}{n_{Y_i}}
                            \sum_{j=0}^{n_{Y_i}-1}
                            \Big(
                                f_{\theta^{(k)}}'(x_i,y_{ij})
                                -
                                f_{\theta^{(k-1)}}'
                                (x_i,y_{ij})
                            \Big)
                        \bigg\|
                        \cdot
                        \|
                            w_i^{(k)}
                        \|
                        \cdot
                        \|
                            \theta^{(k)} - \theta^*
                        \|
                        \nn\\
                &\overset{(f)}{\le}
                        \frac{1}{n_X} \sum_{i=0}^{n_X-1}
                        \frac{1}{n_{Y_i}}
                        \sum_{j=0}^{n_{Y_i}-1}
                        \bigg\|
                            f_{\theta^{(k)}}'(x_i,y_{ij})
                            -
                            f_{\theta^{(k-1)}}'
                            (x_i,y_{ij})
                        \bigg\|
                        \cdot
                        \|
                            w_i^{(k)}
                        \|
                        \cdot
                        \|
                            \theta^{(k)} - \theta^*
                        \|
                        \nn\\
                &\overset{(g)}{\le}
                        \frac{1}{n_X} \sum_{i=0}^{n_X-1}
                        \frac{1}{n_{Y_i}}
                        \sum_{j=0}^{n_{Y_i}-1}
                        B_\theta B_w
                        \| 
                            \theta^{(k)} - \theta^{(k-1)}
                        \|
                        \cdot
                        \|
                            \theta^{(k)} - \theta^*
                        \|
                        \nn\\
                &=
                        B_\theta B_w
                        \| 
                            \theta^{(k)} - \theta^{(k-1)}
                        \|
                        \cdot
                        \|
                            \theta^{(k)} - \theta^*
                        \|
                        \nn\\
                &\overset{(h)}{\le}
                        \frac{B_\theta B_w}{\beta_0}
                        \| 
                            \theta^{(k)} - \theta^{(k-1)}
                        \|^2
                        +
                        \beta_0 B_\theta B_w
                        \|
                            \theta^{(k)} - \theta^*
                        \|^2
    \end{align}
where step (a) uses the fact that $L(\theta,w)$ is convex with respect to the $\theta$, step (b) substitutes the expression of $L_\theta'$, step (c) uses Jensen's inequality, step (d) uses Cauchy-Schwartz inequality, step (e) substitutes the expression for $\barf_i'$, step (f) uses Jensen's inequality, step (g) uses the Lipschitz gradient property of $f_\theta'$ and Lemma \ref{lemma:conjugate_Lipschitz}, step (h) uses $ab \le \frac{1}{\beta_0} a^2 + \beta_0 b^2$. In consequence, the above inequality implies that
    \begin{align}
        &-L(\theta^*, w^{(k)}-w^*) + L(\theta^{(k)},w^*) - L(\theta^*,w^*)
        +
        \big\langle L_\theta'(\theta^{(k-1)}, w^{(k)}), \theta^* \big\rangle
        -
        \big\langle L_\theta'(\theta^{(k-1)}, w^{(k)}), \theta^{(k)} \big\rangle
        \nn\\
                \le&
                        -L(\theta^{(k)}, w^{(k)} - w^*)
                        +
                        \frac{B_\theta B_w}{\beta_0}
                        \| 
                            \theta^{(k)} - \theta^{(k-1)}
                        \|^2
                        +
                        \beta_0 B_\theta B_w
                        \|
                            \theta^{(k)} - \theta^*
                        \|^2
        \label{e:svrpda1:L_partial_bound}
    \end{align}
Using \eqref{e:svrpda1:L_partial_bound}, the $L$ terms in  \eqref{e:SVRPDA_I:total_bound_1} becomes (notice that we keep the first and last $L$ terms in \eqref{e:SVRPDA_I:total_bound_1} intact)
    \begin{align}
        &L(\theta^{(k - 1)}, w^{(k)} - w^*) 
        - 
        L(\theta^*, w^{(k)} - w^*)
        +
        L(\theta^{(k)} , w^*) - L(\theta^{*} , w^*)
        +  
        \Big \langle 
            L'_\theta(\theta^{(k - 1)}, w^{(k)}) ,\, \theta^{*}  
        \Big \rangle 
        \nn\\
        &
        - 
        \Big \langle 
            L'_\theta(\theta^{(k - 1)}, w^{(k)}),\, \theta^{(k)} 
        \Big \rangle
        + 
        \big(n_X - 1) 
        \Big[
            L(\theta^{(k - 1)}, w^{(k)} - w^{(k - 1)}) 
            - 
            L(\theta^*, w^{(k)} - w^{(k - 1)})
        \Big]
        \nn\\
                \overset{(a)}{\le}&
                        L(\theta^{(k-1)}, w^{(k)}-w^*)
                        -
                        L(\theta^{(k)}, w^{(k)} - w^*)
                        +
                        (n_X-1)
                        \Big[
                            L(\theta^{(k-1)},
                            w^{(k)}-w^{(k-1)})
                            -
                            L(\theta^*, w^{(k)}-w^{(k-1)})
                        \Big]
                        \nn\\
                        &
                        +
                        \frac{B_\theta B_w}{\beta_0}
                        \| 
                            \theta^{(k)} - \theta^{(k-1)}
                        \|^2
                        +
                        \beta_0 B_\theta B_w
                        \|
                            \theta^{(k)} - \theta^*
                        \|^2
                        \nn\\
                =&
                        L\Big(\theta^{(k-1)}, w^{(k)} - w^* + (n_X-1)(w^{(k)}-w^{(k-1)}) \Big)
                        -
                        L\Big(\theta^{(k)}, w^{(k)}-w^* + (n_X-1)(w^{(k)} - w^{(k-1)}) \Big)
                        \nn\\
                        &+
                        \frac{B_\theta B_w}{\beta_0}
                        \| 
                            \theta^{(k)} - \theta^{(k-1)}
                        \|^2
                        +
                        \beta_0 B_\theta B_w
                        \|
                            \theta^{(k)} - \theta^*
                        \|^2
                        \nn\\
                \overset{(b)}{\le}&
                        \Big\langle
                            L'_{\theta}
                            \Big(\theta^{(k-1)}, w^{(k)} - w^* + (n_X-1)(w^{(k)}-w^{(k-1)}) \Big),
                            \theta^{(k-1)} - \theta^{(k)}
                        \Big\rangle
                        \nn\\
                        &+
                        \frac{B_\theta B_w}{\beta_0}
                        \| 
                            \theta^{(k)} - \theta^{(k-1)}
                        \|^2
                        +
                        \beta_0 B_\theta B_w
                        \|
                            \theta^{(k)} - \theta^*
                        \|^2
                        \nn\\
                =&
                        \Big\langle
                            \frac{1}{n_X}
                            \sum_{i=0}^{n_X-1}
                            \barf'_i(\theta^{(k-1)})
                            \Big(w_i^{(k)} - w_i^* + (n_X-1)(w_i^{(k)}-w_i^{(k-1)}) \Big)
                            ,
                            \theta^{(k-1)} - \theta^{(k)}
                        \Big\rangle
                        \nn\\
                        &+
                        \frac{B_\theta B_w}{\beta_0}
                        \| 
                            \theta^{(k)} - \theta^{(k-1)}
                        \|^2
                        +
                        \beta_0 B_\theta B_w
                        \|
                            \theta^{(k)} - \theta^*
                        \|^2
                        \nn\\
                =&
                        \Big\langle
                            \frac{1}{n_X}
                            \sum_{i=0}^{n_X-1}
                            \barf'_i(\theta^{(k-1)})
                            \Big(w_i^{(k-1)} + n_X(w_i^{(k)}-w_i^{(k-1)}) - w_i^* \Big)
                            ,
                            \theta^{(k-1)} - \theta^{(k)}
                        \Big\rangle
                        \nn\\
                        &+
                        \frac{B_\theta B_w}{\beta_0}
                        \| 
                            \theta^{(k)} - \theta^{(k-1)}
                        \|^2
                        +
                        \beta_0 B_\theta B_w
                        \|
                            \theta^{(k)} - \theta^*
                        \|^2
                        \nn\\
                \overset{(c)}{\le}&
                        \beta_1
                        \Big\|
                            \frac{1}{n_X}
                            \sum_{i=0}^{n_X-1}
                            \barf'_i(\theta^{(k-1)})
                            \Big(w_i^{(k-1)} + n_X(w_i^{(k)}-w_i^{(k-1)}) - w_i^* \Big)
                        \Big\|^2
                        +
                        \frac{1}{\beta_1}
                        \big\|
                            \theta^{(k-1)} - \theta^{(k)}
                        \big\|^2
                        \nn\\
                        &+
                        \frac{B_\theta B_w}{\beta_0}
                        \| 
                            \theta^{(k)} - \theta^{(k-1)}
                        \|^2
                        +
                        \beta_0 B_\theta B_w
                        \|
                            \theta^{(k)} - \theta^*
                        \|^2
                        \nn\\
                =&
                        \beta_1
                        \Big\|
                            \frac{1}{n_X}
                            \sum_{i=0}^{n_X-1}
                            \barf'_i(\theta^{(k-1)})
                            \Big(w_i^{(k-1)} - w_i^*\Big) 
                            + 
                            \sum_{i=0}^{n_X-1}
                            \barf'_i(\theta^{(k-1)})
                            (w_i^{(k)}-w_i^{(k-1)})
                        \Big\|^2
                        \nn\\
                        &+
                        \frac{1}{\beta_1}
                        \big\|
                            \theta^{(k-1)} - \theta^{(k)}
                        \big\|^2
                        +
                        \frac{B_\theta B_w}{\beta_0}
                        \| 
                            \theta^{(k)} - \theta^{(k-1)}
                        \|^2
                        +
                        \beta_0 B_\theta B_w
                        \|
                            \theta^{(k)} - \theta^*
                        \|^2
                        \nn\\
                \overset{(d)}{\le}&
                        2\beta_1
                        \Big\|
                            \frac{1}{n_X}
                            \sum_{i=0}^{n_X-1}
                            \barf'_i(\theta^{(k-1)})
                            \Big(w_i^{(k-1)} - w_i^*\Big) 
                        \Big\|^2
                        + 
                        2\beta_1
                        \Big\|
                            \sum_{i=0}^{n_X-1}
                            \barf'_i(\theta^{(k-1)})
                            (w_i^{(k)}-w_i^{(k-1)})
                        \Big\|^2
                        \nn\\
                        &+
                        \frac{1}{\beta_1}
                        \big\|
                            \theta^{(k-1)} - \theta^{(k)}
                        \big\|^2
                        +
                        \frac{B_\theta B_w}{\beta_0}
                        \| 
                            \theta^{(k)} - \theta^{(k-1)}
                        \|^2
                        +
                        \beta_0 B_\theta B_w
                        \|
                            \theta^{(k)} - \theta^*
                        \|^2
                        \nn\\
                \overset{(e)}{\le}&
                        \frac{2\beta_1 B_f^2}{n_X}
                        \sum_{i=0}^{n_X-1}
                        \| w_i^{(k-1)} - w_i^* \|^2
                        + 
                        2\beta_1
                        \Big\|
                            \sum_{i=0}^{n_X-1}
                            \barf'_i(\theta^{(k-1)})
                            (w_i^{(k)}-w_i^{(k-1)})
                        \Big\|^2
                        \nn\\
                        &
                        +
                        \Big(
                            \frac{1}{\beta_1}
                            +
                            \frac{B_\theta B_w}{\beta_0}
                        \Big)
                        \big\|
                            \theta^{(k-1)} - \theta^{(k)}
                        \big\|^2
                        +
                        \beta_0 B_\theta B_w
                        \|
                            \theta^{(k)} - \theta^*
                        \|^2
                        \label{e:svrpda1:option1:Lterms3}
    \end{align}
where step (a) applies \eqref{e:svrpda1:L_partial_bound}, step (b) uses convexity of $L$ in $\theta$, step (c) uses $\langle a, b\rangle \le \beta_1 \|a\|^2 + \frac{1}{\beta_1}\|b\|^2$, for some $\beta_1>0$ to be chosen later, step (d) uses $\|a+b\|^2 \le 2\|a\|^2 + 2\|b\|^2$, and step (e) applies Jensen's inequality and bounded gradient assumption to the first term. Before we proceed, we note that, by taking expectation of the second term conditioned on $\clF_k$, we get
    \begin{align}
        &\Expect\Big\{
            \Big\|
                \sum_{i=0}^{n_X-1}
                \barf'_i(\theta^{(k-1)})
                (w_i^{(k)}-w_i^{(k-1)})
            \Big\|^2
            \Big| \clF_k
        \Big\}
            \nn\\
                =&
                        \sum_{i=0}^{n_X-1}
                        \sum_{j=0}^{n_{Y_i}-1}
                        \frac{1}{n_X n_{Y_i}}
                        \Big\|
                            \barf'_i(\theta^{(k-1)})
                            (w_{ij}'-w_i^{(k-1)})
                        \Big\|^2    
                        \nn\\
                \le&
                        B_f^2
                        \sum_{i=0}^{n_X-1}
                        \sum_{j=0}^{n_{Y_i}-1}
                        \frac{1}{n_X n_{Y_i}}
                        \Big\|
                            w_{ij}'-w_i^{(k-1)}
                        \Big\|^2    
                        \nn\\
                =&
                        B_f^2
                        \sum_{i=0}^{n_X-1}
                        \Expect \big\{ \|w_i^{(k)} - w_i^{(k-1)}\|^2 \big| \clF_k \big\}
	\label{e:svrpda1:option1:Lterms4}
    \end{align}
Therefore, the conditional expectation of all the $L$ terms are bounded by
    \begin{align}
    \label{e:SVRPDA_I:option1:Lterms_final}
        &\frac{2\beta_1 B_f^2}{n_X}
        \sum_{i=0}^{n_X-1}
        \| w_i^{(k-1)} - w_i^* \|^2
        + 
        2\beta_1 B_f^2
        \sum_{i=0}^{n_X-1}
        \Expect\{ \|w_i^{(k)} - w_i^{(k-1)}\|^2 | \clF_k \}
        \nn\\
        &
        +
        \Big( 
            \frac{1}{\beta_1} + \frac{B_\theta B_w}{\beta}
        \Big)
        \Expect\big\{
        \big\|
            \theta^{(k-1)} - \theta^{(k)}
        \big\|^2
        \big| \clF_k
        \big\}
        +
        \beta_0 B_\theta B_w
        \Expect\big\{
        \|
            \theta^{(k)} - \theta^*
        \|^2 \big| \clF_k \big\}
        \nn\\
                =&
                        \frac{2\beta_1 B_f^2}{n_X}
                        \| w^{(k-1)} - w^* \|^2
                        + 
                        2\beta_1 B_f^2
                        \Expect\{ \|w^{(k)} - w^{(k-1)}\|^2 | \clF_k \}
                        \nn\\
                        &
                        +
                        \Big(
                            \frac{1}{\beta_1}
                            +
                            \frac{B_\theta B_w}{\beta_0}
                        \Big)
                        \Expect\big\{
                        \big\|
                            \theta^{(k-1)} - \theta^{(k)}
                        \big\|^2
                        \big| \clF_k
                        \big\}
                        +
                        \beta_0 B_\theta B_w
                        \Expect\big\{
                        \|
                            \theta^{(k)} - \theta^*
                        \|^2 \big| \clF_k \big\}
    \end{align}
Therefore, the total bound \eqref{e:SVRPDA_I:total_bound_1} becomes
    \begin{align}
        & \Bigg( \frac{1}{2 \alpha_w} + \gamma \Bigg) \Expect \{ \|w^{(k)} - w^*\|^2 \mid \clF_k \} + \frac{1}{2 \alpha_w} \Expect \{ \|w^{(k)} - w^{(k - 1)} \|^2 \mid \clF_k \}
        \nn\\
        & 
        +
        \Bigg (\frac{1}{2 \alpha_\theta} + \mu \Bigg ) \Expect \Big \{ \|\theta^{(k)} - \theta^*\|^2 \big | \clF_k \Big \}
        + \frac{1}{2 \alpha_\theta} \Expect \Big \{ \|\theta^{(k)} - \theta^{(k - 1)} \|^2 \big | \clF_k \Big \}
        \nn\\
        & \leq \Bigg( \frac{1}{2 \alpha_w} + \frac{\gamma \big(n_X - 1 \big)}{n_X} \Bigg) \| w^{(k - 1)} - w^* \|^2 
        + \frac{1}{2 \alpha_\theta} \| \theta^{(k - 1)} - \theta^* \|^2 
        \nn\\
        & + \Big( 2\alpha_w B_f^2 \Big( 1 - \overline{1/n_{Y}} \Big) + \alpha_\theta B_w^2 B_\theta^2 \Big) \| \theta^{(k - 1)} -  \tiltheta \|^2
        \nn\\
        & 
        + 
                        \frac{2\beta_1 B_f^2}{n_X}
                        \| w^{(k-1)} - w^* \|^2
                        + 
                        2\beta_1 B_f^2
                        \Expect\{ \|w^{(k)} - w^{(k-1)}\|^2 | \clF_k \}
                        \nn\\
                        &
                        +
                        \Big(
                            \frac{1}{\beta_1}
                            +
                            \frac{B_\theta B_w}{\beta_0}
                        \Big)
                        \Expect\big\{
                        \big\|
                            \theta^{(k-1)} - \theta^{(k)}
                        \big\|^2
                        \big| \clF_k
                        \big\}
                        +
                        \beta_0 B_\theta B_w
                        \Expect\big\{
                        \|
                            \theta^{(k)} - \theta^*
                        \|^2 \big| \clF_k \big\}
    \end{align}
By combining the common terms, we obtain
    \begin{align}
        & 
        \Bigg (
            \frac{1}{2 \alpha_\theta} + \mu 
            -
            \beta_0 B_\theta B_w
        \Bigg ) 
        \Expect \Big \{ 
            \|\theta^{(k)} - \theta^*\|^2 \big 
            | \clF_k 
        \Big \}
        +
        \Bigg( \frac{1}{2 \alpha_w} + \gamma \Bigg) 
        \Expect \{ \|w^{(k)} - w^*\|^2 \mid \clF_k \}
        \nn\\
        & 
        + 
        \bigg(
            \frac{1}{2 \alpha_\theta} 
            -
            \frac{1}{\beta_1}
            -
            \frac{B_\theta B_w}{\beta_0}
        \bigg)
        \Expect \Big \{ 
            \|\theta^{(k)} - \theta^{(k - 1)} \|^2 
            \big | \clF_k 
        \Big \}
        +
        \Big( \frac{1}{2 \alpha_w} - 2\beta_1 B_f^2 \Big)
        \Expect \{ \|w^{(k)} - w^{(k - 1)} \|^2 \mid \clF_k \}
        \nn\\
        \leq&
                        \frac{1}{2 \alpha_\theta} \| \theta^{(k - 1)} - \theta^* \|^2 
                        +
                        \bigg( 
                            \frac{1}{2 \alpha_w} 
                            + 
                            \gamma
                            -
                            \frac{\gamma}{n_X}
                            +
                            \frac{2\beta_1 B_f^2}{n_X}
                        \bigg) 
                        \| w^{(k - 1)} - w^* \|^2 
                        \nn\\
                        & 
                        + \Big( 2\alpha_w B_f^2 \Big( 1 - \overline{1/n_{Y}} \Big) + \alpha_\theta B_w^2 B_\theta^2 \Big) \| \theta^{(k - 1)} -  \tiltheta \|^2
        \label{e:svrpda1:total_bound_final_interm1}
    \end{align}
Applying inequality $\|x + y \|^2 \le 2 \|x\|^2 + 2\|y\|^2$ to the last term in \eqref{e:svrpda1:total_bound_final_interm1}, we obtain
    \begin{align}
        & 
        \Bigg (
            \frac{1}{2 \alpha_\theta} + \mu 
            -
            \beta_0 B_\theta B_w
        \Bigg ) 
        \Expect \Big \{ 
            \|\theta^{(k)} - \theta^*\|^2 \big 
            | \clF_k 
        \Big \}
        +
        \Bigg( \frac{1}{2 \alpha_w} + \gamma \Bigg) 
        \Expect \{ \|w^{(k)} - w^*\|^2 \mid \clF_k \}
        \nn\\
        & 
        + 
        \bigg(
            \frac{1}{2 \alpha_\theta} 
            -
            \frac{1}{\beta_1}
            -
            \frac{B_\theta B_w}{\beta_0}
        \bigg)
        \Expect \Big \{ 
            \|\theta^{(k)} - \theta^{(k - 1)} \|^2 
            \big | \clF_k 
        \Big \}
        +
        \Big( \frac{1}{2 \alpha_w} - 2\beta_1 B_f^2 \Big)
        \Expect \{ \|w^{(k)} - w^{(k - 1)} \|^2 \mid \clF_k \}
        \nn\\
        \leq&
                        \bigg(
                            \frac{1}{2 \alpha_\theta} 
                            +
                            4\alpha_w B_f^2 \Big( 1 - \overline{1/n_{Y}} \Big) 
                            + 
                            2\alpha_\theta B_w^2 B_\theta^2
                        \bigg)
                        \| \theta^{(k - 1)} - \theta^* \|^2 
                        \nn\\
                        &
                        +
                        \bigg( 
                            \frac{1}{2 \alpha_w} 
                            + 
                            \gamma
                            -
                            \frac{\gamma}{n_X}
                            +
                            \frac{2\beta_1 B_f^2}{n_X}
                        \bigg) 
                        \| w^{(k - 1)} - w^* \|^2
                        + 
                        \Big( 4\alpha_w B_f^2 \Big( 1 - \overline{1/n_{Y}} \Big) 
                        + 
                        2\alpha_\theta B_w^2 B_\theta^2 \Big) 
                        \| \tiltheta - \theta^* \|^2
        \label{e:svrpda1:total_bound_final_interm2}
    \end{align}
Taking full expectation of the above inequality, we obtain:
    \begin{align}
        & 
        \bigg (
            \frac{1}{2 \alpha_\theta} + \mu 
            -
            \beta_0 B_\theta B_w
        \bigg ) 
        \Expect \|\theta^{(k)} - \theta^*\|^2
        +
        \bigg( \frac{1}{2 \alpha_w} + \gamma \bigg) 
        \Expect \|w^{(k)} - w^*\|^2
        \nn\\
        & 
        + 
        \bigg(
            \frac{1}{2 \alpha_\theta} 
            -
            \frac{1}{\beta_1}
            -
            \frac{B_\theta B_w}{\beta_0}
        \bigg)
        \Expect \|\theta^{(k)} - \theta^{(k - 1)} \|^2 
        +
        \Big( \frac{1}{2 \alpha_w} - 2\beta_1 B_f^2 \Big)
        \Expect \|w^{(k)} - w^{(k - 1)} \|^2
        \nn\\
        \leq&
                        \bigg(
                            \frac{1}{2 \alpha_\theta} 
                            +
                            4\alpha_w B_f^2 \Big( 1 - \overline{1/n_{Y}} \Big) 
                            + 
                            2\alpha_\theta B_w^2 B_\theta^2
                        \bigg)
                        \Expect \| \theta^{(k - 1)} - \theta^* \|^2 
                        \nn\\
                        &
                        +
                        \bigg( 
                            \frac{1}{2 \alpha_w} 
                            + 
                            \gamma
                            -
                            \frac{\gamma}{n_X}
                            +
                            \frac{2\beta_1 B_f^2}{n_X}
                        \bigg) 
                        \Expect \| w^{(k - 1)} - w^* \|^2
                        \nn\\
                        &
                        + 
                        \Big( 4\alpha_w B_f^2 \Big( 1 - \overline{1/n_{Y}} \Big) 
                        + 
                        2\alpha_\theta B_w^2 B_\theta^2 \Big) 
                        \Expect \| \tiltheta - \theta^* \|^2
        \label{e:svrpda1:total_bound_final_interm3}
    \end{align}
In order for the above inequality to converge, the hyperparameters need to be chosen to satisfy the following conditions:
    \begin{align}
        \frac{1}{2\alpha_\theta}
                &\ge
                        \frac{1}{\beta_1} + \frac{B_{\theta}B_w}{\beta_0}
                        \nn\\
        \alpha_w
                &\le
                        \frac{1}{4B_f^2 \beta_1}
                        \nn\\
        \beta_1
                &<
                        \frac{\gamma}{2B_f^2}
                        \nn\\
        4\alpha_w B_f^2 (1-\overline{1/n_{Y}}) + 2\alpha_\theta B_w^2 B_\theta^2
                &<
                        \mu - \beta_0 B_{\theta} B_w
        \label{e:svrpda1:param_condition_interm1}
    \end{align}
which simplifies the recursion to be
    \begin{align}
        & 
        \bigg (
            \frac{1}{2 \alpha_\theta} + \mu 
            -
            \beta_0 B_\theta B_w
        \bigg ) 
        \Expect \|\theta^{(k)} - \theta^*\|^2
        +
        \bigg( \frac{1}{2 \alpha_w} + \gamma \bigg) 
        \Expect \|w^{(k)} - w^*\|^2
        \nn\\
        \leq&
                        \bigg(
                            \frac{1}{2 \alpha_\theta} 
                            +
                            4\alpha_w B_f^2 \Big( 1 - \overline{1/n_{Y}} \Big) 
                            + 
                            2\alpha_\theta B_w^2 B_\theta^2
                        \bigg)
                        \Expect \| \theta^{(k - 1)} - \theta^* \|^2 
                        \nn\\
                        &
                        +
                        \bigg( 
                            \frac{1}{2 \alpha_w} 
                            + 
                            \gamma
                            -
                            \frac{\gamma}{n_X}
                            +
                            \frac{2\beta_1 B_f^2}{n_X}
                        \bigg) 
                        \Expect \| w^{(k - 1)} - w^* \|^2
                        \nn\\
                        &
                        + 
                        \Big( 
                            4\alpha_w B_f^2 \Big( 1 - \overline{1/n_{Y}} \Big) 
                            + 
                            2\alpha_\theta B_w^2 B_\theta^2 
                        \Big) 
                        \Expect \| \tiltheta - \theta^* \|^2
        \label{e:svrpda1:total_bound_final_interm4}
    \end{align}
Inequality \eqref{e:svrpda1:total_bound_final_interm4} can also be further written as
    \begin{align}
        & 
        \Expect \|\theta^{(k)} - \theta^*\|^2
        +
        \frac{
            \frac{1}{2 \alpha_w} + \gamma
        }
        {
            \frac{1}{2 \alpha_\theta} + \mu 
            -
            \beta_0 B_\theta B_w
        }
        \Expect \|w^{(k)} - w^*\|^2
        \nn\\
        \leq&
                        r_P
                        \Expect \| \theta^{(k - 1)} - \theta^* \|^2
                        +
                        r_D
                        \cdot
                        \frac{
                            \frac{1}{2 \alpha_w} + \gamma
                        }
                        {
                            \frac{1}{2 \alpha_\theta} + \mu 
                            -
                            \beta_0 B_\theta B_w
                        }
                        \Expect \| w^{(k - 1)} - w^* \|^2
                        \nn\\
                        &
                        + 
                        \frac{
                            4\alpha_w B_f^2 \big( 1 - \overline{1/n_{Y}} \big) 
                            + 
                            2\alpha_\theta B_w^2 B_\theta^2 
                        }
                        {
                            \frac{1}{2 \alpha_\theta} + \mu 
                            -
                            \beta_0 B_\theta B_w
                        }  
                        \Expect \| \tiltheta - \theta^* \|^2
        \label{e:svrpda1:total_bound_final_interm5}
    \end{align}
where $r_P$ and $r_D$ are the primal and the dual ratios, defined as
    \begin{align}
        r_P
                &=
                        \frac{
                            \frac{1}{2 \alpha_\theta} 
                            +
                            4\alpha_w B_f^2 \Big( 1 - \overline{1/n_{Y}} \Big) 
                            + 
                            2\alpha_\theta B_w^2 B_\theta^2
                        }
                        {
                            \frac{1}{2 \alpha_\theta} + \mu 
                            -
                            \beta_0 B_\theta B_w
                        }
                        \nn\\
        r_D
                &=
                        1 
                        - 
                        \frac{1}{n_X} 
                        \frac{2\alpha_w(\gamma - 2\beta_1 B_f^2)}{1+2\alpha_w\gamma}
                        \nn
    \end{align}

We choose $\beta_0$, $\beta_1$, and the primal and the dual step-sizes to be
    \begin{align}
        \beta_0
                &=
                        \frac{\mu}{2B_\theta B_w},
                        \quad
        \beta_1
                =
                        \frac{\gamma}{4B_f^2}
                        \nn\\
        \alpha_{\theta}
                &=
                        \frac{
                            \frac{1}{\mu}
                        }
                        {
                            64 n_X
                            \big(
                                \frac{B_f^2}{\mu\gamma}
                                +
                                \frac{B_\theta^2 B_w^2}{\mu^2}
                            \big)
                            +
                            n_X
                        }
                =
                        \frac{1}{\mu}
                        \cdot
                        \frac{1}{64 n_X \kappa + n_X}
                        \nn\\
        \alpha_w
                &=
                        \frac{
                            \frac{1}{\gamma}
                        }
                        {
                            64
                            \big(
                                \frac{B_f^2}{\mu\gamma}
                                +
                                \frac{B_\theta^2 B_w^2}{\mu^2}
                            \big)
                            + 1
                        }
                =
                        \frac{1}{\gamma}
                        \cdot
                        \frac{1}{64\kappa + 1}
                        \nn
    \end{align}
where $\kappa$ is the condition number defined as
    \begin{align}
        \kappa
                &=
                        \frac{B_f^2}{\mu\gamma}
                        +
                        \frac{B_\theta^2 B_w^2}{\mu^2}
    \end{align}
It can be verified that the above choice of step-sizes satisfies the condition \eqref{e:svrpda1:param_condition_interm1}. With our choice of the parameters, we also have
    \begin{align}
        \frac{
            \frac{1}{2 \alpha_w} + \gamma
        }
        {
            \frac{1}{2 \alpha_\theta} + \mu 
            -
            \beta_0 B_\theta B_w
        }
                &=
                        \frac{\frac{1}{2 \alpha_w} + \gamma}
                        {\frac{1}{2 \alpha_\theta} + \frac{\mu}{2}}
                =
                        \frac{\gamma}{\mu}
                        \cdot
                        \frac{64\kappa + 3}{64n_X \kappa + n_X + 1}
        \label{e:svrpda1:lyapunov_weight}
    \end{align}
and
    \begin{align}
        \frac{
            4\alpha_w B_f^2 \big( 1 - \overline{1/n_{Y}} \big) 
            + 
            2\alpha_\theta B_w^2 B_\theta^2 
        }
        {
            \frac{1}{2 \alpha_\theta} + \mu 
            -
            \beta_0 B_\theta B_w
        }  
                &=
                        \frac{
                            4\alpha_w B_f^2 \big( 1 - \overline{1/n_{Y}} \big) 
                            + 
                            2\alpha_\theta B_w^2 B_\theta^2 
                        }
                        {
                            \frac{1}{2 \alpha_\theta} + \frac{\mu}{2}
                        }  
                        \nn\\
                &=
                        \frac{
                            \frac{8 B_f^2 \big( 1 - \overline{1/n_{Y}} \big)}{\mu\gamma} 
                            + 
                            \frac{4 B_w^2 B_\theta^2}{\mu^2} \cdot \frac{1}{n_X}
                        }
                        {
                            (64\kappa + 1)(64 n_X \kappa + n_X + 1)
                        }
        \label{e:svrpda1:steady_state_mse}
    \end{align}
Substituting \eqref{e:svrpda1:lyapunov_weight}--\eqref{e:svrpda1:steady_state_mse} into \eqref{e:svrpda1:total_bound_final_interm5}, we obtain
    \begin{align}
        & 
        \Expect \|\theta^{(k)} - \theta^*\|^2
        +
        \frac{\gamma}{\mu}
        \cdot
        \frac{64\kappa + 3}{64n_X \kappa + n_X + 1}
        \Expect \|w^{(k)} - w^*\|^2
        \nn\\
        \leq&
                        r_P
                        \Expect \| \theta^{(k - 1)} - \theta^* \|^2
                        +
                        r_D
                        \cdot
                        \frac{\gamma}{\mu}
                        \cdot
                        \frac{64\kappa + 3}{64n_X \kappa + n_X + 1}
                        \Expect \| w^{(k - 1)} - w^* \|^2
                        \nn\\
                        &
                        + 
                        \frac{
                            \frac{8 B_f^2 \big( 1 - \overline{1/n_{Y}} \big)}{\mu\gamma} 
                            + 
                            \frac{4 B_w^2 B_\theta^2}{\mu^2} \cdot \frac{1}{n_X}
                        }
                        {
                            (64\kappa + 1)(64 n_X \kappa + n_X + 1)
                        }
                        \Expect \| \tiltheta - \theta^* \|^2
        \label{e:svrpda1:total_bound_final_interm6}
    \end{align}

Furthermore, the primal and the dual ratios can be upper bounded as
    \begin{align}
        r_P
                &=
                        \frac{
                            1
                            +
                            8\alpha_\theta\alpha_w B_f^2 \Big( 1 - \overline{1/n_{Y}} \Big) 
                            + 
                            4\alpha_\theta^2 B_w^2 B_\theta^2
                        }
                        {
                            1 
                            + 
                            2\alpha_\theta\mu 
                            -
                            2\alpha_\theta
                            \beta_0 B_\theta B_w
                        }
                        \nn\\
                &=
                        \frac{
                            1
                            +
                            8\alpha_\theta\alpha_w B_f^2 \Big( 1 - \overline{1/n_{Y}} \Big) 
                            + 
                            4\alpha_\theta^2 B_w^2 B_\theta^2
                        }
                        {
                            1 
                            + 
                            \alpha_\theta\mu
                        }
                        \nn\\
                &=
                        1
                        -
                        \frac{
                            \alpha_\theta \mu
                            -
                            8\alpha_\theta\alpha_w B_f^2 \Big( 1 - \overline{1/n_{Y}} \Big) 
                            - 
                            4\alpha_\theta^2 B_w^2 B_\theta^2
                        }
                        {
                            1 
                            + 
                            \alpha_\theta\mu
                        }
                        \nn\\
                &\le
                        1
                        -
                        \frac{1}{\frac{1024}{13}n_X \kappa + \frac{16}{13} n_X + \frac{16}{13}}
                        \nn\\
                &\le
                        1
                        -
                        \frac{1}{78.8n_X \kappa + 1.3 n_X + 1.3}
        \label{e:svrpda1:r_P}
                        \\
        r_D
                &=
                        1 
                        - 
                        \frac{1}{n_X} 
                        \frac{2\alpha_w(\gamma - 2\beta_1 B_f^2)}{1+2\alpha_w\gamma}
                        \nn\\
                &=
                        1 
                        - 
                        \frac{1}{n_X} 
                        \frac{\alpha_w\gamma}{1+2\alpha_w\gamma}
                        \nn\\
                &=
                        1 - \frac{1}{64n_X \kappa + 3 n_X}
                <
                        r_P
        \label{e:svrpda1:r_D}
    \end{align}
Therefore, inequality \eqref{e:svrpda1:total_bound_final_interm6} can be further upper bounded as
    \begin{align}
        & 
        \Expect \|\theta^{(k)} - \theta^*\|^2
        +
        \frac{\gamma}{\mu}
        \cdot
        \frac{64\kappa + 3}{64n_X \kappa + n_X + 1}
        \Expect \|w^{(k)} - w^*\|^2
        \nn\\
        \leq&
                        r_P
                        \Big(
                            \Expect \| \theta^{(k - 1)} - \theta^* \|^2
                            +
                            \frac{\gamma}{\mu}
                            \cdot
                            \frac{64\kappa + 3}{64n_X \kappa + n_X + 1}
                            \Expect \| w^{(k - 1)} - w^* \|^2
                        \Big)
                        \nn\\
                        &
                        + 
                        \frac{
                            \frac{8 B_f^2 \big( 1 - \overline{1/n_{Y}} \big)}{\mu\gamma} 
                            + 
                            \frac{4 B_w^2 B_\theta^2}{\mu^2} \cdot \frac{1}{n_X}
                        }
                        {
                            (64\kappa + 1)(64 n_X \kappa + n_X + 1)
                        }
                        \Expect \| \tiltheta - \theta^* \|^2
                        \nn\\
                \le&    
                        r_P
                        \Big(
                            \Expect \| \theta^{(k - 1)} - \theta^* \|^2
                            +
                            \frac{\gamma}{\mu}
                            \cdot
                            \frac{64\kappa + 3}{64n_X \kappa + n_X + 1}
                            \Expect \| w^{(k - 1)} - w^* \|^2
                        \Big)
                        \nn\\
                        &
                        + 
                        \frac{
                            \frac{8 B_f^2 \big( 1 - \overline{1/n_{Y}} \big)}{\mu\gamma} 
                            + 
                            \frac{4 B_w^2 B_\theta^2}{\mu^2} \cdot \frac{1}{n_X}
                        }
                        {
                            (64\kappa + 1)(64 n_X \kappa + n_X + 1)
                        }
                        \Big(
                            \Expect \| \tiltheta - \theta^* \|^2
                            +
                            \frac{\gamma}{\mu}
                            \cdot
                            \frac{64\kappa + 3}{64n_X \kappa + n_X + 1}
                            \Expect \| \tilde{w} - w^* \|^2
                        \Big)
                        \nn\\
                \overset{(a)}{=}&
                        r_P
                        \Big(
                            \Expect \| \theta^{(k - 1)} - \theta^* \|^2
                            +
                            \frac{\gamma}{\mu}
                            \cdot
                            \frac{64\kappa + 3}{64n_X \kappa + n_X + 1}
                            \Expect \| w^{(k - 1)} - w^* \|^2
                        \Big)
                        \nn\\
                        &
                        + 
                        \frac{
                            \frac{8 B_f^2 \big( 1 - \overline{1/n_{Y}} \big)}{\mu\gamma} 
                            + 
                            \frac{4 B_w^2 B_\theta^2}{\mu^2} \cdot \frac{1}{n_X}
                        }
                        {
                            (64\kappa + 1)(64 n_X \kappa + n_X + 1)
                        }
                        \Big(
                            \Expect \| \tiltheta_{s-1} - \theta^* \|^2
                            +
                            \frac{\gamma}{\mu}
                            \cdot
                            \frac{64\kappa + 3}{64n_X \kappa + n_X + 1}
                            \Expect \| \tilde{w}_{s-1} - w^* \|^2
                        \Big)
        \label{e:svrpda1:total_bound_final_interm7}
    \end{align}
where step (a) uses the fact that the $\tiltheta = \tiltheta_{s-1}$ and $\tilde{w} = \tilde{w}_{s-1}$ when we are considering the $s$-th stage/outer-loop (see Algorithm \ref{alg:svrpda1} in the main paper). Define the following Lyapunov functions:
    \begin{align}
        P_{s,k}
                &=
                        \Expect \|\theta^{(k)} - \theta^*\|^2
                        +
                        \frac{\gamma}{\mu}
                        \cdot
                        \frac{64\kappa + 3}{64n_X \kappa + n_X + 1}
                        \Expect \|w^{(k)} - w^*\|^2
                        \nn\\
        P_{s}
                &=
                        \Expect \|\tiltheta_s - \theta^*\|^2
                        +
                        \frac{\gamma}{\mu}
                        \cdot
                        \frac{64\kappa + 3}{64n_X \kappa + n_X + 1}
                        \Expect \|\tilde{w}_s - w^*\|^2
                        \nn
    \end{align}
As a result, we can rewrite inequality \eqref{e:svrpda1:total_bound_final_interm7} as
    \begin{align}
        P_{s,k}
                &\le
                        r_P \cdot P_{s,k-1}
                        +
                        \frac{
                            \frac{8 B_f^2 \big( 1 - \overline{1/n_{Y}} \big)}{\mu\gamma} 
                            + 
                            \frac{4 B_w^2 B_\theta^2}{\mu^2} \cdot \frac{1}{n_X}
                        }
                        {
                            (64\kappa + 1)(64 n_X \kappa + n_X + 1)
                        }
                        P_{s-1}
                        \nn\\
                &\le
                        \bigg(
                            1
                            -
                            \frac{1}{78.8n_X \kappa + 1.3 n_X + 1.3}
                        \bigg)
                        \cdot P_{s,k-1}
                        +
                        \frac{
                            \frac{8 B_f^2 \big( 1 - \overline{1/n_{Y}} \big)}{\mu\gamma} 
                            + 
                            \frac{4 B_w^2 B_\theta^2}{\mu^2} \cdot \frac{1}{n_X}
                        }
                        {
                            (64\kappa + 1)(64 n_X \kappa + n_X + 1)
                        }
                        P_{s-1}
    \end{align}
Furthermore, at the $s$-th stage (outer loop iteration), when Option I is used in Algorithm \ref{alg:svrpda1} in the main paper, we have $\tiltheta_s = \theta^{(M)}$ and $\tilde{w}_s = w^{(M)}$. Therefore, it holds that
    \begin{align}
        P_s
                &=
                        P_{s,M}
                        \nn\\
                &\overset{(a)}{\le}
                        \bigg(
                            1
                            -
                            \frac{1}{78.8n_X \kappa + 1.3 n_X + 1.3}
                        \bigg)^M
                        P_{s,0}
                        \nn\\
                        &\quad
                        +
                        \sum_{k=0}^{M-1}
                        \bigg(
                            1
                            -
                            \frac{1}{78.8n_X \kappa + 1.3 n_X + 1.3}
                        \bigg)^k
                        \times
                        \frac{
                            \frac{8 B_f^2 \big( 1 - \overline{1/n_{Y}} \big)}{\mu\gamma} 
                            + 
                            \frac{4 B_w^2 B_\theta^2}{\mu^2} \cdot \frac{1}{n_X}
                        }
                        {
                            (64\kappa + 1)(64 n_X \kappa + n_X + 1)
                        }
                        P_{s-1}
                        \nn\\
                &\le
                        \bigg(
                            1
                            -
                            \frac{1}{78.8n_X \kappa + 1.3 n_X + 1.3}
                        \bigg)^M
                        P_{s,0}
                        \nn\\
                        &\quad
                        +
                        \sum_{k=0}^{+\infty}
                        \bigg(
                            1
                            -
                            \frac{1}{78.8n_X \kappa + 1.3 n_X + 1.3}
                        \bigg)^k
                        \times
                        \frac{
                            \frac{8 B_f^2 \big( 1 - \overline{1/n_{Y}} \big)}{\mu\gamma} 
                            + 
                            \frac{4 B_w^2 B_\theta^2}{\mu^2} \cdot \frac{1}{n_X}
                        }
                        {
                            (64\kappa + 1)(64 n_X \kappa + n_X + 1)
                        }
                        P_{s-1}
                        \nn\\
                &=
                        \bigg(
                            1
                            -
                            \frac{1}{78.8n_X \kappa + 1.3 n_X + 1.3}
                        \bigg)^M
                        P_{s,0}
                        \nn\\
                        &\quad
                        +
                        \bigg(
                            \frac{8 B_f^2 \big( 1 - \overline{1/n_{Y}} \big)}{\mu\gamma} 
                            + 
                            \frac{4 B_w^2 B_\theta^2}{\mu^2} \cdot \frac{1}{n_X}
                        \bigg)
                        \frac{
                            78.8n_X \kappa + 1.3 n_X + 1.3
                        }
                        {
                            (64\kappa + 1)(64 n_X \kappa + n_X + 1)
                        }
                        P_{s-1}
                        \nn\\
                &\le
                        \bigg(
                            1
                            -
                            \frac{1}{78.8n_X \kappa + 1.3 n_X + 1.3}
                        \bigg)^M
                        P_{s,0}
                        +
                        \bigg(
                            \frac{8 B_f^2 \big( 1 - \overline{1/n_{Y}} \big)}{\mu\gamma} 
                            + 
                            \frac{4 B_w^2 B_\theta^2}{\mu^2} \cdot \frac{1}{n_X}
                        \bigg)
                        \frac{
                            1.3
                        }
                        {
                            64\kappa + 1
                        }
                        P_{s-1}
                        \nn\\
                &\le
                        \bigg(
                            1
                            -
                            \frac{1}{78.8n_X \kappa + 1.3 n_X + 1.3}
                        \bigg)^M
                        P_{s,0}
                        +
                        \frac{
                            \frac{16 B_f^2 \big( 1 - \overline{1/n_{Y}} \big)}{\mu\gamma} 
                            + 
                            \frac{8 B_w^2 B_\theta^2}{\mu^2} \cdot \frac{1}{n_X}
                        }
                        {
                            64\kappa + 1
                        }
                        P_{s-1}
                        \nn\\
                &\overset{(b)}{=}
                        \bigg[
                            \bigg(
                                1
                                -
                                \frac{1}{78.8n_X \kappa + 1.3 n_X + 1.3}
                            \bigg)^M
                            +
                            \frac{
                                \frac{16 B_f^2 \big( 1 - \overline{1/n_{Y}} \big)}{\mu\gamma} 
                                + 
                                \frac{8 B_w^2 B_\theta^2}{\mu^2} \cdot \frac{1}{n_X}
                            }
                            {
                                64\kappa + 1
                            }
                        \bigg]
                        P_{s-1}
                        \nn\\
                &\le
                        \bigg[
                            e^{-\frac{M}{78.8n_X \kappa + 1.3 n_X + 1.3}}
                            +
                            \frac{1}{4}
                        \bigg]
                        P_{s-1}
    \end{align}
where step (a) iteratively applies inequality \eqref{e:svrpda1:total_bound_final_interm7}, and step (b) uses the fact $P_{s,0}=P_{s-1}$ (because $
\theta^{(0)}=\tiltheta_{s-1}$ and $w^{(0)}=\tilde{w}_{s-1}$ as shown in Algorithm \ref{alg:svrpda1}). Choosing $M = \lceil 78.8n_X \kappa + 1.3 n_X + 1.3 \rceil$, where $\lceil \cdot \rceil$ denotes the roundup operation, we have
    \begin{align}
        P_s
                &\le
                        (e^{-1} + 1/4)P_{s-1}
                <
                        \frac{3}{4} P_{s-1}
                \le
                        \Big(\frac{3}{4}\Big)^s P_{0}
        \label{e:svrpda1:P_s_bound_Option1}
    \end{align}

Therefore, $P_s$ converges to zero at a linear rate of $3/4$. Furthermore, we requires a total of $\ln \frac{1}{\epsilon}$ outer loop iterations to reach $\epsilon$-solution. And for each outer loop iteration, it requires $M$ steps of inner-loop primal-dual updates, which is $O(1)$ per step (in number oracle calls), and $O(n_X n_Y)$ for evaluating the batch gradients for the control variates, where $n_Y=(n_{Y_0} + \cdots + n_{Y_{n_X-1}})/n_X$. Therefore, the complexity per outer loop iteration is $O(n_Xn_Y + M)$, and the total complexity is 
    \begin{align}
        O\Big(
            \big( n_X n_Y + n_X \kappa + n_X \big) \ln \frac{1}{\epsilon}
        \Big)
        \label{e:svrpda1:total_complexity}
    \end{align}
Noting that $\Expect\|\tiltheta_s-\theta^*\|^2 \le P_s$, the above bound \eqref{e:svrpda1:P_s_bound_Option1} also implies that $\Expect\|\tiltheta_s-\theta^*\|^2$ also converges to zero at a linear rate of $3/4$ and the total complexity to reach $\Expect\|\tiltheta_s-\theta^*\|^2 \le \epsilon$ is also given by \eqref{e:svrpda1:total_complexity}.

\subsection{Convergence for Option II}

Next, we move on to analyze the Option~II case of the algorithm, wherein $\tiltheta$ at the end of each state is chosen to be one of the $M$ previous $\theta^{(k)}$ values (see Algorithm \ref{alg:svrpda1} in the main paper).

Adding \eqref{e:SVRPDA_I:dual_bd_2} and \eqref{e:SVRPDA_I:final_bd_primal_final} we obtain the total bound for the primal and dual variable updates:

\begin{equation}
\begin{aligned}
& \Bigg( \frac{1}{2 \alpha_w} + \gamma \Bigg) \Expect \{ \|w^{(k)} - w^*\|^2 \mid \clF_k \} + \frac{1}{2 \alpha_w} \Expect \{ \|w^{(k)} - w^{(k - 1)} \|^2 \mid \clF_k \}
\\
& \Bigg (\frac{1}{2 \alpha_\theta} + \mu \Bigg ) \Expect \Big \{ \|\theta^{(k)} - \theta^*\|^2 \big | \clF_k \Big \}
+ \frac{1}{2 \alpha_\theta} \Expect \Big \{ \|\theta^{(k)} - \theta^{(k - 1)} \|^2 \big | \clF_k \Big \}
\\
& \leq \Bigg( \frac{1}{2 \alpha_w} + \frac{\gamma \big(n_X - 1 \big)}{n_X} \Bigg) \| w^{(k - 1)} - w^* \|^2 
+ \frac{1}{2 \alpha_\theta} \| \theta^{(k - 1)} - \theta^* \|^2 \\
& {+ \Big( 2\alpha_w B_f^2 \Big( 1 - \overline{1/n_{Y}} \Big) + \alpha_\theta B_w^2 B_\theta^2 \Big) \| \theta^{(k - 1)} -  \tiltheta \|^2}
\\
& + \Expect \Big \{ L(\theta^{(k - 1)}, w^{(k)} - w^*) - L(\theta^*, w^{(k)} - w^*) \mid \clF_k \Big \}
\\
& + \big(n_X - 1) \Expect \Big \{ L(\theta^{(k - 1)}, w^{(k)} - w^{(k - 1)}) - L(\theta^*, w^{(k)} - w^{(k - 1)}) \mid \clF_k \Big \}
\\
& + \Expect \Big \{ \Big[ L(\theta^{(k)} , w^*) - L(\theta^{*} , w^*) \Big ] \big | \clF_k \Big \}
\\
& +  \Expect \Big \{ \Big[ \Big \langle L'_\theta(\theta^{(k - 1)}, w^{(k)}) ,\, \theta^{*}  \Big \rangle - \Big \langle L'_\theta(\theta^{(k - 1)}, w^{(k)}) ,\, \theta^{(k)} \Big \rangle \Big] \big | \clF_k \Big \}
\end{aligned}
\label{e:opt2:SVRPDA_I:total_bound_1}
\end{equation}


We will now bound the $L$ terms in \eqref{e:opt2:SVRPDA_I:total_bound_1}. First consider the second $L$ term in \eqref{e:opt2:SVRPDA_I:total_bound_1}:
\begin{align}
& (n_X - 1) \Big[L(\theta^{(k - 1)}, w^{(k)} - w^{(k - 1)}) - L(\theta^*, w^{(k)} - w^{(k - 1)}) \Big]
\nn
\\
\overset{(a)}{\leq} & (n_X - 1) \Big[ \big \langle L'(\theta^{(k - 1)}, w^{(k)} - w^{(k - 1)}) \,, \theta^{(k-1)} - \theta^* \big \rangle
\Big ]
\nn
\\
\overset{(b)}{\leq} & {\beta_2 (n_X - 1)^2} \| L'(\theta^{(k - 1)}, w^{(k)} - w^{(k - 1)}) \|^2 + \frac{1}{\beta_2} \| \theta^{(k-1)} - \theta^* \|^2
\label{e:opt2:SVRPDA_I:second_L_term}
\end{align}
for some $\beta_2 > 0$ to be determined later, where step $(a)$ uses convexity of the function $L$ in its first variable, and $(b)$ uses $a^2+b^2 \ge 2ab \ge ab$.

We further lower bound the first term in \eqref{e:opt2:SVRPDA_I:second_L_term}. Ignoring the scaling factors, this can be rewritten as follows:
\begin{equation}
\begin{aligned}
\| L'(\theta^{(k - 1)}, w^{(k)} - w^{(k - 1)}) \|^2
& \overset{(a)}{=} \Big \|\frac{1}{n_X} \sum_{i = 0}^{n_X - 1} \barf'_i(\theta) (w_i^{(k)} - w_i^{(k - 1)}) \Big \|^2
\\
& \overset{(b)}{=} \frac{1}{n_X^2} \Big \| \sum_{i = 0}^{n_X - 1} \barf'_i(\theta) (w_i^{(k)} - w_i^{(k - 1)}) \Big \|^2
\\
& \overset{(c)}{\leq} \frac{B_f^2}{n_X^2} \Big \| \sum_{i = 0}^{n_X - 1} w_i^{(k)} - w_i^{(k - 1)} \Big \|^2
\end{aligned}
\label{e:opt2:SVRPDA_I:sec_L_term_2}
\end{equation}
where $(a)$ follows from the definition \eqref{e:SVRPDA_I:b_g_3}, $(b)$ is direct, by removing the $1/n_X$ outside the $\|\cdot\|^2$ operator, and $(c)$ uses the bounded gradients assumption (Assumption~\ref{a:f:smooth_bg}).

Notice that, conditioned on $\clF_k$, $w^{(k)}$ is the only random variable in \eqref{e:opt2:SVRPDA_I:sec_L_term_2}. Furthermore, for each $i$ and $j$, using \eqref{e:SVRPDA_I:dual_dummy_variable} we have, 
\begin{equation}
\sum_{i = 0}^{n_X - 1} w_i^{(k)} - w_i^{(k-1)} = w'_{ij} - w_i^{(k-1)} \qquad \text{with probability \quad $1/n_X n_{Y_i}$}
\label{e:opt2:SVRPDA_I:joint_prob}
\end{equation}

Taking expectation, conditioned on $\clF_k$ on both sides of \eqref{e:opt2:SVRPDA_I:sec_L_term_2},
\begin{equation}
\begin{aligned}
\Expect \Big \{ \| L'(\theta^{(k - 1)}, w^{(k)} - w^{(k - 1)}) \|^2 \Big | \clF_k \Big \}
& {\leq} \frac{B_f^2}{n_X^2} \Expect \Big \{ \Big \| \sum_{i = 0}^{n_X - 1} w_i^{(k)} - w_i^{(k - 1)} \Big \|^2 \Big | \clF_k \Big \}
\\
& \overset{(a)}{=} \frac{B_f^2}{n_X^2} \frac{1}{n_X n_{Y_i}} \sum_{i = 0}^{n_X - 1} \sum_{j = 0}^{n_{Y_i} - 1} \big \| w'_{ij} - w_i^{(k - 1)} \big \|^2
\\
& \overset{(b)}{=} \frac{B_f^2}{n_X^2} \sum_{i = 0}^{n_X - 1} 
\Expect \Big \{  \big \| w^{(k)}_{i} - w_i^{(k - 1)} \big \|^2 \Big | \clF_k \Big \} 
\\
& \overset{(c)}{=} \frac{B_f^2}{n_X^2}
\Expect \Big \{  \big \| w^{(k)} - w^{(k - 1)} \big \|^2 \Big | \clF_k \Big \} 
\end{aligned}
\label{e:opt2:SVRPDA_I:sec_L_term_3}
\end{equation}
where $(a)$ follows from \eqref{e:opt2:SVRPDA_I:joint_prob}, $(b)$ follows from the last identity in \eqref{e:SVRPDA_I:expec_w_i}, and $(c)$ is just definition. It follows from \eqref{e:opt2:SVRPDA_I:second_L_term} and \eqref{e:opt2:SVRPDA_I:sec_L_term_3} that:
\begin{equation}
\begin{aligned}
& (n_X - 1) \Expect \Big \{  \Big[L(\theta^{(k - 1)}, w^{(k)} - w^{(k - 1)}) - L(\theta^*, w^{(k)} - w^{(k - 1)}) \Big] \Big | \clF_k \Big \}
\\
& {\leq} \frac{\beta_2 B_f^2 (n_X - 1)^2}{n_X^2}
\Expect \Big \{  \big \| w^{(k)} - w^{(k - 1)} \big \|^2 \Big | \clF_k \Big \}  + \frac{1}{\beta_2} \| \theta^{(k-1)} - \theta^* \|^2
\end{aligned}
\label{e:opt2:SVRPDA_I:second_L_term_final}
\end{equation}

\bigskip

Next consider the remaining $L$ terms (lines $4$ and $5$ of \eqref{e:opt2:SVRPDA_I:total_bound_1}). We have:
\begin{align}
& L \big (\theta^{(k - 1)}, \,  w^{(k)} - w^* \big) -  L \big (\theta^*, \,  w^{(k)} - w^* \big) 
+ {L(\theta^{(k)} , w^*) - L(\theta^*, w^*)}
\nn\\
&=  
L \big (\theta^{(k - 1)}, \,  w^{(k)} \big) -  L \big (\theta^*, \,  w^{(k)} \big) + L \big (\theta^*, \, w^* \big) -  L \big (\theta^{(k - 1)}, \,  w^* \big) + {L(\theta^{(k)} , w^*) - L(\theta^*, w^*)}
\nn\\
&=
L \big (\theta^{(k - 1)}, \,  w^{(k)} \big) -  L \big (\theta^*, \,  w^{(k)} \big) + L(\theta^{(k)} , w^*) -  L \big (\theta^{(k - 1)}, \,  w^* \big)
\label{e:opt2:SVRPDA_I:total_5terms_bound1}
\end{align}
Therefore, the remaining $L$ terms of \eqref{e:opt2:SVRPDA_I:total_bound_1} can be bounded as:
\begin{align}
& {L(\theta^{(k)} , w^*) - L(\theta^*, w^*)} 
+ 
L \big (\theta^{(k - 1)},  w^{(k)} \!-\! w^* \big) \!-\!  L \big (\theta^*,  w^{(k)} \!-\! w^* \big) \!-\! \langle L'_\theta(\theta^{(k - 1)}, w^{(k)}) , \theta^{(k)} \!-\! \theta^* \rangle 
\nn\\
&\overset{(a)}{=}
L \big (\theta^{(k - 1)}, \,  w^{(k)} \big) -  L \big (\theta^*, \,  w^{(k)} \big) 
+ L(\theta^{(k)} , w^*) -  L \big (\theta^{(k - 1)}, \,  w^* \big)
\!-\! \langle L'_\theta(\theta^{(k - 1)}, w^{(k)}) , \theta^{(k)} \!-\! \theta^* \rangle 
\nn\\
&\overset{(b)}{\le}
\big\langle
L'_\theta(\theta^{(k - 1)}, \,  w^{(k)}), \theta^{(k - 1)}-\theta^*
\big\rangle
+
\big\langle
L'_{\theta}(\theta^{(k)} , w^*), \theta^{(k)}-\theta^{(k - 1)}
\big\rangle
-
\langle 
L'_\theta(\theta^{(k - 1)}, w^{(k)}) , \theta^{(k)} \!-\! \theta^* 
\rangle 
\nn\\
&\overset{(c)}{=}
-\big\langle
L'_\theta(\theta^{(k - 1)}, \,  w^{(k)}), \theta^{(k)} - \theta^{(k - 1)}
\big\rangle
+
\big\langle
L'_{\theta}(\theta^{(k)} , w^*), \theta^{(k)}-\theta^{(k - 1)}
\big\rangle
\nn\\
&\overset{(d)}{=}
-\big\langle
L'_\theta(\theta^{(k - 1)}, \,  w^{(k)}), \theta^{(k)} - \theta^{(k - 1)}
\big\rangle
+
\big\langle
L'_\theta(\theta^{(k - 1)}, \,  w^*), \theta^{(k)} - \theta^{(k - 1)}
\big\rangle
\nn\\
&
\quad-
\big\langle
L'_\theta(\theta^{(k - 1)}, \,  w^*), \theta^{(k)} - \theta^{(k - 1)}
\big\rangle
+
\big\langle
L'_{\theta}(\theta^{(k)} , w^*), \theta^{(k)}-\theta^{(k - 1)}
\big\rangle
\nn\\
&\overset{(e)}{=}
\big\langle
L'_\theta(\theta^{(k - 1)}, \,  w^*) - L'_\theta(\theta^{(k - 1)}, \,  w^{(k)}), 
\theta^{(k)} - \theta^{(k - 1)}
\big\rangle
\nn
\\
& \hspace{0.4in}+
\big\langle
L'_{\theta}(\theta^{(k)} , w^*) - L'_\theta(\theta^{(k - 1)}, \,  w^*),
\theta^{(k)}-\theta^{(k - 1)}
\big\rangle
\nn\\
&\overset{(f)}{=}
-\big\langle
L'_\theta(\theta^{(k - 1)}, \,  w^{(k)} - w^*), 
\theta^{(k)} - \theta^{(k - 1)}
\big\rangle
\nn
\\
& \hspace{0.4in}+
\big\langle
L'_{\theta}(\theta^{(k)} , w^*) - L'_\theta(\theta^{(k - 1)}, \,  w^*),
\theta^{(k)}-\theta^{(k - 1)}
\big\rangle
\label{e:opt2:SVRPDA_I:total_5terms_bound2}
\end{align}

where step $(a)$ substitutes \eqref{e:opt2:SVRPDA_I:total_5terms_bound1}, step $(b)$ uses the convexity of $L(\theta,w)$ in $\theta$ by applying $f(x)-f(y) \le \langle f'(x), x-y \rangle$, step $(c)$ merges the first and the third terms in line $(b)$, step $(d)$ adds and subtracts the same term (i.e., the second and the third terms), step $(e)$ merges the first term with the second term and also merges the third term and the fourth term, step $(f)$ uses the linearity of $L(\theta,w)$ in $w$. We now proceed to bound the two terms in \eqref{e:opt2:SVRPDA_I:total_5terms_bound2}. For a $\beta_1 > 0$ (to be determined later), the first term in \eqref{e:opt2:SVRPDA_I:total_5terms_bound2} can be upper bounded as
\begin{align}
&\big | 
\langle L'_\theta(\theta^{(k - 1)}, w^{(k)} - w^*) ,\, \theta^{(k)} - \theta^{(k - 1)} \rangle
\big | 
\nn\\
&\overset{(a)}{\leq}
\frac{1}{\beta_1} \big \| L'_\theta(\theta^{(k - 1)}, w^{(k)} - w^*) \big \|^2 
+ 
\beta_1 \big \| \theta^{(k)} - \theta^{(k - 1)} \big \|^2
\nn\\
&\overset{(b)}{\leq}
\frac{B_f^2}{\beta_1 n_X} \big \| w^{(k)} - w^* \big \|^2 
+ 
\beta_1 \big \| \theta^{(k)} - \theta^{(k - 1)} \big \|^2
\label{e:opt2:SVRPDA_I:total_5term_bound2_term1}
\end{align}    
where {$(a)$ uses Cauchy-Schwartz inequality and the fact that $a^2+b^2 \ge 2ab \ge ab$}, and step $(b)$ uses the definition of $L_{\theta}'$ in \eqref{e:SVRPDA_I:b_g_3} and Jensen's inequality for $\|\cdot\|^2$. Next, the second term in \eqref{e:opt2:SVRPDA_I:total_5terms_bound2} can be upper bounded as
\begin{align}
&\big|\big\langle
L'_{\theta}(\theta^{(k)} , w^*) - L'_\theta(\theta^{(k - 1)}, \,  w^*),
\theta^{(k)}-\theta^{(k - 1)}
\big\rangle\big|
\nn\\
&\overset{(a)}{\le}
\| L'_{\theta}(\theta^{(k)} , w^*) - L'_\theta(\theta^{(k - 1)}, \,  w^*) \|
\cdot
\big \| \theta^{(k)} - \theta^{(k - 1)} \big \|
\nn\\
&\overset{(b)}{\le}
B_{\theta} B_w
\| \theta^{(k)} - \theta^{(k - 1)} \|^2
\label{e:opt2:SVRPDA_I:total_5term_bound2_term2}
\end{align}
where step $(a)$ uses Cauchy-Schwartz inequality and step $(b)$ uses Lipschitz condition of the gradient $f_{\theta}'$ together with the boundedness of $w^*$ (Lemma \ref{lemma:conjugate_Lipschitz}). Substituting \eqref{e:opt2:SVRPDA_I:total_5term_bound2_term1}--\eqref{e:opt2:SVRPDA_I:total_5term_bound2_term2} into \eqref{e:opt2:SVRPDA_I:total_5terms_bound2}, we obtain
\begin{align}
& {L(\theta^{(k)} , w^*) - L(\theta^*, w^*)} 
+ 
L \big (\theta^{(k - 1)},  w^{(k)} \!-\! w^* \big) \!-\!  L \big (\theta^*,  w^{(k)} \!-\! w^* \big) \!-\! \langle L'_\theta(\theta^{(k - 1)}, w^{(k)}) , \theta^{(k)} \!-\! \theta^* \rangle 
\nn\\
&\le
\frac{B_f^2}{\beta_1 n_X} \big \| w^{(k)} - w^* \big \|^2 
+ 
\beta_1 \big \| \theta^{(k)} - \theta^{(k - 1)} \big \|^2
+
B_{\theta} B_w
\| \theta^{(k)} - \theta^{(k - 1)} \|^2
\nn\\
&=
\frac{B_f^2}{\beta_1 n_X} \big \| w^{(k)} - w^* \big \|^2 
+ 
(\beta_1 + B_{\theta} B_w) \big \| \theta^{(k)} - \theta^{(k - 1)} \big \|^2
\label{e:opt2:SVRPDA_I:total_5terms_bound3}
\end{align}

We have now bounded all the $L$ terms in \eqref{e:opt2:SVRPDA_I:total_bound_1}.


Substituting both \eqref{e:opt2:SVRPDA_I:second_L_term_final} and  \eqref{e:opt2:SVRPDA_I:total_5terms_bound3} in \eqref{e:opt2:SVRPDA_I:total_bound_1}, we get the final bound, without the $L$ terms as follows:
\begin{equation}
\begin{aligned}
& \Bigg( \frac{1}{2 \alpha_w} + \gamma \Bigg) \Expect \{ \|w^{(k)} - w^*\|^2 \mid \clF_k \} + \frac{1}{2 \alpha_w} \Expect \{ \|w^{(k)} - w^{(k - 1)} \|^2 \mid \clF_k \}
\\
& \Bigg (\frac{1}{2 \alpha_\theta} + \mu \Bigg ) \Expect \Big \{ \|\theta^{(k)} - \theta^*\|^2 \big | \clF_k \Big \}
+ \frac{1}{2 \alpha_\theta} \Expect \Big \{ \|\theta^{(k)} - \theta^{(k - 1)} \|^2 \big | \clF_k \Big \}
\\
& \leq \Bigg( \frac{1}{2 \alpha_w} + \frac{\gamma \big(n_X - 1 \big)}{n_X} \Bigg) \| w^{(k - 1)} - w^* \|^2 
+ \frac{1}{2 \alpha_\theta} \| \theta^{(k - 1)} - \theta^* \|^2 \\
& {+ \Big( 2\alpha_w B_f^2 \Big( 1 - \overline{1/n_{Y}} \Big) + \alpha_\theta B_w^2 B_\theta^2 \Big) \| \theta^{(k - 1)} -  \tiltheta \|^2}
\\
& + \frac{\beta_2 B_f^2 (n_X - 1)^2}{n_X^2} \Expect \Big \{ \|w^{(k)} - w^{(k - 1)}) \|^2 \big | \clF_k \Big \}
+ \frac{1}{\beta_2} \| \theta^{(k-1)} - \theta^* \|^2 
\\
& + \frac{B_f^2}{\beta_1 n_X} \Expect \Big \{ \big \|w^{(k)} - w^* \big \|^2 \big | \clF_k \Big \}
+ (\beta_1 + B_{\theta} B_w) \Expect \Big \{ \big \| \theta^{(k)} - \theta^{(k - 1)} \big \|^2 \big | \clF_k \Big \}
\end{aligned}
\label{e:opt2:SVRPDA_I:total_bound_f1}
\end{equation}

Combining common terms and rearranging,
\begin{equation}
\begin{aligned}
& \Bigg (\frac{1}{2 \alpha_\theta} + \mu \Bigg ) \Expect \Big \{ \|\theta^{(k)} - \theta^*\|^2 \big | \clF_k \Big \}
+ \Bigg( \frac{1}{2 \alpha_w} + \gamma - \frac{B_f^2}{\beta_1 n_X} \Bigg) \Expect \{ \|w^{(k)} - w^*\|^2 \mid \clF_k \}
\\
& \leq \Bigg( \frac{1}{2 \alpha_w} + \frac{\gamma \big(n_X - 1 \big)}{n_X} \Bigg) \| w^{(k - 1)} - w^* \|^2 
\\
& + \Bigg(\frac{1}{2 \alpha_\theta} + \frac{1}{\beta_2} + {4\alpha_w B_f^2 \Big( 1 - \overline{1/n_{Y}} \Big) + 2 \alpha_\theta B_w^2 B_\theta^2} \Bigg) \| \theta^{(k - 1)} - \theta^* \|^2 
\\
&+ \Bigg(\beta_1 \!+\! B_{\theta} B_w \!-\! \frac{1}{2 \alpha_\theta} \Bigg) \Expect \Big \{ \|\theta^{(k)} - \theta^{(k - 1)} \|^2 \big | \clF_k \Big \} 
\\
& + \Bigg(\frac{\beta_2 B_f^2 (n_X - 1)^2}{{n_X^2}} \!-\! \frac{1}{2 \alpha_w} \Bigg) \Expect \{ \|w^{(k)} - w^{(k - 1)} \|^2 \mid \clF_k \}
\\
& {+ \Big( 4 \alpha_w B_f^2 \Big( 1 - \overline{1/n_{Y}} \Big) + 2 \alpha_\theta B_w^2 B_\theta^2 \Big) \| \tiltheta -  \theta^* \|^2}
\end{aligned}
\label{e:opt2:SVRPDA_I:dual_primal_combined_1}
\end{equation}
where we have also used the inequality $(a+b)^2 \leq 2 a^2 + 2 b^2$ for the $\|\theta^{(k-1)} - \tiltheta \|^2$ term.


We now choose the step-sizes $\alpha_{\theta}$, $\alpha_{w}$ and $M$ as in Theorem~\ref{t:final_bound_SVRPDA_I_opt2}:

\begin{equation}
\alpha_\theta = \Big( \frac{25 B_f^2}{\gamma} + {10 B_\theta B_w} + \frac{80 B_w^2 B_\theta^2}{\mu} \Big)^{-1}
\label{e:opt2:SVRPDA_I:alpha_theta:old:combined}
\end{equation}
\begin{equation}
\alpha_w = \frac{\mu}{40 B_f^2}
\label{e:opt2:SVRPDA_I:alpha_w:old}
\end{equation}
\begin{equation}
\begin{aligned}
M = \max \Bigg( \frac{10}{\alpha_\theta \mu} \,, \frac{2 {n_X}}{\alpha_w \gamma} \,, 4 n_X \Bigg)
\end{aligned}
\label{e:opt2:SVRPDA_I:M:old}
\end{equation}

The choice of $\alpha_\theta$ in \eqref{e:opt2:SVRPDA_I:alpha_theta:old:combined} ensures the following three bounds:
\begin{equation}
\alpha_\theta \leq \frac{\gamma}{25 B_f^2} 
\qquad \text{or} 
\qquad \alpha_\theta \leq \frac{1}{10 B_\theta B_w}
\qquad \text{or} 
\qquad \alpha_\theta \leq \frac{\mu}{80 B_w^2 B_\theta^2}
\label{e:opt2:SVRPDA_I:alpha_theta:old}
\end{equation}

Furthermore, we choose $\beta_1$ and $\beta_2$ as follows:
\begin{equation}
\beta_1 = - {B_\theta B_w} + \frac{1}{2 \alpha_\theta}
\label{e:opt2:SVRPDA_I_beta_1:old}
\end{equation}
\begin{equation}
\beta_2 = \frac{1}{2 \alpha_w B_f^2}
\label{e:opt2:SVRPDA_I_beta_2:old}
\end{equation}
The second inequality in \eqref{e:opt2:SVRPDA_I:alpha_theta:old} will ensure positivity of $\beta_1$.

\paragraph{Applying the above choice of hyper-parameters:}

Based on \eqref{e:opt2:SVRPDA_I_beta_1:old} we have:
\begin{equation}
\begin{aligned}
\beta_1 + B_{\theta} B_w - \frac{1}{2 \alpha_\theta} & = 0
\label{e:opt2:SVRPDA_I:beta_1_effect:old}
\end{aligned}
\end{equation}
and using \eqref{e:opt2:SVRPDA_I_beta_2:old}:
\begin{equation}
\begin{aligned}
{ \frac{\beta_2 B_f^2 (n_X-1)^2}{n_X^2}} - \frac{1}{2 \alpha_w} 
& = { \frac{1}{2 \alpha_w B_f^2} \frac{B_f^2 (n_X-1)^2}{n_X^2}} - \frac{1}{2 \alpha_w} 
\\
& = { \frac{1}{2 \alpha_w} \frac{(n_X-1)^2}{n_X^2}} - \frac{1}{2 \alpha_w} 
\\
& < 0
\label{e:opt2:SVRPDA_I:beta_2_effect:old}
\end{aligned}
\end{equation}
Equations \eqref{e:opt2:SVRPDA_I:beta_1_effect:old} and \eqref{e:opt2:SVRPDA_I:beta_2_effect:old} will ensure that the third and fourth terms on the right hand side of \eqref{e:opt2:SVRPDA_I:dual_primal_combined_1} are either $0$ or negative, and therefore can be ignored, reducing the bound in \eqref{e:opt2:SVRPDA_I:dual_primal_combined_1} to:
\begin{equation}
\begin{aligned}
& \Bigg (\frac{1}{2 \alpha_\theta} + \mu \Bigg ) \Expect \Big \{ \|\theta^{(k)} - \theta^*\|^2 \big | \clF_k \Big \}
+ \Bigg( \frac{1}{2 \alpha_w} + \gamma - \frac{B_f^2}{\beta_1 n_X} \Bigg) \Expect \{ \|w^{(k)} - w^*\|^2 \mid \clF_k \}
\\
& \leq \Bigg( \frac{1}{2 \alpha_w} + \frac{\gamma \big(n_X - 1 \big)}{n_X} \Bigg) \| w^{(k - 1)} - w^* \|^2 
\\
& + \Bigg(\frac{1}{\beta_2} + \frac{1}{2 \alpha_\theta} + {4 {\alpha_w } B_f^2 + 2 \alpha_\theta B_w^2 B_\theta^2} \Bigg) \| \theta^{(k - 1)} - \theta^* \|^2
\\
& {+ \Big( 4 {\alpha_w } B_f^2 + 2 \alpha_\theta B_w^2 B_\theta^2 \Big) \| \tiltheta -  \theta^* \|^2}
\end{aligned}
\label{e:opt2:SVRPDA_I:bd_simplification_1:old}
\end{equation}
where we have also used $(1 - \overline{1/n_{Y}}) \leq 1$. 

The $\tiltheta$ in the \eqref{e:opt2:SVRPDA_I:bd_simplification_1:old} is $\tiltheta_{s-1}$, the fixed primal variable at the beginning of stage $s$. Denote $\tiltheta_s$ to be the primal variable randomly chosen among $\theta^{(k)}$ for $1 \leq k \leq M$ at the end of stage $s$. We define $\tilw_{s-1}$ and $\tilw_{s}$ in a similar manner (though neither of them are used in the algorithm), and also note that $\theta^{(0)}$ and $w^{(0)}$ are initialized to $\tiltheta_{s-1}$ and $\tilw_{s-1}$ at the beginning of stage $s$. 
Summing both sides of  \eqref{e:opt2:SVRPDA_I:bd_simplification_1:old} over all $1 \leq k \leq M$,
\begin{equation}
\begin{aligned}
& \Bigg (\frac{1}{2 \alpha_\theta} + \mu \Bigg ) \Expect \Big \{ \|\theta^{(M)} - \theta^*\|^2 \big | \clF_s \Big \}
+ \Bigg( \frac{1}{2 \alpha_w} + \gamma - \frac{B_f^2}{\beta_1 n_X} \Bigg) \Expect \{ \|w^{(M)} - w^*\|^2 \mid \clF_s \}
\\
& + \Bigg (\mu - \frac{1}{\beta_2} - {4 {\alpha_w } B_f^2 - 2\alpha_\theta B_w^2 B_\theta^2} \Bigg ) 
\sum_{k = 1}^{M-1} \Expect \Big \{ \|\theta^{(k)} - \theta^*\|^2 \big | \clF_s \Big \}
\\
& + \Bigg(\frac{\gamma}{n_X} - \frac{B_f^2}{\beta_1 n_X} \Bigg) \sum_{k = 1}^{M-1} \Expect \{ \|w^{(k)} - w^*\|^2 \mid \clF_s \}
\\
& \leq \Bigg( \frac{1}{2 \alpha_w} + \frac{\gamma \big(n_X - 1 \big)}{n_X} \Bigg) \| w^{(0)} - w^* \|^2 
+ \Bigg(\frac{1}{\beta_2} + \frac{1}{2 \alpha_\theta} \Bigg) \| \theta^{(0)} - \theta^* \|^2
\\
& {+ M \Big( 4 {\alpha_w } B_f^2 + 2 \alpha_\theta B_w^2 B_\theta^2 \Big) \| \tiltheta -  \theta^* \|^2}
\end{aligned}
\label{e:opt2:SVRPDA_I:total_bd_sum_M:old}
\end{equation}
Substituting $w^{(0)} = \tilw_{s-1}$ and $\theta^{(0)} = \tiltheta_{s-1}$, and also noting that the first two terms on the left hand side of \eqref{e:opt2:SVRPDA_I:total_bd_sum_M:old} can be combined with the second two terms (note that the difference in coefficients are positive, and positive terms on the left hand side of the inequality can be ignored):
\begin{equation}
\begin{aligned}
& \Bigg (\mu - \frac{1}{\beta_2} - {4 {\alpha_w } B_f^2 - 2 \alpha_\theta B_w^2 B_\theta^2} \Bigg ) 
\sum_{k = 1}^{M} \Expect \Big \{ \|\theta^{(k)} - \theta^*\|^2 \big | \clF_s \Big \}
\\
& + \Bigg(\frac{\gamma}{n_X} - \frac{B_f^2}{\beta_1 n_X} \Bigg) \sum_{k = 1}^{M} \Expect \{ \|w^{(k)} - w^*\|^2 \mid \clF_s \}
\\
& \leq \Bigg( \frac{1}{2 \alpha_w} + \frac{\gamma \big(n_X - 1 \big)}{n_X} \Bigg) \| \tilw_{s-1} - w^* \|^2 
\\
& + \Bigg(\frac{1}{\beta_2} + \frac{1}{2 \alpha_\theta} + M \Big( 4 {\alpha_w } B_f^2 + 2 \alpha_\theta B_w^2 B_\theta^2 \Big) \Bigg) \| \tiltheta_{s-1} - \theta^* \|^2
\end{aligned}
\label{e:opt2:SVRPDA_I:total_bd_sum_M_2:old}
\end{equation}

Dividing both sides of \eqref{e:opt2:SVRPDA_I:total_bd_sum_M_2:old} by $M$ and applying Jensen's inequality on the left hand side, we obtain:
\begin{equation}
\begin{aligned}
& \Bigg (\mu \!-\! \frac{1}{\beta_2} \!-\! {4 {\alpha_w } B_f^2 \!-\! 2 \alpha_\theta B_w^2 B_\theta^2} \Bigg ) 
\Expect \Big \{ \|\tiltheta_{s} - \theta^*\|^2 \big | \clF_s \Big \} 
\\
& + \Bigg(\frac{\gamma}{n_X} \!-\! \frac{B_f^2}{\beta_1 n_X} \Bigg) \Expect \{ \|\tilw_{s} - w^*\|^2 \mid \clF_s \}
\\
& \leq \Bigg( \frac{1}{2 M \alpha_w} \!+\! \frac{\gamma \big(n_X - 1 \big)}{M n_X} \Bigg) \| \tilw_{s-1} - w^* \|^2 
\\
& +  \Bigg(\frac{1}{M \beta_2} \!+\! \frac{1}{2 M \alpha_\theta} \!+\! 4 {\alpha_w } B_f^2 + 2 \alpha_\theta B_w^2 B_\theta^2 \Big) \| \tiltheta_{s-1} - \theta^* \|^2
\end{aligned}
\label{e:opt2:SVRPDA_I:total_bd_til_param:old}
\end{equation}

We now substitute the hyper-parameter values in \eqref{e:opt2:SVRPDA_I:alpha_theta:old:combined} \eqref{e:opt2:SVRPDA_I:alpha_w:old} \eqref{e:opt2:SVRPDA_I:M:old} \eqref{e:opt2:SVRPDA_I_beta_1:old} \eqref{e:opt2:SVRPDA_I_beta_2:old} in \eqref{e:opt2:SVRPDA_I:total_bd_til_param:old} to obtain a linear rate. 

Substituting these values in the coefficient of the first term on the left hand side of \eqref{e:opt2:SVRPDA_I:total_bd_til_param:old}
{(we're using the third bound for $\alpha_\theta$ in \eqref{e:opt2:SVRPDA_I:alpha_theta:old} here)}:
\begin{equation}
\begin{aligned}
\mu - \frac{1}{\beta_2} - {4 {\alpha_w } B_f^2 -  2 \alpha_\theta B_w^2 B_\theta^2} & = \mu - 2 \alpha_w B_f^2 - {4 {\alpha_w } B_f^2 - 2 \alpha_\theta B_w^2 B_\theta^2}
\\
& \geq \mu - 2 \frac{\mu}{40 B_f^2} B_f^2 - {4 \frac{\mu}{40 B_f^2} B_f^2 - 2  \frac{\mu}{80 B_w^2 B_\theta^2} B_w^2 B_\theta^2}
\\
& = \mu - \frac{\mu}{20} - {\frac{\mu}{10} - \frac{\mu}{40}}
\\
& = \mu - \frac{7 \mu}{40}
\\
& = \frac{33 \mu}{40}
\\
& \geq \frac{4 \mu}{5}
\end{aligned}
\label{e:opt2:SVRPDA_I:final_bd_coeff_1:old}
\end{equation}

Next, substituting for the coefficient of the second term on the left hand side of \eqref{e:opt2:SVRPDA_I:total_bd_til_param:old}
(here we use the first and second bounds for $\alpha_\theta$ in \eqref{e:opt2:SVRPDA_I:alpha_theta:old}, both in the inequality in the fourth line):
\begin{equation}
\begin{aligned}
\frac{\gamma}{n_X} - \frac{B_f^2}{\beta_1 n_X} 
& = \frac{\gamma}{n_X} - \frac{B_f^2}{n_X} \Bigg( - B_\theta B_w + \frac{1}{2 \alpha_\theta} \Bigg)^{-1}
\\
& = \frac{\gamma}{n_X} - \frac{B_f^2}{n_X} \Bigg(\frac{1 - 2 \alpha_\theta B_\theta B_w}{2 \alpha_\theta } \Bigg)^{-1}
\\
& = \frac{\gamma}{n_X} - \frac{B_f^2}{n_X} \Bigg(\frac{2 \alpha_\theta }{1 - 2 \alpha_\theta B_\theta B_w} \Bigg)
\\
& \geq \frac{\gamma}{n_X} - \frac{B_f^2}{n_X} \times \frac{2 \gamma}{25 B_f^2} \times \frac{1}{1 - \frac{2}{10 B_\theta B_w} B_\theta B_w}
\\
& = \frac{\gamma}{n_X} - \frac{2 \gamma}{25 n_X} \times \frac{5}{4}
\\
& = \frac{\gamma}{n_X} - \frac{\gamma}{10 n_X}
\\
& \geq \frac{4 \gamma}{5 n_X}
\end{aligned}
\label{e:opt2:SVRPDA_I:final_bd_coeff_2:old}
\end{equation}

Next consider the coefficient of the first term on the right hand side of \eqref{e:opt2:SVRPDA_I:total_bd_til_param:old}
{(we use the second and third values of $M$ in \eqref{e:opt2:SVRPDA_I:M:old}, both in the second line)}:
\begin{equation}
\begin{aligned}
\frac{1}{2 M \alpha_w} + \frac{\gamma \big(n_X - 1 \big)}{M n_X} 
& \leq \frac{1}{2 M \alpha_w} + \frac{\gamma}{M}
\\
& \leq \frac{\gamma}{4 n_X} + \frac{\gamma}{4 n_X}
\\
& = \frac{\gamma}{2 n_X}
\end{aligned}
\label{e:opt2:SVRPDA_I:final_bd_coeff_3:old}
\end{equation}

Finally, we consider the coefficient of the second term on the right hand side of \eqref{e:opt2:SVRPDA_I:total_bd_til_param:old} {(we use the third bound for $\alpha_\theta$ in \eqref{e:opt2:SVRPDA_I:alpha_theta:old}, and the first and third definitions of $M$ in \eqref{e:opt2:SVRPDA_I:M:old})}:
\begin{equation}
\begin{aligned}
\frac{1}{M \beta_2} + \frac{1}{2 M \alpha_\theta} + 4 {\alpha_w } B_f^2 + 2 \alpha_\theta B_w^2 B_\theta^2 
& \leq \frac{2 \alpha_w B_f^2}{M} + \frac{\mu}{20} + 4 \frac{\mu}{40 B_f^2} B_f^2 + \frac{\mu}{80 B_w^2 B_\theta^2} 2 B_w^2 B_\theta^2 
\\
& \leq {\frac{\mu}{40 B_f^2} 2 B_f^2} + \frac{\mu}{20} + \frac{\mu}{10} + \frac{\mu}{40}
\\
& = {\frac{9 \mu}{40}} 
\\
& = \frac{\mu}{4}
\end{aligned}
\label{e:opt2:SVRPDA_I:final_bd_coeff_4:old}
\end{equation}

Using \eqref{e:opt2:SVRPDA_I:final_bd_coeff_1:old}, \eqref{e:opt2:SVRPDA_I:final_bd_coeff_2:old}, \eqref{e:opt2:SVRPDA_I:final_bd_coeff_3:old},  and \eqref{e:opt2:SVRPDA_I:final_bd_coeff_4:old} in \eqref{e:opt2:SVRPDA_I:total_bd_til_param:old},
we have:
\begin{equation}
\begin{aligned}
& \frac{4 \mu}{5} 
\Expect \Big \{ \|\tiltheta_{s} - \theta^*\|^2 \big | \clF_s \Big \} 
+ \frac{4 \gamma}{5 n_X} \Expect \{ \|\tilw_{s} - w^*\|^2 \mid \clF_s \}
\\
& \leq \frac{\gamma}{2 n_X} \| \tilw_{s-1} - w^* \|^2 
+ \frac{\mu}{4} \| \tiltheta_{s-1} - \theta^* \|^2
\end{aligned}
\label{e:opt2:SVRPDA_I:total_bd_final:old}
\end{equation}
Applying full expectation to both sides of the above inequality, we obtain
\begin{equation}
\begin{aligned}
& \frac{4 \mu}{5} 
\Expect \|\tiltheta_{s} - \theta^*\|^2
+ \frac{4 \gamma}{5 n_X} \Expect \|\tilw_{s} - w^*\|^2
\\
& \leq \frac{\gamma}{2 n_X} \Expect \| \tilw_{s-1} - w^* \|^2 
+ \frac{\mu}{4} \Expect \| \tiltheta_{s-1} - \theta^* \|^2.
\end{aligned}
\label{e:opt2:SVRPDA_I:total_bd_final:old2}
\end{equation}
Dividing both sides by $4\mu/5$, we have
    \begin{align} 
        &\Expect \|\tiltheta_{s} - \theta^*\|^2
        + 
        \frac{\gamma}{n_X \mu} \Expect \|\tilw_{s} - w^*\|^2
        \nn\\
        & \leq \frac{5}{16} \Expect \| \tilw_{s-1} - w^* \|^2 
        + 
        \frac{5\gamma}{8\mu n_X} \Expect \| \tiltheta_{s-1} - \theta^* \|^2,
    \end{align}
which can be further bounded as
    \begin{align}
        &\Expect \|\tiltheta_{s} - \theta^*\|^2
        + 
        \frac{\gamma}{n_X \mu} \Expect \|\tilw_{s} - w^*\|^2
        \nn\\
        & \leq \frac{5}{8} 
        \Big(
            \Expect \| \tilw_{s-1} - w^* \|^2 
            + 
            \frac{\gamma}{\mu n_X} \Expect \| \tiltheta_{s-1} - \theta^* \|^2
        \Big).
    \end{align}
Define the Lyapunov function $P_s$ to be
    \begin{align}
        P_s
                &=
                        \Expect \|\tiltheta_{s} - \theta^*\|^2
                        + 
                        \frac{\gamma}{n_X \mu} \Expect \|\tilw_{s} - w^*\|^2.
                        \nn
    \end{align}
Then, inequality \eqref{e:opt2:SVRPDA_I:total_bd_final:old2} can be expressed as
    \begin{align}
        P_s
                &\le
                        \frac{5}{8} P_{s-1}
                \le
                        \Big(\frac{5}{8}\Big)^s P_{0}.
        \label{e:opt2:SVRPDA_I:P_s_bound_Option2}
    \end{align}
Therefore, $P_s$ converges to zero at a linear rate of $5/8$. In order to achieve $\epsilon$-precision solution (i.e., $P_s \le \epsilon$), it requires a total of $O(\ln \frac{1}{\epsilon})$ outer-loop iterations (stages). And for each outer loop iteration, it requires $M$ steps of inner-loop primal-dual updates, which is $O(1)$ per step (in number of oracle calls), and $O(n_X n_Y)$ for evaluating the batch gradients for the control variates, where $n_Y = (n_{Y_0} + \cdots + n_{Y_{n_X-1}})/n_X$. Therefore, the complexity per outer loop iteration is $O(n_Xn_Y + M)$ so that the total complexity can be written as:
    \begin{equation}
        O\Big(\big(n_X n_Y + M \big) \ln \big( \frac{P_0}{\epsilon} \big) \Big).
    \end{equation}
Recall from \eqref{e:opt2:SVRPDA_I:M:old} that $M$ is given by
\[
M = 10/\mu \alpha_\theta + \frac{2 {n_X}}{\alpha_w \gamma} + 4 n_X.
\]
where, by \eqref{e:opt2:SVRPDA_I:alpha_theta:old:combined} and \eqref{e:opt2:SVRPDA_I:alpha_w:old}, the step-sizes $\alpha_\theta$ and $\theta_w$ are given by
\begin{equation}
\alpha_\theta = \Big( \frac{25 B_f^2}{\gamma} + {10 B_\theta B_w} + \frac{80 B_w^2 B_\theta^2}{\mu} \Big)^{-1}, \quad
\alpha_w = \frac{\mu}{40 B_f^2}. \nn
\end{equation}
This implies that $M = O(B_f^2/\mu \gamma + B_w^2 B_\theta^2 / \mu^2 + (B_f^2/\mu \gamma) n_X + n_X)$. In consequence, the total complexity is 
    \begin{align}
        O\Big(
            \big( n_X n_Y + n_X \kappa + n_X \big) \ln \frac{1}{\epsilon}
        \Big),
        \label{e:opt2:SVRPDA_I:total_complexity}
    \end{align}
where
    \begin{align}
        \kappa = B_f^2 / \gamma \mu + B_w^2B_\theta^2 / \mu^2.
    \end{align}
Noting that $\Expect\|\tiltheta_s-\theta^*\|^2 \le P_s$, the bound \eqref{e:opt2:SVRPDA_I:P_s_bound_Option2} implies that $\Expect\|\tiltheta_s-\theta^*\|^2$ also converges to zero at a linear rate of $5/8$ and the total complexity to reach $\Expect\|\tiltheta_s-\theta^*\|^2 \le \epsilon$ is also given by \eqref{e:opt2:SVRPDA_I:total_complexity}.
\section{Special case: \methodexact~with $n_{Y_i} \equiv 1$, $f_{\theta}$ linear in $\theta$ and no outer loop}
\label{appendix:proof_svrpda1_splcase}

First, we observe that in this special case, our \methodexact~algorithm will become a single-loop algorithm, and that the outer-loop in Algorithm \ref{alg:svrpda1} is no longer needed. To see this, first note that when $n_{Y_i} \equiv 1$, $\delta_k^w$ is independent of $\tiltheta$ because the last two terms in \eqref{e:svrpda1:delta_w_k} would cancel each other. Second, when $f_{\theta}$ is linear in $\theta$, the term $\barf'_{i_k} (\tiltheta)$ in \eqref{e:svrpda1:Ukupdate} and $U_0$ in \eqref{e:svrpda1:batch_quantities} are independent of $\tiltheta$, which further implies that $U_k$ (that is recursively defined in \eqref{e:svrpda1:Ukupdate}) is also independent of $\tiltheta$. Finally, we also note that, with linear $f_{\theta}$, the first two terms in \eqref{e:svrpda1:delta_theta_k} cancel with each other, so that $\delta_k^{\theta} \equiv U_k$ is independent of $\tiltheta$. As a result, the inner loop in Algorithm \ref{alg:svrpda1} does not require an outer-loop to update the reference variable $\tiltheta$. 

The following theorem establishes the complexity bound for the SVRPDA-I algorithm in this special case.
\begin{theorem}
\label{t:final_bound_SVRPDA_I_spl_case}
Suppose Assumptions \ref{a:main:g_phi:sc_sm}--\ref{a:main:L:convex} hold. Furthermore, suppose $n_{Y_i} = 1$, $1 \leq i \leq n_X$, and $f_\theta$ is a linear function of $\theta$. Consider just the the inner loop of Algorithm \ref{alg:svrpda1}, with s = 1 fixed, and
\begin{align}
\alpha_{\theta}
&=
\frac{\gamma}{{16} B_f^2 + 4 n_X \mu \gamma}, \quad \text{and} \quad
\alpha_w = \frac{1}{2 \gamma} \frac{3 n_X + \kappa + 1}{n_X + \kappa + 1}
\nn
\end{align}
where $\kappa = {B_f^2} \big / {\gamma \mu}$ is the condition number.
Then, the Lyapunov function 
\[
\Delta^{(k)} \eqdef \Big( \frac{1}{2 \alpha_\theta} \!+\! \mu \Big) \Expect \Big \{ \| \theta^{(k)} \!-\! \theta^* \|^2 \Big | \clF_k \Big \} \!+\! \Big( \frac{1}{2 \alpha_w} \!+\! \gamma \Big)  \Expect \Big \{ \| w^{(k)} \!-\! w^* \|^2 \Big | \clF_k \Big \}
\]
satisfies $\Delta^{(k)} \le \big(1 - {1}/\big({1 + 2 \kappa + 2 n_X}\big) \big)^k \Delta^{(0)}$. Furthermore, the overall computational cost (in number of oracle calls) for reaching $\Delta^{(k)} \le \epsilon$ is upper bounded by
\begin{align}
O\Big(
\big( {n_X + \kappa} \big) \ln \Big(\frac{1}{\epsilon}\Big)
\Big).
\label{e:thm:svrpda1_spl_case:total_complexity}
\end{align}
\end{theorem}
In comparison, the authors in \cite{zhaxia17sto} propose a stochastic primal dual coordinate (SPDC) algorithm for this special case and prove an overall complexity of 
$O\big(
\big( n_X + {\sqrt{n_X \kappa}} \big) \ln \big(\frac{1}{\epsilon}\big)
\big)$ to achieve an $\epsilon$-error solution, where the condition number $\kappa = B_f^2 / \mu \gamma$. This is by far the best complexity for this class of problems. It is interesting to note that the complexity result in \eqref{e:thm:svrpda1_spl_case:total_complexity} and the complexity result in \cite{zhaxia17sto} only differ in their dependency on $\kappa$. This difference is most likely due to the acceleration technique that is employed in the primal update of the SPDC algorithm. We conjecture that the dependency on the condition number of SVRPDA-I can be further improved using a similar acceleration technique.

\subsection{Proof of \Theorem{t:final_bound_SVRPDA_I_spl_case}}

It is useful to first discuss the main implications of choosing $f_\theta$ to be linear in $\theta$ and $n_{Y_i} = 1$ for all $i$. First, based on  Assumption~\ref{a:f:smooth_bg} (or equivalently \textbf{Assumption~\ref{a:main:f:smooth_bg}}), we have $B_\theta = 0$, since $f'_\theta$ is independent of $\theta$. This also implies that $L'_\theta$ is independent of $\theta$, and therefore, Assumption~\ref{a:L:convex} (or equivalently {\textbf{Assumption~\ref{a:main:L:convex}}}) holds with equality:
\begin{align}
L(\theta_1, w) - L(\theta_2, w) = \langle L'_\theta(\theta_2, w), \,\, \theta_1 - \theta_2 \rangle.
\end{align}
In particular, for any $\theta \in \Re^d$ and $w \in \Re^\ell$, $L(\theta, w) = \langle L'_\theta(\theta, w), \,\, \theta \rangle$. Finally, $n_{Y_i}$ = 1 implies 
$
\overline{1/n_{Y}} = \frac{1}{n_X} \sum_{i = 0}^{n_X - 1} 1/n_{Y_i} = 1
$.

Using the above implications in the primal bound \eqref{e:SVRPDA_I:final_bd_primal_final} (in particular, letting $B_\theta = 0$, and using linearity of $L(\theta, w)$), we obtain the primal bound for the special case as follows:
\begin{align}
& \bigg (\frac{1}{2 \alpha_\theta} + \mu \bigg ) \Expect \Big \{ \|\theta^{(k)} - \theta^*\|^2 \big | \clF_k \Big \}
+ \frac{1}{2 \alpha_\theta} \Expect \Big \{ \|\theta^{(k)} - \theta^{(k - 1)} \|^2 \big | \clF_k \Big \}
\nn\\
\leq& 
\frac{1}{2 \alpha_\theta} \| \theta^{(k - 1)} - \theta^* \|^2 
- 
\Expect 
\Big \{ 
\Big[ L(\theta^{(k)}, w^{(k)} - w^*) - L(\theta^* , w^{(k)} - w^*) \Big ] \big | \clF_k 
\Big \}
\label{e:svrpda1:case:primal}
\end{align}
Similarly, using the above implications in the dual bound \eqref{e:SVRPDA_I:dual_bd_2} (in particular, letting $( 1 - \overline{1/n_{Y}} ) = 0$), the dual bound for the special case becomes:
\begin{align}
& 
\Bigg( \frac{1}{2 \alpha_w} + \gamma \Bigg) 
\Expect \{ \|w^{(k)} - w^*\|^2 \mid \clF_k \} 
+ 
\frac{1}{2 \alpha_w} 
\Expect \{ \|w^{(k)} - w^{(k - 1)} \|^2 \mid \clF_k \}
\nn\\
\leq& 
\Bigg( \frac{1}{2 \alpha_w} + \frac{\gamma \big(n_X - 1 \big)}{n_X} \Bigg) 
\| w^{(k - 1)} - w^* \|^2 
\nn\\
& +
\Expect \Big \{ L(\theta^{(k - 1)}, w^{(k-1)} - w^*) - L(\theta^*, w^{(k-1)} - w^*) \mid \clF_k \Big \}
\nn\\
&+ 
n_X 
\Expect \Big \{ 
L(\theta^{(k - 1)}, w^{(k)} - w^{(k - 1)})-L(\theta^*, w^{(k)} - w^{(k - 1)}) \mid \clF_k 
\Big \}
\label{e:svrpda1:case:dual}
\end{align}
Adding \eqref{e:svrpda1:case:primal} and \eqref{e:svrpda1:case:dual}, we obtain the combined primal-dual bound:
\begin{align}
& \bigg (\frac{1}{2 \alpha_\theta} + \mu \bigg ) \Expect \Big \{ \|\theta^{(k)} - \theta^*\|^2 \big | \clF_k \Big \}
+ \frac{1}{2 \alpha_\theta} \Expect \Big \{ \|\theta^{(k)} - \theta^{(k - 1)} \|^2 \big | \clF_k \Big \}
\nn\\
&+  
\Bigg( \frac{1}{2 \alpha_w} + \gamma \Bigg) 
\Expect \{ \|w^{(k)} - w^*\|^2 \mid \clF_k \} 
+ 
\frac{1}{2 \alpha_w} 
\Expect \{ \|w^{(k)} - w^{(k - 1)} \|^2 \mid \clF_k \}
\nn\\
\leq
& 
\frac{1}{2 \alpha_\theta} \| \theta^{(k - 1)} - \theta^* \|^2 
+ 
\Bigg( \frac{1}{2 \alpha_w} + \frac{\gamma \big(n_X - 1 \big)}{n_X} \Bigg) 
\| w^{(k - 1)} - w^* \|^2 
\nn\\
& - 
\Expect 
\Big \{ 
\Big[ L(\theta^{(k)}, w^{(k)} - w^*) - L(\theta^* , w^{(k)} - w^*) \Big ] \big | \clF_k 
\Big \}
\nn\\
& +
\Expect \Big \{ L(\theta^{(k - 1)}, w^{(k-1)} - w^*) - L(\theta^*, w^{(k-1)} - w^*) \mid \clF_k \Big \}
\nn\\
&+ 
n_X 
\Expect \Big \{ 
L(\theta^{(k - 1)}, w^{(k)} - w^{(k - 1)})-L(\theta^*, w^{(k)} - w^{(k - 1)}) \mid \clF_k 
\Big \}
\label{e:svrpda1:case:primal-dual}
\end{align}


As done in the previous proofs, we will first consider all the $L$ terms that appear on the right hand side of \eqref{e:svrpda1:case:primal-dual}. Following exactly the same steps as \eqref{e:svrpda1:option1:Lterms1}, \eqref{e:svrpda1:L_partial_bound}, \eqref{e:svrpda1:option1:Lterms3} and \eqref{e:svrpda1:option1:Lterms4}, we obtain the final bound for the $L$ terms as given in \eqref{e:SVRPDA_I:option1:Lterms_final}, but with $B_\theta = 0$. We still write out the whole simplification details here for completeness, and also to show how this special case is much simpler than the more general case. 

Considering all the $L$ terms in \eqref{e:svrpda1:case:primal-dual}, we have:
\begin{align}
&-\Big[L(\theta^{(k)}, w^{(k)} \!-\! w^*) \!-\! L(\theta^* , w^{(k)} \!-\! w^*)\Big]
+
\Big[L(\theta^{(k \!-\! 1)}, w^{(k\!-\!1)} \!-\! w^*) \!-\! L(\theta^*, w^{(k\!-\!1)} \!-\! w^*)\Big]
\nn\\
&
+
n_X\Big[
L(\theta^{(k \!-\! 1)}, w^{(k)} \!-\! w^{(k \!-\! 1)})\!-\!L(\theta^*, w^{(k)} \!-\! w^{(k \!-\! 1)})
\Big]
\nn\\
=&
\Big[
L(\theta^{(k\!-\!1)}, w^{(k\!-\!1)} \!-\! w^*) \!-\! L(\theta^{(k)}, w^{(k\!-\!1)} \!-\! w^*)
\Big]
\nn\\
&
+
\Big[
L(\theta^{(k)}, w^{(k\!-\!1)} \!-\! w^{(k)}) \!-\! L(\theta^*, w^{(k\!-\!1)} \!-\! w^{(k)})
\Big]
\nn\\
&
+
n_X\Big[
L(\theta^{(k \!-\! 1)}, w^{(k)} \!-\! w^{(k \!-\! 1)})\!-\!L(\theta^*, w^{(k)} \!-\! w^{(k \!-\! 1)})
\Big]
\nn\\
=&
\Big[
L(\theta^{(k\!-\!1)}, w^{(k\!-\!1)} \!-\! w^*) \!-\! L(\theta^{(k)}, w^{(k\!-\!1)} \!-\! w^*)
\Big]
\nn\\
&
+
\Big[
L(\theta^{(k)}, w^{(k\!-\!1)} \!-\! w^{(k)}) 
\!-\! 
L(\theta^{(k\!-\!1)}, w^{(k\!-\!1)} \!-\! w^{(k)})
\Big]
\nn\\
&
+
(n_X\!-\!1)\Big[
L(\theta^{(k \!-\! 1)}, w^{(k)} \!-\! w^{(k \!-\! 1)})\!-\!L(\theta^*, w^{(k)} \!-\! w^{(k \!-\! 1)})
\Big]
\nn\\
=&
\Big[
L(\theta^{(k\!-\!1)}, w^{(k)} \!-\! w^*) \!-\! L(\theta^{(k)}, w^{(k)} \!-\! w^*)
\Big]
\nn\\
&
+
(n_X\!-\!1)\Big[
L(\theta^{(k \!-\! 1)}, w^{(k)} \!-\! w^{(k \!-\! 1)})\!-\!L(\theta^*, w^{(k)} \!-\! w^{(k \!-\! 1)})
\Big]
\nn\\
=&
L\Big(\theta^{(k\!-\!1)}, w^{(k)} \!-\! w^* \!+\! (n_X\!-\!1)(w^{(k)}\!-\!w^{(k\!-\!1)}) \Big)
\!-\!
L\Big(\theta^{(k)}, w^{(k)}\!-\!w^* \!+\! (n_X\!-\!1)(w^{(k)} \!-\! w^{(k\!-\!1)}) \Big)
\nn\\
\overset{(a)}{\le}&
\Big\langle
L'_{\theta}
\Big(\theta^{(k\!-\!1)}, w^{(k)} \!-\! w^* + (n_X\!-\!1)(w^{(k)}\!-\!w^{(k\!-\!1)}) \Big),
\theta^{(k\!-\!1)} \!-\! \theta^{(k)}
\Big\rangle
\nn\\
=&
\Big\langle
\frac{1}{n_X}
\sum_{i=0}^{n_X\!-\!1}
\barf'_i(\theta^{(k\!-\!1)})
\Big(w_i^{(k)} \!-\! w_i^* + (n_X\!-\!1)(w_i^{(k)}\!-\!w_i^{(k\!-\!1)}) \Big)
,
\theta^{(k\!-\!1)} \!-\! \theta^{(k)}
\Big\rangle
\nn\\
=&
\Big\langle
\frac{1}{n_X}
\sum_{i=0}^{n_X\!-\!1}
\barf'_i(\theta^{(k\!-\!1)})
\Big(w_i^{(k\!-\!1)} + n_X(w_i^{(k)}\!-\!w_i^{(k\!-\!1)}) \!-\! w_i^* \Big)
,
\theta^{(k\!-\!1)} \!-\! \theta^{(k)}
\Big\rangle
\nn\\
\overset{(b)}{\le}&
\beta_1
\Big\|
\frac{1}{n_X}
\sum_{i=0}^{n_X\!-\!1}
\barf'_i(\theta^{(k\!-\!1)})
\Big(w_i^{(k\!-\!1)} + n_X(w_i^{(k)}\!-\!w_i^{(k\!-\!1)}) \!-\! w_i^* \Big)
\Big\|^2
+
\frac{1}{\beta_1}
\big\|
\theta^{(k\!-\!1)} \!-\! \theta^{(k)}
\big\|^2
\nn\\
=&
\beta_1
\big\|
\frac{1}{n_X}
\sum_{i=0}^{n_X\!-\!1}
\barf'_i(\theta^{(k\!-\!1)})
\big(w_i^{(k\!-\!1)} \!-\! w_i^*\big) 
\!+\! 
\sum_{i=0}^{n_X\!-\!1}
\barf'_i(\theta^{(k\!-\!1)})
(w_i^{(k)}\!-\!w_i^{(k\!-\!1)})
\big\|^2
\!+\!
\frac{1}{\beta_1}
\big\|
\theta^{(k\!-\!1)} \!-\! \theta^{(k)}
\big\|^2
\nn\\
\overset{(c)}{\le}&
2\beta_1
\big\|
\frac{1}{n_X} \!
\sum_{i=0}^{n_X\!-\!1} \!
\barf'_i(\theta^{(k\!-\!1)})
\Big(w_i^{(k\!-\!1)} \!-\! w_i^*\Big) 
\big\|^2
+ 
2\beta_1
\big\|
\sum_{i=0}^{n_X\!-\!1}
\barf'_i(\theta^{(k\!-\!1)})
(w_i^{(k)}\!-\!w_i^{(k\!-\!1)})
\big\|^2
\nn\\
&
+
\frac{1}{\beta_1}
\big\|
\theta^{(k\!-\!1)} \!-\! \theta^{(k)}
\big\|^2
\nn\\
\hspace{-0.24in}\overset{(d)}{\le}&
\frac{2\beta_1 B_f^2}{n_X}
\!
\sum_{i=0}^{n_X\!-\!1} \!
\| w_i^{(k\!-\!1)} \!-\! w_i^* \|^2
\!+\! 
2\beta_1
\big\|  \!
\sum_{i=0}^{n_X\!-\!1}  \!
\barf'_i(\theta^{(k\!-\!1)})
(w_i^{(k)}\!-\!w_i^{(k\!-\!1)})
\big\|^2
\!+\!
\frac{1}{\beta_1}
\big\|
\theta^{(k\!-\!1)} \!-\! \theta^{(k)}
\big\|^2
\label{e:svrpda1:spl_case:Lsimplificaiton}
\end{align}
where step (a) uses convexity of $L$ in $\theta$, step (b) uses $\langle a, b\rangle \le \beta_1 \|a\|^2 + \frac{1}{\beta_1}\|b\|^2$, step (c) uses $\|a+b\|^2 \le 2\|a\|^2 + 2\|b\|^2$, and step (d) applies Jensen's inequality and bounded gradient assumption to the first term.  The second term in \eqref{e:svrpda1:spl_case:Lsimplificaiton} can be simplified as follows:
\begin{align}
\Expect\Big\{
\Big\|
\sum_{i=0}^{n_X-1}
\barf'_i(\theta^{(k-1)})
(w_i^{(k)}-w_i^{(k-1)})
\Big\|^2
\Big| \clF_k
\Big\}
\!=\!&
\sum_{i=0}^{n_X-1} \!
\sum_{j=0}^{n_{Y_i}-1} \!
\frac{1}{n_X n_{Y_i}}
\Big\|
\barf'_i(\theta^{(k-1)})
(w_{ij}'-w_i^{(k-1)})
\Big\|^2    
\nn\\
\le&
B_f^2
\sum_{i=0}^{n_X-1}
\sum_{j=0}^{n_{Y_i}-1}
\frac{1}{n_X n_{Y_i}}
\Big\|
w_{ij}'-w_i^{(k-1)}
\Big\|^2    
\nn\\
=&
B_f^2
\sum_{i=0}^{n_X-1}
\Expect\{ \|w_i^{(k)} - w_i^{(k-1)}\|^2 \}
\label{e:svrpda1:spl_case:Lsimplificaiton_2}
\end{align}
Therefore, using \eqref{e:svrpda1:spl_case:Lsimplificaiton} and \eqref{e:svrpda1:spl_case:Lsimplificaiton_2}, the conditional expectation of all the $L$ terms in \eqref{e:svrpda1:case:primal-dual} are bounded by:
\begin{align}
&\frac{2\beta_1 B_f^2}{n_X}
\!\sum_{i=0}^{n_X\!-\!1}\!
\| w_i^{(k\!-\!1)} \!-\! w_i^* \|^2
\!+ \!
2\beta_1 B_f^2
\!\sum_{i=0}^{n_X\!-\!1}\!
\Expect\{ \|w_i^{(k)} \!-\! w_i^{(k\!-\!1)}\|^2 | \clF_k \}
\!+ \!
\frac{1}{\beta_1}
\Expect\big\{
\big\|
\theta^{(k\!-\!1)} \!-\! \theta^{(k)}
\big\|^2
\big| \clF_k
\big\}
\nn\\
=&
\frac{2\beta_1 B_f^2}{n_X}
\| w^{(k\!-\!1)} \!-\! w^* \|^2
+ 
2\beta_1 B_f^2
\Expect\{ \|w^{(k)} \!-\! w^{(k\!-\!1)}\|^2 | \clF_k \}
+
\frac{1}{\beta_1}
\Expect\big\{
\big\|
\theta^{(k\!-\!1)} \!-\! \theta^{(k)}
\big\|^2
\big| \clF_k
\big\}          
\label{e:SVRPDA_I:spl_case:Lterms_final}
\end{align}

Substituting the above upper bound for the $L$ terms in \eqref{e:svrpda1:case:primal-dual} leads to
\begin{align}
&
\bigg (\frac{1}{2 \alpha_\theta} + \mu \bigg ) \Expect \Big \{ \|\theta^{(k)} - \theta^*\|^2 \big | \clF_k \Big \}
+
\Bigg( \frac{1}{2 \alpha_w} + \gamma \Bigg) 
\Expect \{ \|w^{(k)} - w^*\|^2 \mid \clF_k \} 
\nn\\
&
+ 
\Big( \frac{1}{2 \alpha_\theta} - \frac{1}{\beta_1} \Big)
\Expect \Big \{ \|\theta^{(k)} - \theta^{(k - 1)} \|^2 \big | \clF_k \Big \}
+ 
\Big( \frac{1}{2 \alpha_w} - 2\beta_1 B_f^2 \Big)
\Expect \{ \|w^{(k)} - w^{(k - 1)} \|^2 \mid \clF_k \}
\nn\\
\leq& 
\frac{1}{2 \alpha_\theta} \| \theta^{(k - 1)} - \theta^* \|^2 
+
\Bigg( 
\frac{1}{2 \alpha_w} 
+ 
\gamma - \frac{\gamma}{n_X} 
+ 
\frac{2\beta_1 B_f^2}{n_X}
\Bigg) 
\| w^{(k - 1)} - w^* \|^2 
\label{e:SVRPDA_I:total_bd:linear}
\end{align}
which will be the final bound we will analyze.

\subsubsection*{Substituting hyper-parameter choices}

Recall the choice of step-sizes $\alpha_\theta$ and $\alpha_w$ defined in Theorem~\ref{t:final_bound_SVRPDA_I_spl_case}:
\begin{equation}
\begin{aligned}
\alpha_\theta = \frac{\gamma}{{16} B_f^2 + 4 n_X \mu \gamma}
\\
\alpha_w = \frac{1}{2 \gamma} \frac{3 n_X + \kappa + 1}{n_X + \kappa + 1}
\end{aligned}
\label{e:alpha_theta_alpha_w_new}
\end{equation}
where the condition number $\kappa$ is also defined in Theorem~\ref{t:final_bound_SVRPDA_I_spl_case}:
\[
\kappa = {B_f^2}/{\gamma \mu}
\]
Furthermore, choosing $\beta_1$ such that
\begin{equation}
\frac{1}{2 \alpha_\theta} = \frac{1}{\beta_1},
\label{e:SVRPDA_I:linear:alpha_theta_beta_1}
\end{equation}
we have:
\begin{equation}
\begin{aligned}
\alpha_\theta \alpha_w 
& = \frac{\gamma}{{16} B_f^2 + 4 n_X \mu \gamma} \times \frac{1}{2 \gamma} \frac{3 n_X + \kappa + 1}{n_X + \kappa + 1}
\\
& = \frac{1}{32 B_f^2 + 8 n_X \mu \gamma} \times \frac{3 n_X + \kappa + 1}{n_X + \kappa + 1}
\\
& < \frac{1}{32 B_f^2 + 8 n_X \mu \gamma} \times \frac{3 n_X + 3 \kappa + 3}{n_X + \kappa + 1}
\\
& = \frac{3}{32 B_f^2 + 8 n_X \mu \gamma}
\\
& < \frac{1}{8 B_f^2 + 2 n_X \mu \gamma}
\\
& < \frac{1}{8 B_f^2}
\label{e:SVRPDA_I:alpha_theta_alpha_w_bd}
\end{aligned}
\end{equation}
so that,
\begin{equation}
\begin{aligned}
\frac{1}{2 \alpha_w} & > 4 \alpha_\theta B_f^2
\\
\implies \frac{1}{2 \alpha_w} & > 2 \beta_1 B_f^2
\label{e:SVRPDA_I:linear:alpha_w_beta_1}
\end{aligned}
\end{equation}
Identities \eqref{e:SVRPDA_I:linear:alpha_theta_beta_1} and \eqref{e:SVRPDA_I:linear:alpha_w_beta_1}
make sure that the third and the fourth terms on the left hand side of \eqref{e:SVRPDA_I:total_bd:linear} are either zero or positive, so that they can be ignored, resulting in:
\begin{align}
&
\bigg (\frac{1}{2 \alpha_\theta} + \mu \bigg ) \Expect \Big \{ \|\theta^{(k)} - \theta^*\|^2 \big | \clF_k \Big \}
+
\Bigg( \frac{1}{2 \alpha_w} + \gamma \Bigg) 
\Expect \{ \|w^{(k)} - w^*\|^2 \mid \clF_k \} 
\nn\\
\leq& 
\frac{1}{2 \alpha_\theta} \| \theta^{(k - 1)} - \theta^* \|^2 
+
\Bigg( 
\frac{1}{2 \alpha_w} 
+ 
\gamma - \frac{\gamma}{n_X} 
+ 
\frac{2\beta_1 B_f^2}{n_X}
\Bigg) 
\| w^{(k - 1)} - w^* \|^2 
\label{e:SVRPDA_I:total_bd:linear_2}
\end{align}

We further look at the coefficients of other error terms in \eqref{e:SVRPDA_I:total_bd:linear_2}. To this end, we have:
\begin{equation}
\begin{aligned}
\frac{1}{2 \alpha_\theta \mu} 
& = \frac{16 B_f^2 + 4 n_X \mu \gamma}{2 \mu \gamma}
\\
& = \frac{8 B_f^2}{\mu \gamma} + 2 n_X
\\
& = 2 \kappa + 2 n_X
\end{aligned}
\label{e:1by2alpha_theta_mu}
\end{equation}
and therefore, letting $r_P$ to denote the ratio of the coefficients of $ \|\theta^{(k-1)} - \theta^*\|^2 $ and \\ $\Expect \big \{ \|\theta^{(k)} - \theta^*\|^2 \big | \clF_k \big \}$, we have:
\begin{equation}
\begin{aligned}
r_P & \eqdef \Big( \frac{1}{2\alpha_\theta} \Big) \Big/ \Big( \frac{1}{2\alpha_\theta} + \mu \Big)
\\
& = 1 - \frac{2\alpha_\theta \mu}{1 + 2\alpha_\theta \mu}, \quad
\\
& = 1 - \frac{1}{1 + \frac{1}{2 \alpha_\theta \mu}}
\\
& = 1 - \frac{1}{1 + 2 \kappa + 2 n_X}
\end{aligned}
\label{e:rp_new}
\end{equation}
Also, for the dual terms we have $r_D$ denoting the ratio of the coefficients of  $ \|w^{(k - 1)} - w^*\|^2 $ and $\Expect \{ \|w^{(k)} - w^*\|^2 \mid \clF_k \} $, so that:
\begin{equation}
\begin{aligned}
r_D 
& \eqdef \Bigg( 
\frac{1}{2 \alpha_w} 
+ 
\gamma - \frac{\gamma}{n_X} 
+ 
\frac{2\beta_1 B_f^2}{n_X}
\Bigg) \big / \Bigg( 
\frac{1}{2 \alpha_w} 
+ 
\gamma
\Bigg)
\\
& = 1 - \frac{1}{n_X} \frac{2\alpha_w\gamma - 8\alpha_w\alpha_\theta B_f^2}{1+2\alpha_w\gamma}
\\
& \overset{(a)}{\leq} 1 - \frac{1}{n_X}\frac{2\alpha_w \gamma -1}{2\alpha_w \gamma + 1}
\\
& \overset{(b)}{=} 1 - \frac{1}{n_X}\frac{2 n_X}{4 n_X + 2 \kappa + 2}
\\
& = 1 - \frac{1}{2 n_X + \kappa + 1}
\end{aligned}
\label{e:rd_new}
\end{equation}
where step (a) follows from \eqref{e:SVRPDA_I:alpha_theta_alpha_w_bd}, and step (b) follows from \eqref{e:SVRPDA_I:linear:2alpha_w_gamma} below:
\begin{equation}
\begin{aligned}
\frac{2 \alpha_w \gamma - 1}{2 \alpha_w \gamma + 1}
& = \frac{\frac{3 n_X + \kappa + 1}{n_X + \kappa + 1} - 1}{\frac{3 n_X + \kappa + 1}{n_X + \kappa + 1} + 1}
\\
& = \frac{3 n_X + \kappa + 1 - n_X - \kappa - 1}{3 n_X + \kappa + 1 + n_X + \kappa + 1}
\\
& = \frac{2 n_X}{4 n_X + 2 \kappa + 2}
\end{aligned}
\label{e:SVRPDA_I:linear:2alpha_w_gamma}
\end{equation}
From \eqref{e:rp_new}, the above bound \eqref{e:rd_new} on $r_D$ implies $r_D \leq r_P$.

Recalling the definition of $\Delta^{(k)}$, the Lyapunov function defined in Theorem~\ref{t:final_bound_SVRPDA_I_spl_case}; Substituting for the step-size values in \eqref{e:alpha_theta_alpha_w_new}, we have:
\begin{equation*}
\begin{aligned}
\Delta^{(k)} & = \Big( \frac{1}{2 \alpha_\theta} \!+\! \mu \Big) \Expect \Big \{ \| \theta^{(k)} \!-\! \theta^* \|^2 \Big | \clF_k \Big \} \!+\! \Big( \frac{1}{2 \alpha_w} \!+\! \gamma \Big)  \Expect \Big \{ \| w^{(k)} \!-\! w^* \|^2 \Big | \clF_k \Big \}
\\
& = \Big( \mu \big (2 \kappa \!+\! 2 n_X \!+\! 1 \big) \Big) \Expect \Big \{ \| \theta^{(k)} \!-\! \theta^* \|^2 \Big | \clF_k \Big \} \!+\! \Big( {\gamma} \Big( \frac{n_X \!+\! \kappa \!+\! 1}{3 n_X \!+\! \kappa \!+\! 1} \!+\! 1 \Big) \Big)  \Expect \Big \{ \| w^{(k)} \!-\! w^* \|^2 \Big | \clF_k \Big \}
\\
& = \Big( \mu \big (2 \kappa \!+\! 2 n_X \!+\! 1 \big) \Big) \Expect \Big \{ \| \theta^{(k)} \!-\! \theta^* \|^2 \Big | \clF_k \Big \} \!+\! \Big( {\gamma} \Big( \frac{4 n_X \!+\! 2 \kappa \!+\! 2}{3 n_X \!+\! \kappa \!+\! 1}\Big) \Big)  \Expect \Big \{ \| w^{(k)} \!-\! w^* \|^2 \Big | \clF_k \Big \}
\end{aligned}
\end{equation*}
Based on \eqref{e:rp_new} and \eqref{e:rd_new}, and the fact that $r_D \leq r_P$,
the inequality \eqref{e:SVRPDA_I:total_bd:linear} then implies:
\begin{equation}
\Delta^{(k)} \leq r \Delta^{(k-1)}
\label{e:SVRPDA_I:lyap_rate}
\end{equation}
where, 
\begin{equation}
r = r_P = 1 - {1}/\big({1 + 2 \kappa + 2 n_X}\big).
\end{equation}
 
The bound \eqref{e:SVRPDA_I:lyap_rate} implies that after $k$ iterations, the error $\Delta^{(k)}$ satisfies:
\[
\Delta^{(k)} \leq r^k \Delta^{(0)}
\]
Therefore, for $\Delta^{(k)} < \epsilon$, it suffices to have $r^k \Delta^{(0)} < \epsilon$, so that the number of iterations $k$ satisfies:
\begin{equation}
\begin{aligned}
& k \ln r \leq \ln \Big ( \frac{\epsilon}{\Delta^{(0)}}  \Big)
\\
\implies & k \ln \Big( 1 - {1}/\big({1 + 2 \kappa + 2 n_X}\big) \Big) \leq \ln \Big ( \frac{\epsilon}{\Delta^{(0)}}  \Big)
\\
\implies &  {- k }/\big({1 + 2 \kappa + 2 n_X}\big) \leq \ln \Big ( \frac{\epsilon}{\Delta^{(0)}}  \Big)
\\
\implies & k  \geq \Big( {1 + 2 \kappa + 2 n_X} \Big)  \ln \Big ( \frac{\Delta^{(0)}}{\epsilon} \Big)
\end{aligned}
\label{e:svrpda1_spl_case:final_complexity}
\end{equation}
where we have used $-\ln(1-x) \geq x$. \eqref{e:svrpda1_spl_case:final_complexity} implies the final complexity result of Theorem~\ref{t:final_bound_SVRPDA_I_spl_case}.
\qed

\end{document}